\documentclass[10pt]{amsart}

\usepackage[OT1]{fontenc}
\usepackage{amsmath}
\usepackage{amssymb}
\usepackage{amsthm}

\usepackage{xcolor}
\usepackage[hang,tight,raggedright]{subfigure}
\usepackage[small,it]{caption}
\usepackage{tabularx,booktabs}
\usepackage[mathscr]{eucal}
\usepackage{units}
\usepackage{url}
\usepackage{aliascnt}
\usepackage{calc}
\usepackage{cite}
\usepackage[inline]{enumitem}
\setlist{
topsep=.3ex,
itemsep=.05ex,
parsep=.1ex,
partopsep=.1ex,
rightmargin=0pt
} 
\setlist[itemize]{leftmargin=4ex}

\setlist[enumerate]{labelsep=*, leftmargin=1.5pc}
\setlist[enumerate,1]{label=(\arabic*), ref=(\arabic*)}
\setlist[enumerate,2]{label=(\emph{\alph*}),
                      ref=(\theenumi)\emph{\alph*}}
\setlist[enumerate,3]{label=\roman*), ref=\theenumii.\roman*}

\usepackage[totalwidth=.65\paperwidth,
                totalheight=.74\paperheight,
                a4paper,
                headsep=.8cm,
                hmarginratio=1:2,
                vmarginratio=2:2,
                marginparsep=.03\paperwidth,
                marginparwidth=.2\paperwidth,
                bindingoffset=15mm]          {geometry}
\usepackage[colorlinks,linkcolor=darkblue,
            citecolor=darkblue,
            filecolor=darkblue,urlcolor=darkgreen,
            anchorcolor=darkblue,menucolor=darkblue,
            bookmarksopen=false,bookmarks=false,pdfpagelabels,plainpages=false
            ]{hyperref}

\definecolor{darkblue}{rgb}{.8,.15,.15}
\definecolor{darkgreen}{rgb}{0.15,.4,.5}

\usepackage{tikz}
\usetikzlibrary{calc}
\usetikzlibrary{decorations.pathreplacing}
\usetikzlibrary{arrows,shapes}

\tikzstyle{perspective adjusted}=[%
x={(.8,.3)},%
y={(.2,1.5)},%
z={(2.4,-.5)}]

\newcommand{\theoremname}{Theorem}
\newcommand{\corollaryname}{Corollary}
\newcommand{\lemmaname}{Lemma}
\newcommand{\propositionname}{Proposition}
\newcommand{\remarkname}{Remark}
\newcommand{\examplename}{Example}
\newcommand{\definitionname}{Definition}

\newcommand{\figuresref}[1]{\hyperref[#1]{Figures~\ref*{#1}}}
\newcommand{\propositionsref}[1]{\hyperref[#1]{Propositions~\ref*{#1}}}

\theoremstyle{plain}
\newtheorem{theorem}{\theoremname}[section]
\newtheorem*{theorem*}{\theoremname}

\newaliascnt{corollary}{theorem}
\newtheorem{corollary}[corollary]{\corollaryname}

\newaliascnt{lemma}{theorem}
\newtheorem{lemma}[lemma]{\lemmaname}

\newaliascnt{proposition}{theorem}
\newtheorem{proposition}[proposition]{\propositionname}

\theoremstyle{definition}
\newaliascnt{remark}{theorem}

\newaliascnt{example}{theorem}

\newaliascnt{definition}{theorem}
\newtheorem{definition}[definition]{\definitionname}

\aliascntresetthe{corollary}
\aliascntresetthe{lemma}
\aliascntresetthe{proposition}
\aliascntresetthe{remark}
\aliascntresetthe{definition}

\numberwithin{figure}{section}
\numberwithin{table}{section}

\DeclareMathOperator{\bdim}{bd}

\DeclareMathOperator{\cayley}{Cay}
\DeclareMathOperator{\neigh}{\mathscr N}
\DeclareMathOperator{\cneigh}{\mathscr{CN}}
\DeclareMathOperator{\partners}{\mathcal P}
\DeclareMathOperator{\face}{\mathsf F}

\DeclareMathOperator{\lattice}{\mathcal L}
\DeclareMathOperator{\joinedprod}{{\mathcal{JP}}}
\DeclareMathOperator{\circular}{{\mathcal{C}}}
\newcommand{\vertices}{\mathbf V}
\newcommand{\JPG}{\mathcal{JPG}}

\newcommand{\R}{\mathbb R}
\newcommand{\FG}{\mathsf{FG}}
\newcommand{\res}{\mathsf{res}}
\newcommand{\red}{\mathsf{red}}
\newcommand{\gr}{\Gamma}
\newcommand{\birkhoff}[1][n]{\mathcal B_{#1}}
\newcommand{\ov}[1]{\overline{#1}}

\newcommand{\zv}[1]{\mathbf 0^{(#1)}}
\newcommand{\conv}{\mathsf{conv}}
\renewcommand{\wedge}{\mathsf{wedge}}

\title{Faces of Birkhoff Polytopes}

\author[Andreas Paffenholz]{Andreas Paffenholz}
\thanks{The author is supported by the Priority Program 1489 of the German Research Council (DFG)}
\address{TU Darmstadt, Fachbereich Mathematik, Dolivostr. 15, 64293 Darmstadt, Germany}
\email{paffenholz@mathematik.tu-darmstadt.de}

\date{\today}

\begin{document}

\numberwithin{equation}{section}
\begin{abstract}
  The Birkhoff polytope $\birkhoff$ is the convex hull of all $(n\times n)$ permutation matrices,
  \emph{i.e.}, matrices where precisely one entry in each row and column is one, and zeros at all
  other places. This is a widely studied polytope with various applications throughout mathematics.

  In this paper we study combinatorial types $\lattice$ of faces of a Birkhoff polytope. The
  \emph{Birkhoff dimension} $\bdim(\lattice)$ of $\lattice$ is the smallest $n$ such that
  $\birkhoff$ has a face with combinatorial type $\lattice$.

  By a result of Billera and Sarangarajan, a combinatorial type $\lattice$ of a $d$-dimensional face
  appears in some $\birkhoff[k]$ for $k\le 2d$, so $\bdim(\lattice)\le 2d$. We will characterize
  those types with $\bdim(\lattice)\ge 2d-3$, and we prove that any type with $\bdim(\lattice)\ge d$
  is either a product or a wedge over some lower dimensional face. Further, we computationally
  classify all $d$-dimensional combinatorial types for $2\le d\le 8$.
\end{abstract}

\maketitle

\vspace*{1cm}

\section{Introduction}

The Birkhoff polytope $\birkhoff$ is the convex hull of all $(n\times n)$ permutation matrices,
\emph{i.e.}, matrices that have precisely one $1$ in each row and column, and zeros at all other
places.  Equally, $\birkhoff$ is the set of all doubly stochastic $(n\times n)$-matrices,
\emph{i.e.}, non-negative matrices whose rows and columns all sum to $1$, or the perfect matching
polytope of the complete bipartite graph $K_{n,n}$. The Birkhoff polytope $\birkhoff$ has dimension
$(n-1)^2$ with $n!$ vertices and $n^2$ facets. The Birkhoff-von Neumann Theorem shows that
$\birkhoff$ can be realized as the intersection of the positive orthant with a family of
hyperplanes.

Birkhoff polytopes are a widely studied class of
polytopes~\cite{BS96,BG77-1,BG77-2,BG77-3,BG76,1065.52007, CR99, MR960139, MR2575172,Loera2009} with
many applications in different areas of mathematics, \emph{e.g.}, enumerative combinatorics
\cite{Stanley1,1077.52011}, optimization \cite{Tinhofer,MR953322,MR1961267}, statistics
\cite{Pak00,MR1380519}, or representation theory \cite{Onn93,1055.51003}.  Yet, despite all these
efforts, quite fundamental questions about the combinatorial and geometric structure of this
polytope, and its algorithmic treatment, are still open. In particular, we know little about numbers
of faces apart from those of facets, vertices, and edges.

In this paper, we study combinatorial types of faces of Birkhoff polytope. The \emph{combinatorial
  type} of a face $F$ of some Birkhoff polytope is given by its face lattice $\lattice$. For such a
combinatorial type we can define the \emph{Birkhoff dimension} $\bdim(\lattice)$ of $\lattice$ as
the minimal $n$ such that $\birkhoff$ has a face combinatorially equivalent to $\lattice$.  

By a result of Billera and Sarangarajan~\cite{BS96} any combinatorial type of a $d$-dimensional face
of $\birkhoff$ already appears in $\birkhoff[2d]$, so $\bdim(\lattice)\le 2d$.  Here, we
characterize combinatorial types of $d$-dimensional faces with $\bdim(\lattice)\ge 2d-3$. More
precisely, we show in \autoref{thm:CompleteClassification} that the only combinatorial type
$\lattice$ with $\bdim(\lattice)=2d$ is the $d$-cube (\autoref{prop:cube}). If
$\bdim(\lattice)=2d-1$, then $\lattice$ must by a product of a cube and a triangle
(\autoref{prop:2dm1}). $\bdim(\lattice)=2d-2$ allows three new types, a pyramid over a cube, the
product of a cube with a pyramid over a cube, and the product of two triangles with a cube
(\autoref{prod:2dm2}). Finally, faces with $\bdim(\lattice)=2d-3$ are either products or certain
Cayley polytopes of products of lower dimensional faces, the joined products and reduced joined
products defined in \autoref{subsec:2d-3} (\autoref{thm:ClassWedges}).

More generally, we show in \autoref{subsec:wedges} that any combinatorial type $\lattice$ of a
$d$-dimensional face with $\bdim(\lattice)\ge d$ is either a product of two lower dimensional faces,
or a wedge of a lower dimensional face over one of its faces (\autoref{thm:Wedges}). We further
characterize combinatorial types of faces $F$ of some Birkhoff polytope for which the pyramid over
$F$ is again a face of some Birkhoff polytope.

Finally, we enumerate all combinatorial types of $d$-dimensional faces of some $\birkhoff$ for $2\le
d\le 8$. This is done with an algorithm that classifies face graphs corresponding to combinatorially
different faces of $\birkhoff$. The algorithm has been implemented as an extension to the software
system \texttt{polymake}~\cite{polymake_birkhoff,GJ00} for polyhedral geometry
(\autoref{sec:append-low-dimens}). The computed data in \texttt{polymake} format can be found
at~\cite{BirkhoffFaces}.

Following work of Billera and Sarangarajan~\cite{BS96} we use elementary bipartite graphs
(\emph{face graphs}) to represent combinatorial types of faces. A graph is \emph{elementary} if
every edge is contained in some perfect matching in the graph. A perfect matching in a bipartite
graph with $n$ nodes in each layer naturally defines an $(n\times n)$-matrix with entries in
$\{0,1\}$, which gives the correspondence to a face of $\birkhoff$. The correspondence of faces and
graphs is explained in~\autoref{fb:subsec:birkhoff}. We use the language of face graphs in
\autoref{sec:irred} and~\autoref{sec:combtypes} to study the structure of these graphs and the
corresponding faces.

Previously, Brualdi and Gibson have done an extensive study of faces of Birkhoff polytopes in a
series of papers~\cite{BG77-3,BG77-1,BG77-2,BG76}. They used 0/1-matrices to represent types of
faces, which naturally correspond to elementary bipartite graphs by placing edges at all non-zero
entries. They studied combinatorial types of faces with few vertices, the diameter of $\birkhoff$,
and some constructions for new faces from given ones. We review some of their results in
\autoref{sec:combtypes}, as we need them for our constructions in \autoref{sec:manynodes}. 

A fair amount of work also has gone into the computation of the Ehrhart polynomial or the volume of the
Birkhoff polytope. Until recently, only low dimensional cases were known~\cite{CR99,1065.52007}
using a computational approach. In 2009, Canfield and McKay~\cite{MR2575172} obtained an asymptotic
formula for the volume, and in the same year De Loera et al.~\cite{Loera2009} gave an exact formula
by computing the Ehrhart polynomial.

Birkhoff polytopes are a special case of the much more general concept of a permutation
polytope. These are polytopes obtained as the convex hull of all permutation matrices corresponding
to some subgroup $G$ of the the full permutation group $S_n$. So the Birkhoff polytope is the
permutation polytope of $S_n$. Permutation polytopes have been introduced by Guralnick and Perkinson
in~\cite{GP06}. They studied these objects from a group theoretic view point and provided formulas
for the dimension and the diameter. A systematic study of combinatorial properties of general
permutation polytopes and a computational classification of $d$-dimensional permutation polytopes
and $d$-dimensional faces of some higher dimensional permutation polytope for $d\le 4$ can be found
in~\cite{BHNP07-1}.

Several subpolytopes of the Birkhoff polytope have been shown to have an interesting structure and
some beautiful properties. Here, in particular the polytope of \emph{even} permutation matrices
attracted much attention~\cite{CW04,HP04,MR1111560}, but also many other classes of groups have been
considered~\cite{BHNP07-2,1212.4442,MR1749805,CRY00}.

\textbf{Acknowledgments.} This work has benefited from discussions with various people, in
particular Christian Haase and Benjamin Nill. The work on this paper has been supported by a postdoc
position in the Emmy-Noether project HA 4383/1 and the Priority Program SPP 1489 of the German
Research Society (DFG).  The computation of low dimensional faces has been done with an extension to
the software system \texttt{polymake}~\cite{GJ00}.

\section{Background and Basic Definitions}
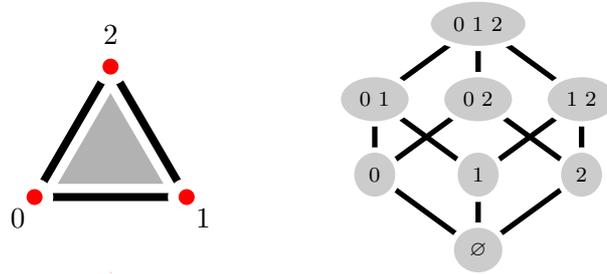
\begin{figure}[tb]
  \centering
  \begin{tikzpicture}[scale=.5]
    \draw [fill=black!30] (0,0) -- (4,0) -- (2,3.46) -- cycle;

    \draw [line width=10pt,color=white] (0,0) -- (4,0);
    \draw [line width=10pt,color=white] (0,0) -- (2,3.46);
    \draw [line width=10pt,color=white] (4,0) -- (2,3.46);

    \draw [line width=3pt] (0,0) -- (4,0);
    \draw [line width=3pt] (0,0) -- (2,3.46);
    \draw [line width=3pt] (4,0) -- (2,3.46);

    \draw [color=white,fill=red,line width=4pt] (0,0) circle (10pt);
    \draw [color=white,fill=red,line width=4pt] (4,0) circle (10pt);
    \draw [color=white,fill=red,line width=4pt] (2,3.46) circle (10pt);

    \draw [color=white,fill=red,line width=4pt] (2,-2) circle (0pt);

    \node at (0,0) [left,yshift=-8pt] {\large$0$};
    \node at (4,0) [right,yshift=-8pt] {\large$1$};
    \node at (2,3.46) [above,yshift=5pt] {\large$2$};
  \end{tikzpicture}
  \qquad\qquad
  \begin{tikzpicture}[yscale=.5,xscale=.68]
    \tikzstyle{every node}=[color=white, text=black, ellipse, draw, line width=2pt, fill=black!20,
    inner sep=4pt, minimum width=2pt]
    \draw [line width=2pt] (2,0) -- (0,2);
    \draw [line width=2pt] (2,0) -- (2,2);
    \draw [line width=2pt] (2,0) -- (4,2);

    \draw [line width=2pt] (0,4) -- (0,2);
    \draw [line width=2pt] (2,4) -- (0,2);
    \draw [line width=2pt] (4,4) -- (2,2);
    \draw [line width=2pt] (2,4) -- (4,2);
    \draw [line width=2pt] (0,4) -- (2,2);
    \draw [line width=2pt] (4,4) -- (4,2);

    \draw [line width=2pt] (2,6) -- (0,4);
    \draw [line width=2pt] (2,6) -- (2,4);
    \draw [line width=2pt] (2,6) -- (4,4);

    \node at (0,4) {\footnotesize$0\ 1$};
    \node at (2,4) {\footnotesize$0\ 2$};
    \node at (4,4) {\footnotesize$1\ 2$};

    \node at (0,2) {\footnotesize$0$};
    \node at (2,2) {\footnotesize$1$};
    \node at (4,2) {\footnotesize$2$};

    \node at (2,0) {\footnotesize$\varnothing$};
    \node at (2,6) {\footnotesize$0\ 1\ 2$};
  \end{tikzpicture}
  \caption{A triangle and its face lattice.}
  \label{fig:facelattice}
\end{figure}

\subsection{Polytopes}

A \emph{polytope} $P\subseteq \R^d$ is the convex hull $P=\conv(v_1, \ldots, v_k)$ of a finite set
of points $v_1, \ldots, v_k\in\R^d$. Dually, any polytope can be written as the bounded intersection
of a finite number of affine half-spaces in the form $P:=\{x\mid Ax\le b\}$. We repeat some notions
relevant for polytopes. For a thorough discussion and proofs we refer to~\cite{Ziegler}.

A (proper) \emph{face} $F$ of a polytope $P$ is the intersection of $P$ with an affine hyperplane
$H$ such that $P$ is completely contained in one of the closed half-spaces defined by $H$. (The
intersection may be empty.) We also call the empty set and the polytope $P$ a face of $P$. Any face
$F$ is itself a polytope. The dimension of a polytope $P\subseteq \R^d$ is the dimension of the
minimal affine space containing it. It is \emph{full-dimensional} if its dimension is $d$.

$0$-dimensional faces of $P$ are called \emph{vertices}, $1$-dimensional faces are
\emph{edges}. Proper faces of maximal dimension are called \emph{facets}. $P$ is the convex hull of
its vertices, and the vertices of any face are a subset of the vertices of $P$. Thus, a polytope has
only a finite number of faces. Let $f_i$ be the number of $i$-dimensional faces of $P$, $0\le i\le
\dim P-1$. The $f$-vector of a $d$-dimensional polytope $P$ is the non-negative integral vector
$f(P):=(f_0, \ldots, f_{d-1})$.

Inclusion of sets defines a partial order on the faces of a polytope. The \emph{face lattice} or
\emph{combinatorial type} $\lattice(P)$ of a polytope $P$ is the partially ordered set of all faces
of $P$ (including the empty face and $P$ itself). This defines a Eulerian lattice. See
\autoref{fig:facelattice} for an example. It contains all combinatorial information of the
polytope. Two polytopes $P, P'$ are \emph{combinatorially isomorphic} or \emph{have the same
  combinatorial type} if their face lattices are isomorphic as posets.  For any given Eulerian
lattice $\mathcal L$ we call a subset $P\subset\R^d$ a \emph{geometric realization} of $\mathcal L$
if $P$ is a polytope with a face lattice isomorphic to $\mathcal L$.  Note, that not all Eulerian
lattices are a face lattice of a polytope.

An \emph{$r$-dimensional simplex} (or $r$-simplex) is the convex hull of $r+1$ affinely independent
points in $\R^d$.  A polytope is called \emph{simplicial} if all facets are simplices. It is
\emph{simple} if the dual is simplicial. Equally, a $d$-dimensional polytope $P$ is simple if each
vertex is incident to precisely $d$ edges. The \emph{$d$-dimensional $0/1$-cube $C^d$} is the convex
hull of all $d$-dimensional $0/1$-vectors.  This is a simple $d$-polytope with $2^d$ vertices and
$2d$ facets.  More generally, we denote by a \emph{$d$-cube} any $d$-dimensional polytope that is
combinatorially isomorphic to the $0/1$-cube (it need not be full dimensional).

Let $P_1\subset\R^{d_1}$ and $P_2\subset\R^{d_2}$ be two (geometrically realized) polytopes with
vertex sets $\vertices(P_1)=\{v_1, \ldots, v_k\}$ and $\vertices(P_2)=\{w_1, \ldots, w_l\}$. With
$\zv{d}$ we denote the $d$-dimensional zero vector.

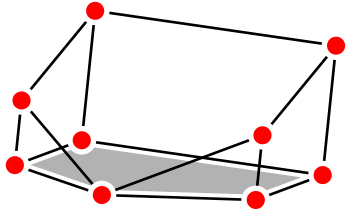
\begin{figure}[bt]
  \centering
\tikzstyle{perspective adjusted birkhoff}=[%
x={(.8,.3)},%
y={(.2,1.5)},%
z={(2.4,-.5)}]
  \begin{tikzpicture}[perspective adjusted birkhoff,scale=1.1]
    \foreach \x in {0,1,2} {
      \foreach \y in {0,1,2} {
        \foreach \z in {0,1,2} {
          \coordinate (a\x\y\z) at (\x,\z/2,\y);
        }
      }
    }

    \fill[line width=1pt,black!30] (a010) -- (a020) -- (a120) -- (a200) -- (a100) -- cycle;

    \draw [line width=4pt,white] (a010) -- (a020) -- (a120) -- (a200) -- (a100) -- cycle;
    \draw [line width=1pt, black] (a010) -- (a020) -- (a120) -- (a200) -- (a100) -- cycle;
    \draw [line width=1pt] (a010) -- (a021) -- (a122) -- (a202) -- (a101) --
    cycle;
    \draw [line width=1pt] (a020) -- (a021);
    \draw [line width=1pt] (a120) -- (a122);
    \draw [line width=1pt] (a200) -- (a202);
    \draw [line width=1pt] (a100) -- (a101);
      
    \draw [color=white,fill=red,line width=2pt] (a010) circle (4pt);
    \draw [color=white,fill=red,line width=2pt] (a020) circle (4pt);
    \draw [color=white,fill=red,line width=2pt] (a120) circle (4pt);
    \draw [color=white,fill=red,line width=2pt] (a100) circle (4pt);
    \draw [color=white,fill=red,line width=2pt] (a200) circle (4pt);
    \draw [color=white,fill=red,line width=2pt] (a021) circle (4pt);
    \draw [color=white,fill=red,line width=2pt] (a122) circle (4pt);
    \draw [color=white,fill=red,line width=2pt] (a101) circle (4pt);
    \draw [color=white,fill=red,line width=2pt] (a202) circle (4pt);

  \end{tikzpicture}
  \caption{The wedge over a vertex of a pentagon.}
  \label{fig:wedge}
\end{figure}
The \emph{(geometric) product} of $P_1$ and $P_2$ is the polytope
\begin{align*}
  P_1\times P_2\ :=\ \conv\left(\,(v_i,w_j)\in\R^{d_1+d_2}\mid 1\le i\le k, 1\le j\le l\,\right).
\end{align*}
This is the same as the set of all points $(v,w)$ for $v\in P_1$ and $w\in P_2$.  The
\emph{(geometric) join} of $P_1$ and $P_2$ is the polytope
\begin{align*}
  P_1\star P_2\ :=\ \conv\left(\,P_1\times\{\zv{d_2}\}\times\{0\}\cup \{\zv{d_1}\}\times
    P_2\times\{1\}\,\right)\ \subseteq\ \R^{d_1+d_2+1}.
\end{align*}
More generally, we say, that a polytope $P$ is a \emph{product} or \emph{join} of two polytopes
$P_1$ and $P_2$, if $P$ is combinatorially isomorphic to the geometric product or geometric join of
some realizations of the face lattices of $P_1$, or $P_2$.

If $F$ is a face of a polytope $P:=\{x\mid Ax\le b\}\subseteq\R^d$ and $\langle c, x\rangle\le d$ a
linear functional defining $F$, then the \emph{wedge} $\wedge_F(P)$ of $P$ over $F$ is defined to be
the polytope
\begin{align}\label{eq:faces:wedge}
  \wedge_F(P)\ :=\ \left\{\, (x,x_0)\in\R^{d+1}\mid Ax\le b,\, 0\le x_0\le
  d-\langle c, x\rangle\,\right\}\,.
\end{align}
See \autoref{fig:wedge} for an example. Again, we say more generally that $P$ is a \emph{wedge} of a
polytope $Q$ over some face $F$ of $Q$ if $P$ is combinatorially equivalent to $\wedge_F(Q)$.

We also extend these notions to combinatorial types, \emph{i.e.}, we say that a combinatorial type
$\lattice$ (or face lattice) of a polytope $P$ is a \emph{cube}, \emph{simplex}, \emph{product},
\emph{join}, or \emph{wedge}, if some geometric realization (and, hence, also any other) of
$\lattice$ is.

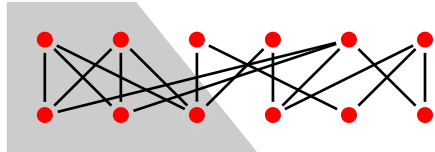
\begin{figure}[b]
  \centering
  \begin{tikzpicture}[scale=1]
    \foreach \x in {0,1,2,3,4,5} {
      \coordinate (u\x) at (\x,1);
      \coordinate (v\x) at (\x,0);
    }
      
    \fill[color=black!20] (-.5,-.5) -- (-.5,1.5) -- (1.2,1.5) -- (2.8,-.5)  -- cycle;

    \draw [line width=1pt] (u0) -- (v0);
    \draw [line width=1pt] (u0) -- (v1);
    \draw [line width=1pt] (u0) -- (v2);
    \draw [line width=1pt] (u1) -- (v1);
    \draw [line width=1pt] (u1) -- (v2);
    \draw [line width=1pt] (u1) -- (v0);
    \draw [line width=1pt] (u2) -- (v2);
    \draw [line width=1pt] (u2) -- (v4);
    \draw [line width=1pt] (u3) -- (v3);
    \draw [line width=1pt] (u3) -- (v2);
    \draw [line width=1pt] (u4) -- (v3);
    \draw [line width=1pt] (u4) -- (v5);
    \draw [line width=1pt] (u4) -- (v1);
    \draw [line width=1pt] (u4) -- (v0);
    \draw [line width=1pt] (u5) -- (v3);
    \draw [line width=1pt] (u5) -- (v5);
    \draw [line width=1pt] (u5) -- (v4);
    \foreach \x in {2,3,4,5} {
      \draw [color=white,fill=red,line width=2pt] (u\x) circle (4pt);
      \draw [color=white,fill=red,line width=2pt] (v\x) circle (4pt);
    }
    \foreach \x in {0,1} {
      \draw [color=black!20,fill=red,line width=2pt] (u\x) circle (4pt);
      \draw [color=black!20,fill=red,line width=2pt] (v\x) circle (4pt);
    }
    \draw [color=black!20,fill=red,line width=2pt] (v2) circle (4pt);
  \end{tikzpicture}
  \caption{The common neighbors of the first three nodes in the lower
    layer are the first two nodes in the upper layer.}
  \label{fig:commonNeighbors}
\end{figure}
With $\neigh(v)$ for a node $v$ of a graph $G$ we denote the \emph{neighborhood} of $v$,
\emph{i.e.}, the set of all nodes in $G$ that are connected to $v$ by an edge. If $M$ is a set of
nodes in $G$, then we denote by $\cneigh(M)$ the \emph{set of common neighbors} of all nodes in $M$,
\emph{i.e.}, the set
\begin{align*}
  \cneigh(M)\ :=\ \bigcap_{v\in M}\neigh(v)\,.
\end{align*}
See \autoref{fig:commonNeighbors} for an example.

\subsection{The Birkhoff polytope}
\label{fb:subsec:birkhoff}
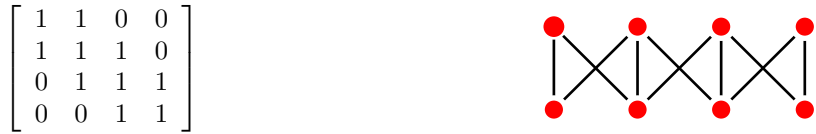
\begin{figure}[t]
  \centering
  \begin{minipage}{.4\linewidth}
    $\left[
      \begin{array}{rrrr}
        1& 1& 0& 0\\ 1& 1& 1& 0\\ 0& 1& 1& 1\\ 0& 0& 1& 1
      \end{array}
    \right]$
  \end{minipage}
  \quad 
  \begin{minipage}{.4\linewidth}
    \centering 
    \begin{tikzpicture}[scale=1.1,line join=round]
      \draw[line width=1pt] (1,1)--(1,0);
      \draw[line width=1pt] (2,1)--(3,0);
      \draw[line width=1pt] (3,1)--(2,0);
      \draw[line width=1pt] (1,1)--(2,0);
      \draw[line width=1pt] (2,1)--(2,0);
      \draw[line width=1pt] (2,1)--(1,0);
      \draw[line width=1pt] (3,1)--(3,0);
      \draw[line width=1pt] (3,1)--(4,0);
      \draw[line width=1pt] (4,1)--(3,0);
      \draw[line width=1pt] (4,1)--(4,0);
      \draw[color=white,fill=red,line width=1pt] (1,1) circle (4pt);
      \draw[color=white,fill=red,line width=2pt] (2,1) circle (4pt);
      \draw[color=white,fill=red,line width=2pt] (3,1) circle (4pt);
      \draw[color=white,fill=red,line width=2pt] (4,1) circle (4pt);
      \draw[color=white,fill=red,line width=2pt] (1,0) circle (4pt);
      \draw[color=white,fill=red,line width=2pt] (2,0) circle (4pt);
      \draw[color=white,fill=red,line width=2pt] (3,0) circle (4pt);
      \draw[color=white,fill=red,line width=2pt] (4,0) circle (4pt);
    \end{tikzpicture}
  \end{minipage}
  \caption{A face of $\birkhoff[3]$ and its graph. The upper layer represents
    the rows, the lower layer the columns of the matrix. An edge of
    the graph represents a $1$ in the matrix at the position
    corresponding to its end points.}
\label{fig:second-example}
\end{figure}

Let $S_n$ be the group of permutations on $n$ elements. To any element $\sigma\in S_n$ we can
associate a $0/1$-matrix $M(\sigma)\in\R^{n\times n}$ that has a $1$ at position $(i,j)$ if and only
if $\sigma(i)=j$. The \emph{$n$-th Birkhoff polytope} is
\begin{align*}
  \birkhoff\ :=\ \conv\left(\,M(\sigma)\,\mid\, \sigma\in S_n\right)\ \subseteq\ \R^{n\times n}\,.
\end{align*}
The Birkhoff-von Neumann Theorem shows that $\birkhoff$ can equally be characterized as the set of
all non-negative $(n\times n)$-matrices whose rows and columns all sum to $1$. Equivalently, the
facets of $\birkhoff$ are precisely defined by the inequalities $x_{ij}\ge 0$ for $1\le i, j\le
n$. It has dimension $(n-1)^2$ with $n^2$ facets and $n!$ vertices.

More generally, we associate a $0/1$-matrix $M(\Sigma)\in\R^{n\times n}$ to any subset
$\Sigma\subseteq S_n$ in the following way. $M(\Sigma)$ has a $1$ at position $(i,j)$ if there is
some $\tau\in \Sigma$ with $\tau(i)=j$, and $0$ otherwise. If $\Sigma=\{\sigma\}$ for some
$\sigma\in S_n$, then $M(\Sigma)=M(\sigma)$.

We can view $M(\Sigma)$ as a dual vector in $(\R^{n\times n})^*$. The functional $M(\Sigma)$
satisfies
\begin{align*}
  \langle M(\Sigma), x\rangle\ \le\ n\ \ \text{for all } x\in \birkhoff\,.
\end{align*}
Any $x=M(\sigma)$ for a $\sigma\in \Sigma$ satisfies this with equality, so this inequality defines
a proper non-empty face
\begin{align*}
  \face(\Sigma)\ :=\ \left\{\,M(\sigma)\,\mid\, \langle M(\Sigma), M(\sigma)\rangle = n\,\right\}\,.
\end{align*}
of the polytope $\birkhoff$, and all $\sigma\in \Sigma$ are vertices of that face.  However, there
may be more. Namely, any permutation $\tau$ such that for any $i\in [n]$ there is $\sigma\in \Sigma$
with $\tau(i)=\sigma(i)$ is also a vertex of $\face(\Sigma)$.  The well-known fact that any face is
defined by a subset of the facet inequalities implies the following proposition.
\begin{proposition}
  Any face $F$ of $\birkhoff$ is of the type $\face(\Sigma)$ for some $\Sigma\subseteq S_n$, and
  $\face(\Sigma)$ is the smallest face containing all vertices corresponding to elements of
  $\Sigma$.
\end{proposition}
Different subsets of $S_n$ may define the same face of $\birkhoff$, so this correspondence is not a
bijection. For example, $\face(\Sigma)$ is the same square in $\birkhoff[4]$ for either of the sets
$\Sigma=\{(), (1\ 2)(3\ 4)\}\subset S_4$ and $\Sigma=\{(1\ 2), (3\ 4)\}\subset S_n$ (and the
vertices of the square correspond to the union of those two sets).

For the following considerations there is a different representation of faces that is easier to deal
with. For any subset $\Sigma\subseteq S_n$ we associate a bipartite graph $\gr(\Sigma)$ with $n$
nodes in each color class to the matrix $M(\Sigma)$.  Let $U = \{u_1, \ldots, u_n\}$ and $V=\{v_1,
\ldots, v_n\}$ be two disjoint vertex sets and draw an edge between the nodes $u_i$ and $v_j$ if and
only if there is $\sigma\in\Sigma$ with $\sigma(i)=j$. This gives a bipartite graph with two color
classes $U$ and $V$ of equal size $n$. In the following, we call $U$ the \emph{upper layer} and $V$
the \emph{lower layer}. \autoref{fig:second-example} shows an example. In this example, $\Sigma$ can
be chosen to contain the identity permutation and the transpositions $(1\ 2)$, $(2\ 3)$ and $(3\
4)$. The face $\face(\Sigma)$ also contains the vertex corresponding to the permutation $(1\ 2)(3\
4)$.  Clearly, the bipartite graph is just a different representation of the matrix. We can recover
the matrix by putting a $1$ at each position $(i,j)$ where node $i$ of the upper layer is connected
to node $j$ of the lower layer, and $0$ everywhere else.

Any vertex of the face $\face(\Sigma)$ corresponds to a perfect matching in the graph $\gr(\Sigma)$,
and any perfect matching in the graph defines a vertex. Conversely, any bipartite graph with the
property that every edge is contained in a perfect matching defines a face of $\birkhoff$ as the
convex hull of the permutations defined by its perfect matchings. In the following, we will use the
term \emph{face graph} for bipartite graphs such that each edge is contained in some perfect
matching of the graph.

Note, that in the literature graphs in which every edge is in some perfect matching are called
\emph{elementary}. So a face graph is a bipartite elementary graph.  Elementary graphs are
well-studied objects, see, \emph{e.g.}, the work of Lov{\'a}sz~\cite{Lovasz83}, Lov{\'a}sz and
Plummer~\cite{LP86}, Brualdi and Shader~\cite{BS93}, and the extensive work of de Carvalho, Lucchesi
and Murty~\cite{dCLM05,dCLM03,dCLM02-1,dCLM02-2,dCLM99}.  In the special case of bipartite graphs,
being elementary implies that both layers have the same number of nodes. An important property of
elementary graphs is the existence of an ear decomposition, which we will explain now.

\begin{figure}[t]
  \centering
    \subfigure[A face graph $G$ with three ears $P_1$, $P_2$, and $P_3$]{%
      \begin{tikzpicture}[line join=round,scale=1.1]
        \filldraw(-2,0) circle (0pt);
        \filldraw(6,0) circle (0pt);
        \draw[line width=1pt] (1,1)--(1,0);
        \draw[line width=1pt] (2,1)--(3,0);
        \draw[line width=1pt] (3,1)--(2,0);
        \draw[line width=1pt] (1,1)--(2,0);
        \draw[line width=1pt] (2,1)--(2,0);
        \draw[line width=1pt] (2,1)--(1,0);
        \draw[line width=1pt] (3,1)--(3,0);
        \draw[line width=1pt] (3,1)--(4,0);
        \draw[line width=1pt] (4,1)--(3,0);
        \draw[line width=1pt] (4,1)--(4,0);
        \draw[color=white,fill=red,line width=2pt] (1,1) circle (4pt);
        \draw[color=white,fill=red,line width=2pt] (2,1) circle (4pt);
        \draw[color=white,fill=red,line width=2pt] (3,1) circle (4pt);
        \draw[color=white,fill=red,line width=2pt] (4,1) circle (4pt);
        \draw[color=white,fill=red,line width=2pt] (1,0) circle (4pt);
        \draw[color=white,fill=red,line width=2pt] (2,0) circle (4pt);
        \draw[color=white,fill=red,line width=2pt] (3,0) circle (4pt);
        \draw[color=white,fill=red,line width=2pt] (4,0) circle (4pt);
      \end{tikzpicture}
    }%

    \subfigure[The ear $P_1$]{%
      \begin{tikzpicture}[line join=round,scale=1.1]
        \draw[line width=1pt] (1,1)--(1,0);
        \draw[line width=1pt] (1,1)--(2,0);
        \draw[line width=1pt] (2,1)--(2,0);
        \draw[line width=1pt] (2,1)--(1,0);
        \draw[line width=1pt,color=gray,dashed](3,1)--(3,0);
        \draw[line width=1pt,color=gray,dashed](2,1)--(3,0);
        \draw[line width=1pt,color=gray,dashed](3,1)--(2,0);
        \draw[line width=1pt,color=gray,dashed](3,1)--(4,0);
        \draw[line width=1pt,color=gray,dashed](4,1)--(3,0);
        \draw[line width=1pt,color=gray,dashed](4,1)--(4,0);
        \draw[color=white,fill=red,line width=2pt] (1,1) circle (4pt);
        \draw[color=white,fill=red,line width=2pt] (2,1) circle (4pt);
        \draw[color=white,fill=red,line width=2pt] (3,1) circle (4pt);
        \draw[color=white,fill=red,line width=2pt] (4,1) circle (4pt);
        \draw[color=white,fill=red,line width=2pt] (1,0) circle (4pt);
        \draw[color=white,fill=red,line width=2pt] (2,0) circle (4pt);
        \draw[color=white,fill=red,line width=2pt] (3,0) circle (4pt);
        \draw[color=white,fill=red,line width=2pt] (4,0) circle (4pt);
      \end{tikzpicture}
    }\quad\qquad%
    \subfigure[The ear $P_2$]{%
      \begin{tikzpicture}[line join=round,scale=1.1]
        \draw[line width=1pt] (2,1)--(3,0);
        \draw[line width=1pt] (3,1)--(2,0);
        \draw[line width=1pt] (3,1)--(4,0);
        \draw[line width=1pt] (4,1)--(3,0);
        \draw[line width=1pt] (4,1)--(4,0);
        \draw[line width=1pt,color=gray,dashed](1,1)--(1,0);
        \draw[line width=1pt,color=gray,dashed](1,1)--(2,0);
        \draw[line width=1pt,color=gray,dashed](2,1)--(1,0);
        \draw[line width=1pt,color=gray,dashed](2,1)--(2,0);
        \draw[line width=1pt,color=gray,dashed](3,1)--(3,0);
        \draw[color=white,fill=red,line width=2pt] (1,1) circle (4pt);
        \draw[color=white,fill=red,line width=2pt] (2,1) circle (4pt);
        \draw[color=white,fill=red,line width=2pt] (3,1) circle (4pt);
        \draw[color=white,fill=red,line width=2pt] (4,1) circle (4pt);
        \draw[color=white,fill=red,line width=2pt] (1,0) circle (4pt);
        \draw[color=white,fill=red,line width=2pt] (2,0) circle (4pt);
        \draw[color=white,fill=red,line width=2pt] (3,0) circle (4pt);
        \draw[color=white,fill=red,line width=2pt] (4,0) circle (4pt);
      \end{tikzpicture}
    }\quad\qquad%
    \subfigure[The ear $P_3$]{%
      \begin{tikzpicture}[line join=round,scale=1.1]
        \draw[line width=1pt] (3,1)--(3,0);
        \draw[line width=1pt,color=gray,dashed](1,1)--(1,0);
        \draw[line width=1pt,color=gray,dashed](1,1)--(2,0);
        \draw[line width=1pt,color=gray,dashed](2,1)--(1,0);
        \draw[line width=1pt,color=gray,dashed](2,1)--(2,0);
        \draw[line width=1pt,color=gray,dashed](2,1)--(3,0);
        \draw[line width=1pt,color=gray,dashed](3,1)--(2,0);
        \draw[line width=1pt,color=gray,dashed](3,1)--(4,0);
        \draw[line width=1pt,color=gray,dashed](4,1)--(3,0);
        \draw[line width=1pt,color=gray,dashed](4,1)--(4,0);
        \draw[color=white,fill=red,line width=2pt] (1,1) circle (4pt);
        \draw[color=white,fill=red,line width=2pt] (2,1) circle (4pt);
        \draw[color=white,fill=red,line width=2pt] (3,1) circle (4pt);
        \draw[color=white,fill=red,line width=2pt] (4,1) circle (4pt);
        \draw[color=white,fill=red,line width=2pt] (1,0) circle (4pt);
        \draw[color=white,fill=red,line width=2pt] (2,0) circle (4pt);
        \draw[color=white,fill=red,line width=2pt] (3,0) circle (4pt);
        \draw[color=white,fill=red,line width=2pt] (4,0) circle (4pt);
      \end{tikzpicture}
    }%
  \caption{An ear decomposition of an elementary graph}
  \label{fig:eardecomp}
\end{figure}
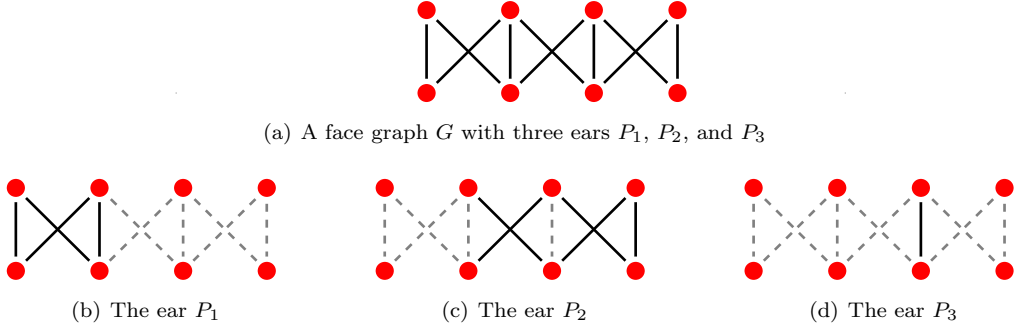
\begin{definition}
  An \emph{ear decomposition} of an elementary graph $G$ is a decomposition of the graph into edge
  disjoint paths and cycles $P_1, P_2, \ldots, P_r$ subject to the following two conditions:
  \begin{enumerate}
  \item $P_1$ is a cycle.
  \item If $P_i$, $1\le i\le r$ is a path, then its endpoints lie in different layers of $G$. These
    are the only two points that $P_i$ has in common with the graph $P_1\cup P_2\cup \ldots\cup
    P_{i-1}$.
  \item If $P_i$, $1\le i\le r$ is a cycle, then $P_i$ is disjoint from $P_1\cup P_2\cup \ldots\cup
    P_{i-1}$.
  \end{enumerate}
\end{definition}
See \autoref{fig:eardecomp} for an example.  The following result can, \emph{e.g.}, be found in the
book of Lov{\'a}sz and Plummer \cite[Thm.\ 4.1.6]{LP86}.
\begin{theorem}
  A bipartite graph $G$ is elementary if and only if all its connected components have an ear
  decomposition.  \qed
\end{theorem}
By a simple counting argument one can show that the number of ears is independent of the chosen ear
decomposition. If $n$ is the number of vertices in each layer, $m$ the number of edges, and $k$ the
number of connected components, then the graph has $r=m-2n+k+1$ ears.  The ear decomposition also
guarantees that an elementary graph is $2$-connected and any node has degree at least $2$.

Let $\FG(n)$ be the set of face graphs where each layer has $n$ vertices. By the above there is a
bijection between faces of $\birkhoff$ and elements of $\FG(n)$. Let $\gr(F)$ be the face graph
corresponding to the face $F$. For faces $E, F$ of $\birkhoff$ the face $E$ is a face of $F$ if and
only if $\gr(E)$ is a subgraph of $\gr(F)$. Hence, we can read off the face lattice of a face $F$
from its graph $\gr(F)$. Two graphs that define isomorphic lattices are said to have the same
\emph{combinatorial type}.

For a face graph $G$ with upper layer $U$ and any set $S\subseteq U$ we have $|\neigh(S)|>|S|$
unless $S\cup\neigh(S)$ is the vertex set of a connected component of $G$ (otherwise, any edge with
one end in $\neigh(S)$ and one not in $S$ could never be part of a perfect matching). Hence, if $G$
is a face graph then any graph obtained by adding edges without reducing the number of connected
components is again a face graph.

Edges of the face correspond to unions of two perfect matchings $M_1, M_2$ that do not imply any
other perfect matchings. Hence, for an edge, $M_1\cup M_2$ is a disjoint union of edges and a single
cycle.

To study face graphs and the faces of $\birkhoff$ they define it is sometimes convenient to consider
a more general representation of a face as a graph.  A multi-graph $\ov G$ is a graph where more
than one edge between two nodes is allowed. As for simple graphs, a matching in a multi-graph is a
selection of edges such that no vertex is incident to more than one edge. It is perfect if every
node is incident to precisely one edge. Again, we can define the associated lattice of face
multi-graphs and their combinatorial types.

In a face graph we can replace any edge with a path of length $3$ without changing the number of
perfect matchings and their inclusion relation. For a face multi-graph $\ov G$ we define its
\emph{resolution} $\res(\ov G)$ to be the graph obtained from $\ov G$ by replacing all but one edge
between any pair of nodes by a path of length $3$. See \autoref{fig:face-multi} for an example.

\begin{figure}[tb] 
  \centering
  \subfigure[]{%
    \begin{tikzpicture}[line join=round,scale=1.1]
      \draw[line width=1pt]  (1,1) to [bend left] (1,0);
      \draw[line width=1pt]  (1,1) to [bend right] (1,0);
      \draw[line width=1pt]  (2,1) to (3,0);
      \draw[line width=1pt]  (3,1) to (2,0);
      \draw[line width=1pt]  (1,1) to (2,0);
      \draw[line width=1pt]  (2,1) to (2,0);
      \draw[line width=1pt]  (2,1) to (1,0);
      \draw[line width=1pt]  (2,1) to [bend left] (3,0);
      \draw[line width=1pt]  (2,1) to [bend right] (3,0);
      \draw[line width=1pt]  (3,1) to (3,0);
      \draw[color=white,fill=red,line width=2pt] (1,1) circle (4pt);
      \draw[color=white,fill=red,line width=2pt] (2,1) circle (4pt);
      \draw[color=white,fill=red,line width=2pt] (3,1) circle (4pt);
      \draw[color=white,fill=red,line width=2pt] (1,0) circle (4pt);
      \draw[color=white,fill=red,line width=2pt] (2,0) circle (4pt);
      \draw[color=white,fill=red,line width=2pt] (3,0) circle (4pt);
    \end{tikzpicture}
  }%
  \hspace*{1cm}
  \subfigure[]{%
    \begin{tikzpicture}[line join=round,scale=1.1]
      \draw[line width=1pt]  (1,1) to [bend right] (1,0);
      \draw[line width=1pt]  (3,1) to (2,0);
      \draw[line width=1pt]  (2,1) to [bend right] (3,0);
      \draw[color=white,fill=red,line width=2pt] (1,1) circle (4pt);
      \draw[color=white,fill=red,line width=2pt] (2,1) circle (4pt);
      \draw[color=white,fill=red,line width=2pt] (3,1) circle (4pt);
      \draw[color=white,fill=red,line width=2pt] (1,0) circle (4pt);
      \draw[color=white,fill=red,line width=2pt] (2,0) circle (4pt);
      \draw[color=white,fill=red,line width=2pt] (3,0) circle (4pt);
    \end{tikzpicture}
  }%

  \subfigure[]{%
    \begin{tikzpicture}[line join=round,scale=1.1]
      \draw[line width=1pt]  (0,1) to (0,0);
      \draw[line width=1pt]  (0,1) to (1,0);
      \draw[line width=1pt]  (1,1) to (0,0);
      
      \draw[line width=1pt]  (1,1) to (1,0);
      \draw[line width=1pt]  (3,1) to (2,0);
      \draw[line width=1pt]  (1,1) to (2,0);
      \draw[line width=1pt]  (2,1) to (2,0);
      \draw[line width=1pt]  (2,1) to (1,0);
      \draw[line width=1pt]  (2,1) to (3,0);
      \draw[line width=1pt]  (3,1) to (3,0);
      
      \draw[line width=1pt]  (2,1) to (4,0);
      \draw[line width=1pt]  (4,1) to (4,0);
      \draw[line width=1pt]  (4,1) to (3,0);
      
      \draw[line width=1pt]  (2,1) to (5,0);
      \draw[line width=1pt]  (5,1) to (5,0);
      \draw[line width=1pt]  (5,1) to (3,0);     

      \draw[color=white,fill=red,line width=2pt] (0,1) circle (4pt);
      \draw[color=white,fill=red,line width=2pt] (1,1) circle (4pt);
      \draw[color=white,fill=red,line width=2pt] (2,1) circle (4pt);
      \draw[color=white,fill=red,line width=2pt] (3,1) circle (4pt);
      \draw[color=white,fill=red,line width=2pt] (4,1) circle (4pt);
      \draw[color=white,fill=red,line width=2pt] (5,1) circle (4pt);
      
      \draw[color=white,fill=red,line width=2pt] (0,0) circle (4pt);
      \draw[color=white,fill=red,line width=2pt] (1,0) circle (4pt);
      \draw[color=white,fill=red,line width=2pt] (2,0) circle (4pt);
      \draw[color=white,fill=red,line width=2pt] (3,0) circle (4pt);
      \draw[color=white,fill=red,line width=2pt] (4,0) circle (4pt);
      \draw[color=white,fill=red,line width=2pt] (5,0) circle (4pt);
    \end{tikzpicture}
  }%
  \caption{A face multi-graph $\ov G$, a perfect matching in $\ov G$,
    and the resolution $\res(\ov G)$ of $G$.}
  \label{fig:face-multi}
\end{figure}
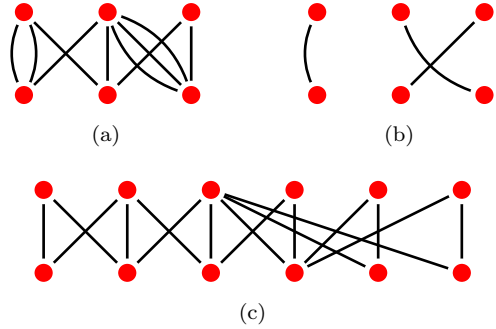
We can also define a converse operation. Let $\gr(F)$ be a face graph. For a vertex $v$ of degree
$2$ we introduce the \emph{reduction} $\red(\gr(F),v)$ at a vertex $v$.  Let $F$ be a face graph
with a vertex $v$ of degree $2$ and neighbors $u_1$, $u_2$. The reduction $\red(\gr(F),v)$ of $F$ at
$v$ is the graph obtained from $\gr(F)$ by contracting $v$. This graph may have double edges. %
See \autoref{fig:contraction} for an illustration. %
This construction already appears in the paper of Billera and Sarangarajan~\cite{BS96}.  We
summarize the properties of $\gr(F)$ and a face multi-graph $\ov G$.
\begin{proposition}\label{prop:isomgraphs}
  Let $F$ be a face of some Birkhoff polytope, and $\ov G$ a face multi graph corresponding to
  $F$. Then
  \begin{enumerate}
  \item $\gr(F)$ and its reduction $\red(\gr(F))$ have the same combinatorial type,
  \item $\ov G$ and its resolution $\res(\ov G)$ have the same combinatorial type.\qed
  \end{enumerate}
\end{proposition}

In the next sections we want to study combinatorial types of faces of $\birkhoff$ by studying their
face graphs. The following proposition tells us that we can mostly restrict our attention to
connected face graphs.
\begin{proposition}
  Let $G$ be a face graph with connected components $G_1, \ldots, G_k$, $k\ge 2$. Then $G_1, \ldots,
  G_k$ are face graphs and the face corresponding to $G$ is the product of the faces corresponding
  to $G_1, \ldots, G_k$.
\end{proposition}
\begin{proof}
  Any perfect matching in $G$ induces a perfect matching in $G_i$, $1\le i\le k$. Hence, any edge in
  $G_i$ is contained in a perfect matching of $G_i$, so all $G_i$ are face graphs. Perfect matchings
  correspond to vertices, and any combination of perfect matchings in the $G_i$ defines a perfect
  matching of $G$. Hence, the face of $G$ is a product.
\end{proof}
In particular, prisms over faces are again faces, and their face graph is obtained by adding a
disjoint cycle of length $4$.
\begin{figure}[tb]
  \centering
  \subfigure[The graph $G$]{%
    \begin{tikzpicture}[scale=1.1]
      
      \coordinate (u1) at ( 0,1);
      \coordinate (u2) at ( 1,1);
      \coordinate (u3) at ( 2,1);
      \coordinate (u4) at ( 3,1);
      \coordinate (v1) at ( 0,0);
      \coordinate (v2) at ( 1,0);
      \coordinate (v3) at ( 2,0);
      \coordinate (v4) at ( 3,0);

      \draw[line width=1pt]  (u1) -- (v3);
      \draw[line width=1pt]  (u2) -- (v2);
      \draw[line width=1pt]  (u2) -- (v3);
      \draw[line width=1pt]  (u2) -- (v4);
      \draw[line width=1pt]  (u3) -- (v3);
      \draw[line width=1pt]  (u3) -- (v4);
      \draw[line width=1pt]  (u4) -- (v3);
      \draw[line width=1pt]  (u4) -- (v4);
      \draw[line width=1pt]  (u1) -- (v2);
      \draw[line width=1pt]  (u1) -- (v1);
      \draw[line width=1pt]  (u2) -- (v1);
      \draw[color=white,fill=red,line width=2pt] (u1) circle (4pt);
      \draw[color=white,fill=red,line width=2pt] (u2) circle (4pt);
      \draw[color=white,fill=red,line width=2pt] (u3) circle (4pt);
      \draw[color=white,fill=red,line width=2pt] (u4) circle (4pt);
      \draw[color=white,fill=red,line width=2pt] (v1) circle (4pt);
      \draw[color=white,fill=red,line width=2pt] (v2) circle (4pt);
      \draw[color=white,fill=red,line width=2pt] (v3) circle (4pt);
      \draw[color=white,fill=red,line width=2pt] (v4) circle (4pt);
      
      \node [above,yshift=3pt] at (u2) {$v_1$};
      \node [above,yshift=3pt] at (u3) {$v_2$};
      \node [above,yshift=3pt] at (u4) {$v$};
      \node [below,yshift=-3pt] at (v2) {$u_1$};
      \node [below,yshift=-3pt] at (v3) {$u_2$};
    \end{tikzpicture}
  }%
  \hspace*{.3cm}
  \subfigure[The graph $\red(G,v)$]{
\begin{tikzpicture}[scale=1.1]
  
  \coordinate (u1) at ( 0,1);
  \coordinate (u2) at ( 1,1);
  \coordinate (u3) at ( 2,1);
  \coordinate (u4) at ( 3,1);
  \coordinate (v1) at ( 0,0);
  \coordinate (v2) at ( 1,0);
  \coordinate (v3) at ( 2,0);
  \coordinate (v4) at ( 3,0);

  \draw[line width=1pt]  (u1) -- (v3);
  \draw[line width=1pt]  (u2) -- (v2);
  \draw[line width=1pt]  (u2) -- (v3);
  \draw[line width=1pt]  (u3) -- (v3);
  \draw[line width=1pt]   (u3) to [bend left=30] (v3);
  \draw[line width=1pt]   (u2) to [bend left=30] (v3);
  \draw[line width=1pt]  (u1) -- (v2);
  \draw[line width=1pt]  (u1) -- (v1);
  \draw[line width=1pt]  (u2) -- (v1);
  \draw[color=white,fill=red,line width=2pt] (u1) circle (4pt);
  \draw[color=white,fill=red,line width=2pt] (u2) circle (4pt);
  \draw[color=white,fill=red,line width=2pt] (u3) circle (4pt);
  \draw[color=white,fill=red,line width=2pt] (v1) circle (4pt);
  \draw[color=white,fill=red,line width=2pt] (v2) circle (4pt);
  \draw[color=white,fill=red,line width=2pt] (v3) circle (4pt);
  \draw[color=white,fill=red,line width=2pt] (-1,0) circle (0pt);

  \node [above,yshift=3pt] at (u2) {$v_1$};
  \node [above,yshift=3pt] at (u3) {$v_2$};
  \node [below,yshift=-3pt] at (v2) {$u$};
\end{tikzpicture}
  }
  \hspace*{.8cm}
  \subfigure[The graph $\res(\red(G,v))$]{
\begin{tikzpicture}[scale=1.1]
  
  \coordinate (u1) at ( 0,1);
  \coordinate (u2) at ( 1,1);
  \coordinate (u3) at ( 2,1);
  \coordinate (u4) at ( 3,1);
  \coordinate (u5) at ( 4,1);
  \coordinate (v1) at ( 0,0);
  \coordinate (v2) at ( 1,0);
  \coordinate (v3) at ( 2,0);
  \coordinate (v4) at ( 3,0);
  \coordinate (v5) at ( 4,0);

  \draw[line width=1pt]  (u1) -- (v3);
  \draw[line width=1pt]  (u2) -- (v2);
  \draw[line width=1pt]  (u2) -- (v3);
  \draw[line width=1pt]  (u2) -- (v4);
  \draw[line width=1pt]  (u3) -- (v3);
  \draw[line width=1pt]  (u3) -- (v5);
  \draw[line width=1pt]  (u4) -- (v4);
  \draw[line width=1pt]  (u5) -- (v5);
  \draw[line width=1pt]  (u5) -- (v3);
  \draw[line width=1pt]  (u4) -- (v3);
  \draw[line width=1pt]  (u4) -- (v4);
  \draw[line width=1pt]  (u1) -- (v2);
  \draw[line width=1pt]  (u1) -- (v1);
  \draw[line width=1pt]  (u2) -- (v1);
  \draw[color=white,fill=red,line width=2pt] (u1) circle (4pt);
  \draw[color=white,fill=red,line width=2pt] (u2) circle (4pt);
  \draw[color=white,fill=red,line width=2pt] (u3) circle (4pt);
  \draw[color=white,fill=red,line width=2pt] (u4) circle (4pt);
  \draw[color=white,fill=red,line width=2pt] (u5) circle (4pt);
  \draw[color=white,fill=red,line width=2pt] (v1) circle (4pt);
  \draw[color=white,fill=red,line width=2pt] (v2) circle (4pt);
  \draw[color=white,fill=red,line width=2pt] (v3) circle (4pt);
  \draw[color=white,fill=red,line width=2pt] (v4) circle (4pt);
  \draw[color=white,fill=red,line width=2pt] (v5) circle (4pt);

  \node [above,yshift=3pt] at (u2) {$v_1$};
  \node [above,yshift=3pt] at (u3) {$v_2$};
  \node [below,yshift=-3pt] at (v2) {$u$};
\end{tikzpicture}
  }
  \caption{Contracting a vertex of degree $2$}
  \label{fig:contraction}
\end{figure}
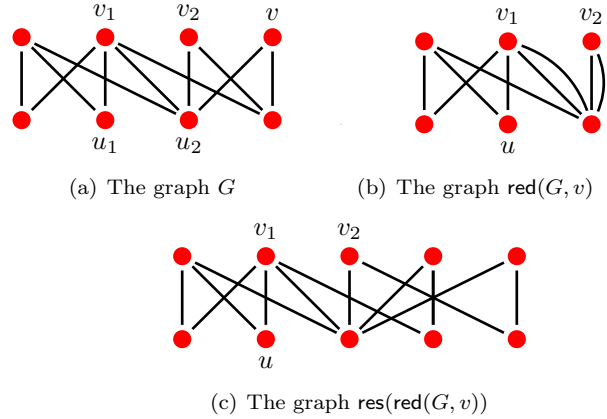

The converse statement, \emph{i.e.}, that the face graph of a face that combinatorially is a product
of two other faces is always disconnected, follows, \emph{e.g.}, from \cite[Cor.~4.7]{BG77-2}, where
they prove that the (vertex-edge) graph of two faces is isomorphic if and only if the two face
graphs have the same number of components and there is a correspondence between the components such
that each pair has isomorphic graphs.

We can read off the dimension of $\face(\Sigma)$ for $\Sigma\subseteq S_n$ from the graph
$\gr(\Sigma)$. Assume first that the face graph is connected with $n$ nodes in each layer and $m$
edges. Then the Birkhoff-von Neumann Theorem implies that the dimension $d$ is at most $m-2n+1$. On
the other hand, the graph has an ear decomposition with $m-2n+2$ ears, so we have at least $m-2n+2$
linearly independent vertices in the face. So $d=m-2n+1$ for a connected face graph. Using that
disconnected graphs define products we obtain
\begin{align}\label{eq:dimension}
  \dim \face(\Sigma) = m-2n+k\,,
\end{align}
where $k$ is the number of connected components of the graph.  The following theorem of Billera and
Sarangarajan tells us that we can restrict the search for combinatorial types of $d$-dimensional
faces to $\birkhoff$ for $n\le 2d$. We repeat the simple proof, as it is quite instructive.
\begin{theorem}[Billera and Sarangarajan~\cite{BS96}]\label{thm:billera}
  Any $d$-dimensional combinatorial type of face already appears in $\birkhoff[2d]$.
\end{theorem}
\begin{proof}
  We assume first that the face graph $G$ is connected. Then $G$ has $m=2n+(d-1)$ edges. Let $k_2$
  be the number of nodes of degree $2$. A node in $G$ has degree at least $2$, so
  \begin{align*}
    k_2\ \ge\ 2n-2(d-1)\,.
  \end{align*}
  We now successively contract all nodes of degree $2$ using the above reduction. We obtain a face
  multi-graph $G'$ on $2n' = 2n - k_2 \ \le \ 2(d-1)$ nodes, \emph{i.e.}, at most $d-1$ on each
  layer.  Note that one reduction step can remove more than one node of degree $2$. Let $m'$ be the
  number of edges of $G'$. $G'$ defines a face combinatorially equivalent to the one of $G$, in
  particular, its dimension is $d=m'-2n'+1$. The graph $G'$ may have multiple edges. Let $e'$ be the
  minimal number of edges we have to remove to obtain a simple graph $\ov G$. Then $\ov G$ is
  connected and a face graph corresponding to a face of dimension
  \begin{align*}
    0\ \le \ \ov d\ =\ m'-2n'+1-e' \ =\ d-e'\,.
  \end{align*}
  Thus, $e'\le d$, and we can resolve each multiple edge to obtain a face graph $\tilde G$ with at
  most $2(d-1)+2d\le 2(2d-1)\le 4d$ nodes. Hence, $\tilde G$ defines a face of $\birkhoff$ that is
  combinatorially isomorphic to the one of $G$.

  The statement for disconnected graphs follows using induction by
  replacing the graph in each component with the above procedure.
\end{proof}

\section{Irreducibility}
\label{sec:irred}

In general, a combinatorial type of a face can occur many times as a
geometrically realized face of $\birkhoff$. Hence, there are many
different possibilities to represent a combinatorial type of a face as
an face graph.  Brualdi and Gibson \cite[Conj.\ 1]{BG77-3} conjecture
that any two combinatorially isomorphic faces are affinely equivalent,
but as far as we know this is still open.

Let $G$ be a face graph. In the following, we want to examine some
version of \emph{minimality} for such a representation. This will,
however, not lead to a unique ``standard'' representation.  We say
that a node $v$ in $G$ is \emph{reducible}, if $v$ has degree $2$ in
$G$ and the common neighborhood of the two vertices adjacent to $v$
only contains the node $v$.  $v$ is called \emph{irreducible}
otherwise.  A face graph $G$ is called \emph{irreducible}, if all its
nodes are irreducible, and \emph{reducible} otherwise.  An irreducible
representation of a certain $d$-face of a Birkhoff polytope need not
be unique. \autoref{fig:NumberNodesnotUnique} shows two irreducible
representations of the $4$-simplex on a different number of nodes.
See \autoref{fig:RedElementary} for an example of a reducible node.
By \autoref{prop:isomgraphs} the face graphs $G$ and
$G':=\red(G,v)$, for any node $v$ of $G$, define the same
combinatorial type. Hence, we can mostly restrict our considerations
to irreducible face graphs.

\begin{figure}[bt]
  \centering
  \begin{tikzpicture}[scale=1.1]
    \foreach \x in {0,1,2,3} {
      \coordinate (u\x) at (\x,1);
      \coordinate (v\x) at (\x,0);
    }
      
    \draw [line width=1pt] (u0) -- (v0);
    \draw [line width=1pt] (u0) -- (v1);
    \draw [line width=1pt] (u0) -- (v2);
    \draw [line width=1pt] (u1) -- (v1);
    \draw [line width=1pt] (u1) -- (v2);
    \draw [line width=1pt] (u2) -- (v0);
    \draw [line width=1pt] (u2) -- (v1);
    \draw [line width=1pt] (u2) -- (v2);
    \draw [line width=1pt] (u2) -- (v3);
    \draw [line width=1pt] (u3) -- (v2);
    \draw [line width=1pt] (u3) -- (v3);

    \foreach \x in {0,1,2,3} {
      \draw [color=white,fill=red,line width=2pt] (u\x) circle (4pt);
      \draw [color=white,fill=red,line width=2pt] (v\x) circle (4pt);
    }
  \end{tikzpicture}
  \hspace{1cm}
  \begin{tikzpicture}[scale=1.1]
    \foreach \x in {0,1,2,3,4} {
      \coordinate (u\x) at (\x,1);
      \coordinate (v\x) at (\x,0);
    }
      
    \draw [line width=1pt] (u0) -- (v0);
    \draw [line width=1pt] (u0) -- (v2);
    \draw [line width=1pt] (u1) -- (v1);
    \draw [line width=1pt] (u1) -- (v2);
    \draw [line width=1pt] (u2) -- (v0);
    \draw [line width=1pt] (u2) -- (v1);
    \draw [line width=1pt] (u2) -- (v2);
    \draw [line width=1pt] (u2) -- (v3);
    \draw [line width=1pt] (u2) -- (v4);
    \draw [line width=1pt] (u3) -- (v2);
    \draw [line width=1pt] (u3) -- (v3);
    \draw [line width=1pt] (u4) -- (v2);
    \draw [line width=1pt] (u4) -- (v4);

    \foreach \x in {0,1,2,3,4} {
      \draw [color=white,fill=red,line width=2pt] (u\x) circle (4pt);
      \draw [color=white,fill=red,line width=2pt] (v\x) circle (4pt);
    }
  \end{tikzpicture}
  \caption{Two irreducible graphs both defining a $4$-simplex.}
  \label{fig:NumberNodesnotUnique}
\end{figure}
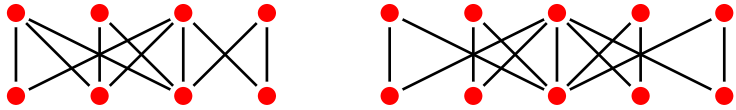
\begin{lemma}\label{lemma:partnerlemma}
  Let $G$ be a face graph, $v\in G$ an irreducible node of
  degree $2$ in $G$ and $w_1, w_2$ its neighbors.  Let
  \begin{align*}
    N:=\left(\neigh(w_1)\cap \neigh(w_2)\right)\setminus\{v\}\,.
  \end{align*}
  Then either the graph induced by $v,w_1,w_2$ and the nodes in $N$ is
  a connected component of $G$ (and then necessarily the set $N$
  contains a single node $u$) or all points $u\in N$ and at least one
  of $w_1,w_2$ have degree $\ge 3$.
\end{lemma}
\begin{figure}[t]
  \centering
  \begin{tikzpicture}[scale=1.1]
    \foreach \x in {0,1,2,3} {
      \coordinate (u\x) at (\x,1);
      \coordinate (v\x) at (\x,0);
    }
      
    \draw [line width=1pt] (u0) -- (v0);
    \draw [line width=1pt] (u0) -- (v1);
    \draw [line width=1pt] (u0) -- (v3);
    \draw [line width=1pt] (u1) -- (v0);
    \draw [line width=1pt] (u1) -- (v1);
    \draw [line width=1pt] (u1) -- (v3);
    \draw [line width=2pt,dashed] (u2) -- (v1);
    \draw [line width=2pt,dashed] (u2) -- (v2);
    \draw [line width=1pt] (u3) -- (v0);
    \draw [line width=1pt] (u3) -- (v2);
    \draw [line width=1pt] (u3) -- (v3);

    \foreach \x in {0,1,2,3} {
      \draw [color=white,fill=red,line width=2pt] (u\x) circle (4pt);
      \draw [color=white,fill=red,line width=2pt] (v\x) circle (4pt);
    }
    \draw [color=white,fill=red,line width=2pt] (u2) circle (4pt) node [above,color=black,yshift=5pt] {$u$};
  \end{tikzpicture}
  \caption{A reducible face graph: The node $u$ incident to the
    two thick dashed edges is reducible.}
  \label{fig:RedElementary}
\end{figure}
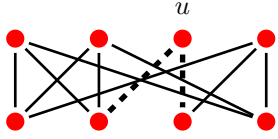
\begin{proof}
  We prove this by contradiction. We are done, if $v,w_1,w_2$ and the nodes in $N$ form a connected
  component. Hence assume this is not the case.

  Assume first that $w_1,w_2$ both have degree $2$.  Then there is some node $u\in N$ that has a
  third neighbor $x$ different from $w_1,w_2$.  However, the edge $(u,x)$ must be contained in a
  perfect matching $M$ of the graph.  This perfect matching cannot use the edges $(u,w_1)$ and
  $(u,w_2)$.  As both nodes $w_1,w_2$ have degree $2$, $M$ must use the remaining edge on both
  nodes.  But these both contain $v$, so $M$ is not a perfect matching.

  So one of $w_1, w_2$ must have degree $\ge 3$.  Assume this is $w_1$, and let $u$, $x$ be the two
  neighbors of $w_1$ different from $v$.  If $u$ would have degree $2$, then $u$ and $v$ are
  contained in the edges $e_i:=(u,w_i)$, $f_i:=(v,w_i)$, $i=1,2$, and no other.  Hence, any perfect
  matching $M$ must choose either $e_1$ or $e_2$, and, correspondingly, $f_2$ or $f_1$.  In either
  case $w_1$ is covered, hence, the edge $(w_1,x)$ can never be chosen, so $G$ is not a face graph.
\end{proof}
If $G$ is an irreducible face graph then we say that a node $v$ is \emph{minimal} if the degree of
$v$ is $2$. For a minimal node $v$ the set
\begin{align*}
  \partners(v):=\cneigh\left(\neigh\left(v\right)\right)
  \setminus{\{v\}}
\end{align*}
is the set of \emph{partners} of the node $v$. This is the same as the set of nodes connected to $v$
via two different paths of length $2$.  A node $x\in\partners(v)$ is called a \emph{partner} of $v$.

Note, that any partners of a node always lie in the same layer as the node itself.
\autoref{lemma:partnerlemma} guarantees that any minimal node in an irreducible face graph has at
least one partner. We use the term \emph{partner} more generally for any node $x$ that is a partner
of some other $v$, without reference to the node $v$.  In particular, $x$ can be partner of several
different nodes in $G$.  However, the next corollary bounds this number from above.
\begin{proposition}\label{cor:partnerbound}
  Let $G$ be a face graph. Any partner $x$ in $G$ of degree $k$ has at most $k-1$ nodes it is
  partner for.
  
  Moreover, if $x$ is a partner for precisely $k-1$ nodes, then these $k-1$ nodes and $x$ are the
  upper or lower layer of a connected component in $G$.
\end{proposition}
\begin{proof}
  If $v$ is a node that has $x$ as its partner, then in any perfect matching $M$, $v$ uses up one of
  the nodes adjacent to $x$ for the edge covering $v$.  Now also $x$ needs to be covered, hence
  there can be at most $k-1$ nodes choosing $x$ as partner.

  If $x$ is a partner for precisely $k-1$ nodes $v_1,\ldots, v_{k-1}$, then in any perfect matching
  in $G$, all but one node in the neighborhood of $x$ is covered by an edge that has one endpoint
  among the $v_i$, $1\le i\le k-1$ or $x$.  But also $x$ needs to be covered, hence, there cannot be
  another edge that ends in a node in the neighborhood of $x$.
\end{proof}
\begin{corollary}
  A connected irreducible face graph of dimension $d$ with $n$ nodes has at most $d$ nodes of degree
  $2$ in each layer, if $n=d+1$, and at most $d-1$ otherwise.
\end{corollary}
\begin{proof}
  Let $k_2$ be the number of nodes of degree $2$ (minimal nodes) in the upper layer. If all minimal
  nodes have the same partner, then, by the previous proposition, the graph has $k_2+1$ nodes, and
  $2k_2+1+k_2=3k_2+1$ edges. Hence, $d=3k_2+1-2k_2-2+1=k_2$. Otherwise, we have at least two
  partners in the upper layer, and the previous proposition implies $2n+k_2\le d+2n-1$, \emph{i.e.},
  $k_2\le d-1$.
\end{proof}
\begin{corollary}\label{cor:irredNodeBound}
  A connected irreducible $d$-dimensional face graph $G$ has at most $2d-2$ nodes.
\end{corollary}
\begin{proof}
  By the previous corollary the graph has at most $d$ nodes if only one node has degree greater than
  $2$. Otherwise, we have at most $d-1$ nodes of degree $2$, and for each of those we need a partner
  of degree at least $3$. This leaves us with $2d-2$ nodes using up all $2n+d-1$ edges.
\end{proof}

\begin{proposition}\label{prop:degree}
  Let $G$ be an connected irreducible face graph of dimension $d$ on $n$ nodes.  Then the maximum
  degree of a node in $G$ is $2d-n+1$ if $n > d+1$ and $n$ otherwise.
\end{proposition}
\begin{proof}
  The bound for $n\le d+1$ is trivial.  If $G$ has dimension $d$ then $G$ has $d+2n-1$ edges, and
  any node has degree at least $2$. Let $k_2$ be the number of nodes of degree $2$ in the upper
  layer and $\delta$ the degree of a non-minimal node $v$. $v$ can be partner for at most $\delta-2$
  nodes, as otherwise $n\le d+1$.  Any minimal node has degree $2$, and any other node at least
  degree $3$. Hence,
  \begin{align*}
    \delta-3&\le d+2n-1-(2k_2+3(n-k_2))=d-n-1+k_2\\&\le d-n-1+d-1
    =2d-n-2\,.
  \end{align*}
  This implies the bound.
\end{proof}
The bound is best possible, see \autoref{fig:maxdeg}.
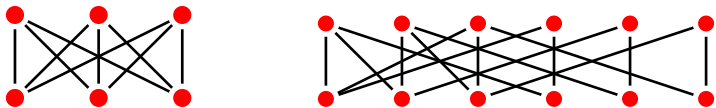
\begin{figure}[tb]
  \centering
  \begin{tikzpicture}[scale=1.1]
    \foreach \x in {0,1,2} {
      \coordinate (u\x) at (\x,1);
      \coordinate (v\x) at (\x,0);
    }
      
    \draw [line width=1pt] (u0) -- (v0);
    \draw [line width=1pt] (u0) -- (v1);
    \draw [line width=1pt] (u0) -- (v2);
    \draw [line width=1pt] (u1) -- (v0);
    \draw [line width=1pt] (u1) -- (v1);
    \draw [line width=1pt] (u1) -- (v2);
    \draw [line width=1pt] (u2) -- (v0);
    \draw [line width=1pt] (u2) -- (v1);
    \draw [line width=1pt] (u2) -- (v2);
    \foreach \x in {0,1,2} {
      \draw [color=white,fill=red,line width=2pt] (u\x) circle (4pt);
      \draw [color=white,fill=red,line width=2pt] (v\x) circle (4pt);
    }    
  \end{tikzpicture}
  \qquad\qquad
  \begin{tikzpicture}
    \foreach \x in {0,1,2,3,4,5} {
      \coordinate (u\x) at (\x,1);
      \coordinate (v\x) at (\x,0);
    }
      
    \draw [line width=1pt] (u0) -- (v0);
    \draw [line width=1pt] (u0) -- (v1);
    \draw [line width=1pt] (u0) -- (v3);
    \draw [line width=1pt] (u1) -- (v1);
    \draw [line width=1pt] (u1) -- (v2);
    \draw [line width=1pt] (u1) -- (v4);
    \draw [line width=1pt] (u2) -- (v0);
    \draw [line width=1pt] (u2) -- (v2);
    \draw [line width=1pt] (u2) -- (v5);
    \draw [line width=1pt] (u3) -- (v0);
    \draw [line width=1pt] (u3) -- (v3);
    \draw [line width=1pt] (u4) -- (v1);
    \draw [line width=1pt] (u4) -- (v4);
    \draw [line width=1pt] (u5) -- (v2);
    \draw [line width=1pt] (u5) -- (v5);
    \foreach \x in {0,1,2,3,4,5} {
      \draw [color=white,fill=red,line width=2pt] (u\x) circle (4pt);
      \draw [color=white,fill=red,line width=2pt] (v\x) circle (4pt);
    }    
  \end{tikzpicture}
  \caption{Face graphs with nodes of maximal degree}
  \label{fig:maxdeg}
\end{figure}

\begin{corollary}\label{cor:MinNumberOfEdges}
  Let $G$ be a connected irreducible face graph with $n\ge 4$ nodes on each of its layers.  Then $G$
  has at least $2n+\lceil\frac{n}{2}\rceil$ edges.
\end{corollary}
\begin{proof}
  Any node has degree at least $2$. By \autoref{lemma:partnerlemma} we have to find a partner of
  higher degree for each node of degree $2$.  On the other hand, \autoref{cor:partnerbound} limits
  the number of minimal nodes a node can be partner for. We distinguish two cases:

  If there is a node $u$ in the graph that is a partner for all minimal nodes, then necessarily
  $\deg(u)=n$, hence the graph has $3n-2\ge 2n+\lceil\frac{n}{2}\rceil$ edges.

  Otherwise, there are $k_2$ minimal nodes and $p\ge 2$ partners in the graph. We consider the cases
  $k_2\ge p$ and $k_2<p$ separately.

  In the first case we have $k_2\ge p$, hence $p\le \lfloor\frac{n}{2}\rfloor$.  The $p$ partners in
  the graph together must be adjacent to at least $2p+k_2$ edges, hence we have at least
  \begin{align*}
    2k_2+2p+k_2+3(n-k_2-p)&=3n-p\ge 2n+\left\lceil\frac n2\right\rceil
  \end{align*}
  edges in the graph.  In the second case the number of edges is at least
  \begin{align*}
    2k_2+3p+3(n-k_2-p)&=3n-k_2\ge 2n+\left\lceil\frac n2\right\rceil\qedhere
  \end{align*}
\end{proof}

\begin{proposition}\label{prop:lowNodeNumbers}
  An irreducible face graph with two nodes in each layer has four edges, an irreducible face graph
  with three nodes in each layer has at least seven edges.
\end{proposition}
\begin{proof}
  The first case is trivial (see \autoref{fig:minimalgraphs}(a)). For the second case just observe
  that we need to have at least one node of degree $\ge 3$.
\end{proof}
The given bounds are tight, as the graphs in \autoref{fig:minimalgraphs} show.

\begin{proposition}
  Let $G$ be an irreducible face graph on $n$ nodes and $v$ a node of degree $k$ in $G$.  Then at
  most $k-1$ neighbors of $v$ can have degree $2$.
\end{proposition}
\begin{proof}
  $v$ must be connected to the partner of all nodes in its neighborhood that have degree $2$.
\end{proof}
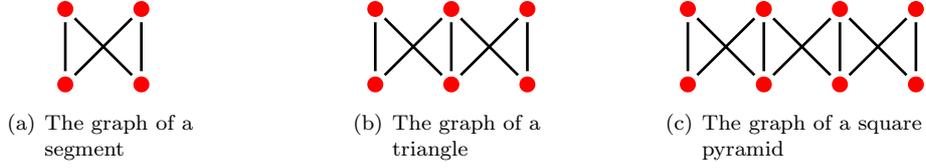
\begin{figure}[bt]
  \centering
  \subfigure[The graph of a segment]{
    \label{fig:graphofsegment}
    \begin{tikzpicture}
      \foreach \x in {-1,0,1,2} {
        \coordinate (u\x) at (\x,1);
        \coordinate (v\x) at (\x,0);
      }
 
      \coordinate (u-1) at (-.5,1);

      \draw [line width=1pt] (u0) -- (v0);
      \draw [line width=1pt] (u0) -- (v1);
      \draw [line width=1pt] (u1) -- (v1);
      \draw [line width=1pt] (u1) -- (v0);
      \foreach \x in {0,1} {
        \draw [color=white,fill=red,line width=2pt] (u\x) circle (4pt);
        \draw [color=white,fill=red,line width=2pt] (v\x) circle (4pt);
      }
      \draw [color=white] (u-1) circle (4pt);
      \draw [color=white] (u2) circle (4pt);
    \end{tikzpicture}
  }
  \qquad\qquad
  \subfigure[The graph of a triangle]{
    \label{fig:graphoftriangle}
    \begin{tikzpicture}
      \foreach \x in {0,1,2} {
        \coordinate (u\x) at (\x,1);
        \coordinate (v\x) at (\x,0);
      }
      
      \draw [line width=1pt] (u0) -- (v0);
      \draw [line width=1pt] (u0) -- (v1);
      \draw [line width=1pt] (u1) -- (v0);
      \draw [line width=1pt] (u1) -- (v1);
      \draw [line width=1pt] (u1) -- (v2);
      \draw [line width=1pt] (u2) -- (v1);
      \draw [line width=1pt] (u2) -- (v2);
      
      \foreach \x in {0,1,2} {
        \draw [color=white,fill=red,line width=2pt] (u\x) circle (4pt);
        \draw [color=white,fill=red,line width=2pt] (v\x) circle (4pt);
      }
    \end{tikzpicture}
  }
  \qquad\qquad
  \subfigure[The graph of a square pyramid]{
    \label{fig:graphofsquarepyramid}
    \begin{tikzpicture}
      \foreach \x in {0,1,2,3} {
        \coordinate (u\x) at (\x,1);
        \coordinate (v\x) at (\x,0);
      }
      
      \draw [line width=1pt] (u0) -- (v0);
      \draw [line width=1pt] (u0) -- (v1);
      \draw [line width=1pt] (u1) -- (v0);
      \draw [line width=1pt] (u1) -- (v1);
      \draw [line width=1pt] (u1) -- (v2);
      \draw [line width=1pt] (u2) -- (v1);
      \draw [line width=1pt] (u2) -- (v2);
      \draw [line width=1pt] (u2) -- (v3);
      \draw [line width=1pt] (u3) -- (v3);
      \draw [line width=1pt] (u3) -- (v2);
      
      \foreach \x in {0,1,2,3} {
        \draw [color=white,fill=red,line width=2pt] (u\x) circle (4pt);
        \draw [color=white,fill=red,line width=2pt] (v\x) circle (4pt);
      }
    \end{tikzpicture}
  }
  \caption{Graphs with a minimal number of edges.}
  \label{fig:minimalgraphs}
\end{figure}

\section{The Structure of Faces of $\birkhoff$}
\label{sec:combtypes}

Here we review some basic properties of facets and faces of Birkhoff polytopes that we need for our
classifications in the following section.

There is a quite canonical way to split the set of vertices of a face of $\birkhoff$ into two
non-empty subsets on parallel hyperplanes at distance one.  Let $G$ be a face graph with some edge
$e$, $M_e$ the set of all perfect matchings in $G$ containing $e$, and $M_{\neg e}$ its
complement. Clearly, both $M_e$ and $M_{\neg e}$ define faces of $\face(G)$, and $G=M_e\cup M_{\neg
  e}$ (not disjoint). Geometrically, if $e$ connects the nodes $i$ and $j$, then all vertices of
$M_e$ satisfy $x_{ij}=1$, while all others lie on the hyperplane $x_{ij}=0$.

We start with some properties of facets of a face $F$ of $\birkhoff$ with face graph $\gr(F)$. The
face graph of a facet of $F$ is a face subgraph of $\gr(F)$. We call a set $C$ of edges in $\gr(F)$
\emph{facet defining} if $\gr(F)-C$ is the face graph of a facet of $F$.  Brualdi and
Gibson~\cite[p.~204f]{BG77-1} show that a facet defining set $C$ is a (usually not perfect) matching
in $\gr(F)$ and that any two different facet defining sets are disjoint. This leads to the following
characterization of face subgraphs of facets.
\begin{theorem}[Brualdi and Gibson~{\cite[Cor.~2.11]{BG77-1}}]\label{prop:BG:facets}
  Let $G$ be a connected face graph and $S$ a face subgraph of $G$.  $\face(S)$ is a facet of
  $\face(G)$ if and only if either
  \begin{enumerate}
  \item $S$ is connected and differs from $G$ by a single edge, or
  \item $S$ splits into disjoint face graphs $S_1, \ldots, S_k$ such that there are nodes $u_i,
    v_i\in S_i$ inducing a decomposition of $G$ as
    \begin{align*}
      G\ & =\ S_1 \cup \ldots S_k\cup \{(u_1,v_2), (u_2,v_3), \ldots,
      (u_k,v_1)\}\,.
    \end{align*}
  \end{enumerate}
\end{theorem}
Note that in the second case $u_i$ and $v_i$ are necessarily on different layers of the graph. See
\autoref{fig:twotypesoffacets} for an illustration of the two types.
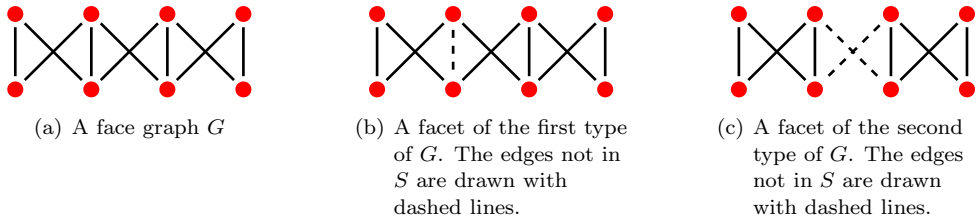
\begin{figure}[t]
  \centering
  \subfigure[A face graph $G$]{
    \label{fig:faceofG}
    \begin{tikzpicture}
      \foreach \x in {0,1,2,3} {
        \coordinate (u\x) at (\x,1);
        \coordinate (v\x) at (\x,0);
      }
 
      \coordinate (u-1) at (-.5,1);
       
      \draw [line width=1pt] (u0) -- (v0);
      \draw [line width=1pt] (u0) -- (v1);
      \draw [line width=1pt] (u1) -- (v1);
      \draw [line width=1pt] (u1) -- (v0);
      \draw [line width=1pt] (u1) -- (v2);
      \draw [line width=1pt] (u2) -- (v1);
      \draw [line width=1pt] (u2) -- (v2);
      \draw [line width=1pt] (u2) -- (v3);
      \draw [line width=1pt] (u3) -- (v2);
      \draw [line width=1pt] (u3) -- (v3);
      \foreach \x in {0,1,2,3} {
        \draw [color=white,fill=red,line width=2pt] (u\x) circle (4pt);
        \draw [color=white,fill=red,line width=2pt] (v\x) circle (4pt);
      }
    \end{tikzpicture}
  }
  \qquad  \quad
  \subfigure[A facet of the first type of $G$.  The edges not in $S$ are drawn with dashed lines.]{
    \label{fig:facetoffirsttype}
    \begin{tikzpicture}
      \foreach \x in {0,1,2,3} {
        \coordinate (u\x) at (\x,1);
        \coordinate (v\x) at (\x,0);
      }
 
      \coordinate (u-1) at (-1.3,1);
      \coordinate (u4) at (4.3,1);

      \draw [line width=1pt] (u0) -- (v0);
      \draw [line width=1pt] (u0) -- (v1);
      \draw [line width=1pt] (u1) -- (v0);
      \draw [line width=1pt,dashed] (u1) -- (v1);
      \draw [line width=1pt] (u1) -- (v2);
      \draw [line width=1pt] (u2) -- (v1);
      \draw [line width=1pt] (u2) -- (v2);
      \draw [line width=1pt] (u2) -- (v3);
      \draw [line width=1pt] (u3) -- (v2);
      \draw [line width=1pt] (u3) -- (v3);
      \foreach \x in {0,1,2,3} {
        \draw [color=white,fill=red,line width=2pt] (u\x) circle (4pt);
        \draw [color=white,fill=red,line width=2pt] (v\x) circle (4pt);
      }
    \end{tikzpicture}
  }
  \qquad  \quad
  \subfigure[A facet of the second type of $G$. The edges not in $S$ are drawn with dashed lines.]{
    \label{fig:facetofseconttype}
    \begin{tikzpicture}
      \foreach \x in {0,1,2,3} {
        \coordinate (u\x) at (\x,1);
        \coordinate (v\x) at (\x,0);
      }
 
      \coordinate (u-1) at (-1.3,1);
      \coordinate (u4) at (4.3,1);
       
      \draw [line width=1pt] (u0) -- (v0);
      \draw [line width=1pt] (u0) -- (v1);
      \draw [line width=1pt] (u1) -- (v1);
      \draw [line width=1pt] (u1) -- (v0);
      \draw [line width=1pt,dashed] (u1) -- (v2);
      \draw [line width=1pt,dashed] (u2) -- (v1);
      \draw [line width=1pt] (u2) -- (v2);
      \draw [line width=1pt] (u2) -- (v3);
      \draw [line width=1pt] (u3) -- (v2);
      \draw [line width=1pt] (u3) -- (v3);
      \foreach \x in {0,1,2,3} {
        \draw [color=white,fill=red,line width=2pt] (u\x) circle (4pt);
        \draw [color=white,fill=red,line width=2pt] (v\x) circle (4pt);
      }
    \end{tikzpicture}
  }
  \caption{A face graph with two types of facets.}
  \label{fig:twotypesoffacets}
\end{figure}

It follows from a Theorem of Hartfiel~\cite[Theorem $\star$]{Hartfiel70} that any connected face
graph $G$ with $\dim(\face(G))\ge 2$ is reducible if all facet defining sets in $G$ have two or more
edges. Geometrically, this implies the following lemma.
\begin{lemma}\label{cor:non_prod_faces}
  Let $F$ be a face of $\birkhoff$. If $F$ is not a product, then $F$ has a facet that is not a
  product.
\end{lemma}
\begin{proof}
  Let $G$ be a connected irreducible graph representing $F$. By the above, $G$ has a facet defining
  set $C$ of size one. Hence $H:=G-C$ is connected, as $G$ is $2$-connected, and the facet defined
  by $H$ is not a product.
\end{proof}

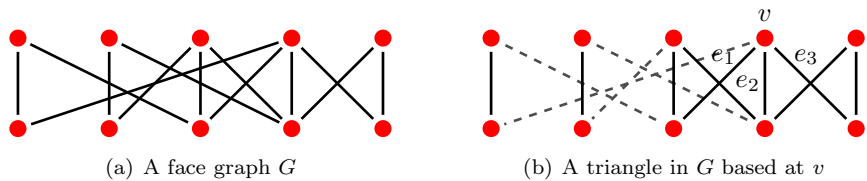
\begin{figure}[b]
  \centering
  \subfigure[A face graph $G$]{
    \label{fig:trianglegraphG}
    \begin{tikzpicture}[scale=1.2]
      \foreach \x in {0,1,2,3,4} {
        \coordinate (u\x) at (\x,1);
        \coordinate (v\x) at (\x,0);
      }
 
      \coordinate (u-1) at (-.5,1);
       
      \draw [line width=1pt] (u0) -- (v0);
      \draw [line width=1pt] (u0) -- (v2);
      \draw [line width=1pt] (u1) -- (v1);
      \draw [line width=1pt] (u1) -- (v3);
      \draw [line width=1pt] (u2) -- (v1);
      \draw [line width=1pt] (u2) -- (v2);
      \draw [line width=1pt] (u2) -- (v3);
      \draw [line width=1pt] (u3) -- (v0);
      \draw [line width=1pt] (u3) -- (v2);
      \draw [line width=1pt] (u3) -- (v3);
      \draw [line width=1pt] (u3) -- (v4);
      \draw [line width=1pt] (u4) -- (v3);
      \draw [line width=1pt] (u4) -- (v4);
      \foreach \x in {0,1,2,3,4} {
        \draw [color=white,fill=red,line width=2pt] (u\x) circle (3.5pt);
        \draw [color=white,fill=red,line width=2pt] (v\x) circle (3.5pt);
      }
    \end{tikzpicture}
  }
  \qquad
  \subfigure[A triangle in $G$ based at $v$]{
    \label{fig:triangleinG}
    \begin{tikzpicture}[scale=1.2]
      \foreach \x in {0,1,2,3,4} {
        \coordinate (u\x) at (\x,1);
        \coordinate (v\x) at (\x,0);
      }
 
      \coordinate (u-1) at (-.5,1);
       
      \draw [line width=1pt] (u0) -- (v0);
      \draw [line width=1pt,dashed,black!70] (u0) -- (v2);
      \draw [line width=1pt,dashed,black!70] (u1) -- (v3);
      \draw [line width=1pt,dashed,black!70] (u2) -- (v1);
      \draw [line width=1pt,dashed,black!70] (u3) -- (v0);
      \draw [line width=1pt] (u1) -- (v1);
      \draw [line width=1pt] (u2) -- (v2);
      \draw [line width=1pt] (u2) -- (v3);
      \draw [line width=1pt] (u3) -- (v2) node[midway,left,xshift=10pt,yshift=10pt] {$e_1$};
      \draw [line width=1pt] (u3) -- (v3) node[midway,left,xshift=2pt] {$e_2$};
      \draw [line width=1pt] (u3) -- (v4) node[midway,right,xshift=-10pt,yshift=10pt] {$e_3$};
      \draw [line width=1pt] (u4) -- (v3);
      \draw [line width=1pt] (u4) -- (v4);

      \foreach \x in {0,1,2,3,4} {
        \draw [color=white,fill=red,line width=2pt] (u\x) circle (3.5pt);
        \draw [color=white,fill=red,line width=2pt] (v\x) circle (3.5pt);
      }

      \draw [color=white,fill=red,line width=2pt] (u3) circle (3.5pt) node[above,black,yshift=2pt] {$v$};
    \end{tikzpicture}
  }
  \caption{A face graph and a triangle in that graph.}
  \label{fig:triangles}
\end{figure}
Brualdi and Gibson in their papers also obtained some results about edges and $2$-dimensional faces
of $\birkhoff$.
\begin{lemma}[Brualdi and Gibson~{\cite[Lemma 3.3 and Lemma 4.2]{BG77-2}}]
  \label{lemma:BG:egdes}
  Let $G$ be a connected face graph and $e_1, e_2, e_3$ edges with a common node $v$.
  \begin{enumerate}
  \item There are perfect matchings $M_1$ and $M_2$ each containing one of the edges such that
    $\face(M_1\cup M_2)$ is an edge in some $\birkhoff$.
  \item If there are perfect matchings $M_1, M_2$ with $e_i\in M_i$, $i=1,2$, such that
    $\face(M_1\cup M_2)$ is an edge, then there is a perfect matching $M_3$ containing $e_3$ such
    that $\face(M_1\cup M_2\cup M_3)$ is a triangle in some $\birkhoff$.
  \end{enumerate}
\end{lemma}
See \autoref{fig:triangles} for an example. More generally, the union of any two perfect matchings
in a face graph is the disjoint union of single edges and cycles. Hence, the minimal face containing
a given pair of vertices is always a cube of some dimension.

Note, that \autoref{lemma:BG:egdes} ensures for any two edges sharing a node the existence of two
perfect matchings containing them that form an edge. Hence, any three edges with a common node
define at least one triangle in the polytope. This implies that the only triangle free faces of
$\birkhoff$ are cubes~\cite[Thm.~4.3]{BG77-3}.
  
In fact, any vertex of a face $F$ of $\birkhoff$ is incident to at least one triangle, unless $F$ is
a cube. Theorem 4.4 of \cite{BG77-2} furthermore tells us that the induced graph of the neighborhood
of any vertex in the polytope graph has $k$ components if and only if the polytope is a $k$-fold
product. Note that one direction is trivial. If the face is a product, then the union of the perfect
matchings of all neighbors of a vertex is already the graph of the face.

Brualdi and Gibson\cite[Thm.~3.3]{BG77-3} showed that a $d$-dimensional face $F:=\face(G)$
corresponding to an irreducible face graph $G$ has at most $3(d-1)$ facets, which is linear in
$d$. Further, if $F$ has exactly $3(d-1)$ facets, then $G$ is a $3$-regular bipartite graph on $d-1$
vertices. Conversely, it is, however, not true that any graph on $d-1$ nodes with constant degree
$3$ defines a face with $3(d-1)$ facets. See~\autoref{fig:notall3dm1} for an example.
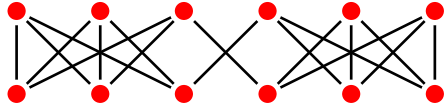
\begin{figure}[t]
  \centering
  \begin{tikzpicture}[scale=1.1]
    \foreach \x in {0,1,2,3,4,5} {
      \coordinate (u\x) at (\x,1);
      \coordinate (v\x) at (\x,0);
    }
    
    \draw [line width=1pt] (u0) -- (v0);
    \draw [line width=1pt] (u0) -- (v1);
    \draw [line width=1pt] (u0) -- (v2);
    \draw [line width=1pt] (u1) -- (v0);
    \draw [line width=1pt] (u1) -- (v1);
    \draw [line width=1pt] (u1) -- (v2);
    \draw [line width=1pt] (u2) -- (v0);
    \draw [line width=1pt] (u2) -- (v1);
    \draw [line width=1pt] (u2) -- (v3);
    \draw [line width=1pt] (u3) -- (v2);
    \draw [line width=1pt] (u3) -- (v4);
    \draw [line width=1pt] (u3) -- (v5);
    \draw [line width=1pt] (u4) -- (v3);
    \draw [line width=1pt] (u4) -- (v4);
    \draw [line width=1pt] (u4) -- (v5);
    \draw [line width=1pt] (u5) -- (v3);
    \draw [line width=1pt] (u5) -- (v4);
    \draw [line width=1pt] (u5) -- (v5);
    \foreach \x in {0,1,2,3,4,5} {
      \draw [color=white,fill=red,line width=2pt] (u\x) circle (4pt);
      \draw [color=white,fill=red,line width=2pt] (v\x) circle (4pt);
    }
  \end{tikzpicture}
  \caption{A $3$-regular graph for a $7$-dimensional face with $17\, <\, 3(7-1)$ facets}
  \label{fig:notall3dm1}
\end{figure}

\section{Face Graphs with Many Nodes}
\label{sec:manynodes}

Let $\lattice$ be the combinatorial type of a face of a Birkhoff polytope. The \emph{Birkhoff
  dimension} $\bdim(\lattice)$ of $\lattice$ is the smallest $n$ such that $\lattice$ is the
combinatorial type of some face of $\birkhoff$. It follows from \autoref{thm:billera} that
$\bdim(\lattice)\le 2d$ for a combinatorial type $\lattice$ of a $d$-dimensional face.  In this
section, we will study some properties of combinatorial types $\lattice$ of $d$-dimensional faces
with $\bdim(\lattice)\ge d$. In particular we will completely characterize those with
$\bdim(\lattice)\ge 2d-3$.

\subsection{Wedges}
\label{subsec:wedges}

In this section we will show that most faces of $\birkhoff$ are wedges over lower dimensional
faces. The following main theorem characterizes graphs that correspond to wedges.
\begin{theorem}
  Let $G$ be a face graph with $n\ge 3$ nodes in the upper layer and two connected adjacent nodes
  $u$ and $v$ of degree $2$. Let $G'$ be the graph obtained by attaching a path of length $3$ to $u$
  and $v$.  Then $G$ is a face graph and the associated face is a wedge over the face of $G$.
\end{theorem}
\begin{proof}
  $G'$ is clearly a face graph.  We prove the theorem by induction over the dimension. The claim is
  true if $G$ is the unique reduced graph on four nodes and four edges.

  In the following we assume that the claim is true for face graphs defining a $(d-1)$-dimensional
  face of $\birkhoff$.

  Let $F$ be the face of $G$ and $F'$ that of $G'$.  Let $u'$ and $v'$ be the two nodes added in
  $G'$ and $e_1=(u,v)$, $e_2=(u',v')$, $f_1=(u,v')$, and $f_2=(u',v)$. See also
  \autoref{fig:genericwegde}. Let $G_1$ be the face graph of all perfect matchings in $G$ that do
  not contain $e_1$ (\emph{i.e.}, the union of the perfect matchings in $M_{\neg e_1}$), and $R$ the
  associated face of $F$ (see~\autoref{fig:wegde}\subref{fig:bf:wegde:G} and
  \autoref{fig:wegde}\subref{fig:bf:wegde:Gone}). As $G$ has at least three nodes in each layer,
  this is a nonempty face. We claim that $F'=\wedge_R(F)$.

  To show that $F'$ is a wedge over $F$ we have to show that $F'$ has two facets $F_1, F_2$
  isomorphic to $F$ that meet in a face isomorphic to $R$, such that any other facet of $F'$ is
  either
  \begin{enumerate}
  \item a prism over a facet of $F$, or
  \item a wedge of a facet of $F$ at a face of $R$,
  \end{enumerate}
  and, conversely, any facet of $F$ (except possibly $R$) is used in one of these two cases.

  Let $G'_1$ be the graph obtained from $G$ by adding $u'$, $v'$ together with $e_2$, and $G'_2$ the
  graph obtained by removing $e_1$ and adding $u'$, $v'$ together with $e_2$, $f_1$ and $f_2$. See
  \autoref{fig:wegde}\subref{fig:bf:wegde:Gpone} and
  \autoref{fig:wegde}\subref{fig:bf:wegde:Gptwo}. Both graphs are subgraphs of $G'$ and clearly
  define facets combinatorially isomorphic to $F$ that intersect in $R$.

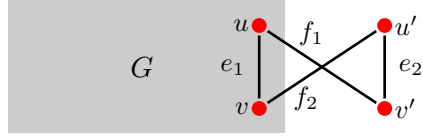
\begin{figure}[t]
  \centering
  \begin{tikzpicture}[scale=1.1]
    \coordinate (g111) at (0,-.3);
    \coordinate (g112) at (3.3,-.3);
    \coordinate (g121) at (0,1.3);
    \coordinate (g122) at (3.3,1.3);

    \coordinate (v) at (3,0);
    \coordinate (u) at (3,1);
    \coordinate (vp) at (4.5,0);
    \coordinate (up) at (4.5,1);

    \coordinate (c) at (1.6,.5);

    \coordinate (g111b) at (0,-.3);
    \coordinate (g112b) at (2,-.3);

    \coordinate (g1) at (1,-1);
    \coordinate (g2) at (4.5,-1);
    \coordinate (g3) at (8,-1);

    \fill[color=black!20] (g111) -- (g112) -- (g122) -- (g121) -- cycle;

    \node at (c) {\large$G$};

    \draw[black,line width=1pt] (u) -- (vp)  node [black,midway,left=6pt,below=3pt,line width=2pt] {$f_2$};
    \draw[black,line width=1pt] (v) -- (up)  node [black,midway,left=4pt,above=4pt,line width=2pt] {$f_1$};
    \draw[black,line width=1pt] (vp) -- (up) node [black,midway,right=1pt,line width=2pt] {$e_2$};
    \draw[black,line width=1pt] (v) -- (u)   node [black,midway,left=1pt,line width=2pt] {$e_1$};

    \draw[color=black!20,fill=red,line width=1pt] (u) circle (3pt) node[black,left] {$u$};
    \draw[color=black!20,fill=red,line width=1pt] (v) circle (3pt) node[black,left] {$v$};
    \draw[color=white,fill=red,line width=1pt] (up) circle (3pt) node[black,right] {$u'$};
    \draw[color=white,fill=red,line width=1pt] (vp) circle (3pt) node[black,right] {$v'$};
  \end{tikzpicture}
  \caption{A wedge over a face graph $G$.}
  \label{fig:genericwegde}
\end{figure}
\begin{figure}[b]
  \centering
  \subfigure[A face graph $G$    \label{fig:bf:wegde:G}]{
    \begin{tikzpicture}
      \foreach \x in {0,1,2,3} {
        \coordinate (u\x) at (\x,1);
        \coordinate (v\x) at (\x,0);
      }
      
      \coordinate (u-1) at (-1.3,1);
      \coordinate (u4) at (4.3,1);
      
      \draw [line width=1pt] (u0) -- (v0);
      \draw [line width=1pt] (u0) -- (v1);
      \draw [line width=1pt] (u1) -- (v0);
      \draw [line width=1pt] (u1) -- (v1);
      \draw [line width=1pt] (u1) -- (v2);
      \draw [line width=1pt] (u2) -- (v1);
      \draw [line width=1pt] (u2) -- (v2);
      \draw [line width=1pt] (u2) -- (v3);
      \draw [line width=1pt] (u3) -- (v0);
      \draw [line width=1pt] (u3) -- (v1);
      \draw [line width=1pt] (u3) -- (v2);
      \draw [line width=1pt] (u3) -- (v3);
      \draw[black,line width=1pt] (v3) -- (u3)   node [black,midway,right=1pt,line width=2pt] {$e_1$};
      
      \foreach \x in {0,1,2,3} {
        \draw [color=white,fill=red,line width=2pt] (u\x) circle (4pt);
        \draw [color=white,fill=red,line width=2pt] (v\x) circle (4pt);
      }
      \draw [color=white,line width=2pt] (u-1) circle (4pt);
      \draw [color=white,line width=2pt] (u4) circle (4pt);
      \draw[color=white,fill=red,line width=1pt] (u3) circle (3pt) node[black,right=1pt,above] {$u\phantom{'}$};
      \draw[color=white,fill=red,line width=1pt] (v3) circle (3pt) node[black,right=1pt,below] {$v\phantom{'}$};
    \end{tikzpicture}
  }
  \subfigure[The face graph $G_1$\label{fig:bf:wegde:Gone}]{
    \begin{tikzpicture}
      \foreach \x in {0,1,2,3} {
        \coordinate (u\x) at (\x,1);
        \coordinate (v\x) at (\x,0);
      }
      
      \coordinate (u-1) at (-1.3,1);
      \coordinate (u4) at (4.3,1);
      
      \draw [line width=1pt] (u0) -- (v0);
      \draw [line width=1pt] (u0) -- (v1);
      \draw [line width=1pt] (u1) -- (v0);
      \draw [line width=1pt] (u1) -- (v1);
      \draw [line width=1pt] (u1) -- (v2);
      \draw [line width=1pt] (u2) -- (v1);
      \draw [line width=1pt] (u2) -- (v2);
      \draw [line width=1pt] (u2) -- (v3);
      \draw [line width=1pt] (u3) -- (v0);
      \draw [line width=1pt] (u3) -- (v1);
      \draw [line width=1pt] (u3) -- (v2);
      
      \foreach \x in {0,1,2,3} {
        \draw [color=white,fill=red,line width=2pt] (u\x) circle (4pt);
        \draw [color=white,fill=red,line width=2pt] (v\x) circle (4pt);
      }
      \draw [color=white,line width=2pt] (u-1) circle (4pt);
      \draw [color=white,line width=2pt] (u4) circle (4pt);
      \draw[color=white,fill=red,line width=1pt] (u3) circle (3pt) node[black,right=1pt,above] {$u\phantom{'}$};
      \draw[color=white,fill=red,line width=1pt] (v3) circle (3pt) node[black,right=1pt,below] {$v\phantom{'}$};
    \end{tikzpicture}
  }
  \subfigure[The face graph $G'_1$\label{fig:bf:wegde:Gpone}]{
    \begin{tikzpicture}
      \foreach \x in {0,1,2,3,4} {
        \coordinate (u\x) at (\x,1);
        \coordinate (v\x) at (\x,0);
      }
      
      \coordinate (u-1) at (-1.3,1);
      \coordinate (u5) at (5.3,1);
      
      \draw [line width=1pt] (u0) -- (v0);
      \draw [line width=1pt] (u0) -- (v1);
      \draw [line width=1pt] (u1) -- (v0);
      \draw [line width=1pt] (u1) -- (v1);
      \draw [line width=1pt] (u1) -- (v2);
      \draw [line width=1pt] (u2) -- (v1);
      \draw [line width=1pt] (u2) -- (v2);
      \draw [line width=1pt] (u2) -- (v3);
      \draw [line width=1pt] (u3) -- (v0);
      \draw [line width=1pt] (u3) -- (v1);
      \draw [line width=1pt] (u3) -- (v2);
      \draw[black,line width=1pt] (v3) -- (u3)   node [black,midway,right=1pt,line width=2pt] {$e_1$};
      \draw[black,line width=1pt] (v4) -- (u4)   node [black,midway,right=1pt,line width=2pt] {$e_2$};
      
      \foreach \x in {0,1,2,3,4} {
        \draw [color=white,fill=red,line width=2pt] (u\x) circle (4pt);
        \draw [color=white,fill=red,line width=2pt] (v\x) circle (4pt);
      }
    \draw[color=white,fill=red,line width=1pt] (u3) circle (3pt) node[black,right=1pt,above] {$u\phantom{'}$};
    \draw[color=white,fill=red,line width=1pt] (v3) circle (3pt) node[black,right=1pt,below] {$v\phantom{'}$};
    \draw[color=white,fill=red,line width=1pt] (u4) circle (3pt) node[black,right=1pt,above] {$u'$};
    \draw[color=white,fill=red,line width=1pt] (v4) circle (3pt) node[black,right=1pt,below] {$v'$};
    \end{tikzpicture}
  }
  \qquad
  \subfigure[The face graph $G'_2$\label{fig:bf:wegde:Gptwo}]{
    \begin{tikzpicture}
      \foreach \x in {0,1,2,3,4} {
        \coordinate (u\x) at (\x,1);
        \coordinate (v\x) at (\x,0);
      }
      
      \coordinate (u-1) at (-1.3,1);
      \coordinate (u5) at (5.3,1);
      
      \draw [line width=1pt] (u0) -- (v0);
      \draw [line width=1pt] (u0) -- (v1);
      \draw [line width=1pt] (u1) -- (v0);
      \draw [line width=1pt] (u1) -- (v1);
      \draw [line width=1pt] (u1) -- (v2);
      \draw [line width=1pt] (u2) -- (v1);
      \draw [line width=1pt] (u2) -- (v2);
      \draw [line width=1pt] (u2) -- (v3);
      \draw [line width=1pt] (u3) -- (v0);
      \draw [line width=1pt] (u3) -- (v1);
      \draw [line width=1pt] (u3) -- (v2);
      \draw[black,line width=1pt] (u3) -- (v4)   node [black,midway,left=1pt,above=3pt,line width=2pt] {$f_1$};
      \draw[black,line width=1pt] (u4) -- (v3)   node [black,midway,right=2pt,below=2pt,line width=2pt] {$f_2$};
      \draw[black,line width=1pt] (v4) -- (u4)   node [black,midway,right=1pt,line width=2pt] {$e_2$};
      
      \foreach \x in {0,1,2,3,4} {
        \draw [color=white,fill=red,line width=2pt] (u\x) circle (4pt);
        \draw [color=white,fill=red,line width=2pt] (v\x) circle (4pt);
      }
    \draw[color=white,fill=red,line width=1pt] (u3) circle (3pt) node[black,right=1pt,above] {$u\phantom{'}$};
    \draw[color=white,fill=red,line width=1pt] (v3) circle (3pt) node[black,right=1pt,below] {$v\phantom{'}$};
    \draw[color=white,fill=red,line width=1pt] (u4) circle (3pt) node[black,right=1pt,above] {$u'$};
    \draw[color=white,fill=red,line width=1pt] (v4) circle (3pt) node[black,right=1pt,below] {$v'$};
    \end{tikzpicture}
  }
  \caption{The various face graphs involved in the wedge construction}
  \label{fig:wegde}
\end{figure}
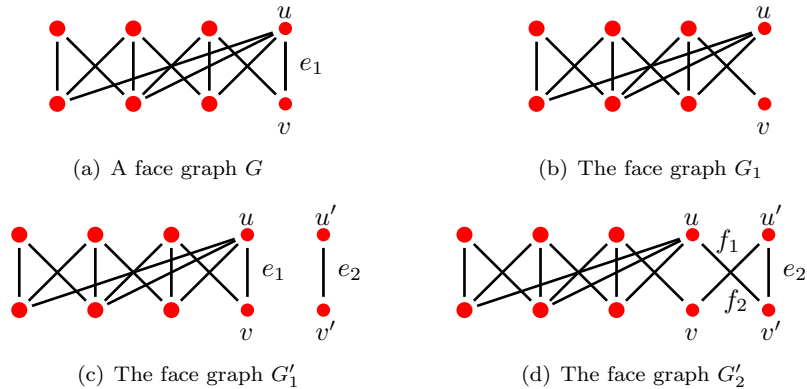
  Let $J'$ be a subgraph of $G'$ that defines a facet of $F'$. There are three possibilities for
  edges contained in $G'$ but not in $J'$:
  \begin{enumerate}
  \item both $f_1$ and $f_2$ are missing, or
  \item $e_1$ is missing, or
  \item $e_1$, $e_2$, $f_1$, and $f_2$ are present and some other edges are missing.
  \end{enumerate}
  The first two cases are the two copies of $F$. In the last case, if $u, u', v$ and $v'$ form a
  connected component of $J$, then we have a prism over a facet of $F$ sharing no vertex with $R$,
  and a wedge over a face of $R$ by induction otherwise.

  Conversely, let $K$ be the graph of a facet $S$ of $F$. If $e_1$ is missing in $K$, then $R$ is a
  facet and corresponds to $K$. So we assume that $e_1$ is present in $K$. If one of $u$, $v$ is
  connected to some other node in $K$ (and, thus, necessarily also the other), then the wedge of $K$
  over the intersection of $S$ and $R$ is contained in $G'$. If $u, v$ form a connected component,
  then $S$ and $R$ are disjoint and the prism over $S$ is contained in $F'$.
\end{proof}
This gives a first classification of combinatorial types with Birkhoff dimension at least $d$.
\begin{corollary}\label{thm:Wedges}
  Let $\lattice$ be the combinatorial type of a $d$-dimensional face of some $\birkhoff$. If
  $\bdim(\lattice)\ge d$, then $\lattice$ is a wedge or a product.
\end{corollary}
\begin{proof}
  Let $F$ be a face of some $\birkhoff$ with combinatorial type $\lattice$, and assume that $F$ is
  not a product. Let $G$ be a reduced graph representing $F$. Thus, $G$ is connected. Then
  $d=m-2n+1$ implies that each layer of $G$ has at least one node of degree $2$. By first completely
  reducing the graph $G$ and subsequently resolving any multiple edges (similar to the proof of
  \autoref{thm:billera}), we can assume that there are two adjacent nodes of degree $2$. Now we can
  use the previous theorem.
\end{proof}
In fact, every $d$-dimensional combinatorial type $\lattice$ with $\bdim(\lattice)\ge n$ is a wedge
if its graph has a component with at least three nodes in each layer, \emph{i.e.}, if the face is
not a cube.
\begin{theorem}\label{thm:StandardWedge}
  Let $F$ be a face of $\birkhoff$. Then any wedge of $F$ over a facet or the complement of a facet
  is also a face of some $\birkhoff[m]$, $m\ge n$.
\end{theorem}
\begin{proof}
  Let $G$ be an irreducible face graph corresponding to $F$.  Let $E$ be a facet of $F$ with facet
  defining set $C$. Let $e$ be any edge in $C$. The wedge over $F$ is obtained by adding a path of
  length $3$ to the endpoints of $e$, and the wedge over the complement is obtained by first
  replacing $e$ by a path of length $3$ and then adding another path of length $3$ to the two new
  nodes. See \autoref{fig:corwegde} for an illustration of the two operations.
\end{proof}
\begin{figure}[t]
  \centering
  \subfigure[A face graph $G$ with a facet defining set $C$ of edges drawn dashed]{
    \begin{tikzpicture}
      \foreach \x in {0,1,2,3} {
        \coordinate (u\x) at (\x,1);
        \coordinate (v\x) at (\x,0);
      }
      
      \coordinate (u-1) at (-2.3,1);
      \coordinate (u4) at (6.3,1);
      
      \draw [line width=1pt] (u0) -- (v0);
      \draw [line width=1pt] (u0) -- (v1);
      \draw [line width=1pt] (u1) -- (v0);
      \draw [line width=1pt] (u1) -- (v1);
      \draw [line width=2pt,dashed] (u1) -- (v2);
      \draw [line width=2pt,dashed] (u2) -- (v1);
      \draw [line width=1pt] (u2) -- (v2);
      \draw [line width=1pt] (u2) -- (v3);
      \draw [line width=1pt] (u3) -- (v2);
      \draw [line width=1pt] (u3) -- (v3);
      
      \foreach \x in {0,1,2,3} {
        \draw [color=white,fill=red,line width=2pt] (u\x) circle (4pt);
        \draw [color=white,fill=red,line width=2pt] (v\x) circle (4pt);
      }
      \draw [color=white,line width=2pt] (u-1) circle (4pt);
      \draw [color=white,line width=2pt] (u4) circle (4pt);
    \end{tikzpicture}
  }

  \hspace*{-.5cm}
  \subfigure[The wedge over the facet defined by $C$]{
    \begin{tikzpicture}
      \foreach \x in {0,1,2,3,4} {
        \coordinate (u\x) at (\x,1);
        \coordinate (v\x) at (\x,0);
      }
      
      \coordinate (u-1) at (-1.3,1);
      \coordinate (u5) at (5.3,1);
      
      \draw [line width=1pt] (u0) -- (v0);
      \draw [line width=1pt] (u0) -- (v1);
      \draw [line width=1pt] (u1) -- (v0);
      \draw [line width=1pt] (u1) -- (v1);
      \draw [line width=1pt] (u1) -- (v2);
      \draw [line width=1pt] (u2) -- (v1);
      \draw [line width=1pt] (u2) -- (v2);
      \draw [line width=1pt] (u2) -- (v3);
      \draw [line width=1pt] (u3) -- (v2);
      \draw [line width=1pt] (u3) -- (v3);
      \draw [line width=2pt,dashed] (u2) -- (v4);
      \draw [line width=2pt,dashed] (u4) -- (v4);
      \draw [line width=2pt,dashed] (u4) -- (v1);
      
      \foreach \x in {0,1,2,3,4} {
        \draw [color=white,fill=red,line width=2pt] (u\x) circle (4pt);
        \draw [color=white,fill=red,line width=2pt] (v\x) circle (4pt);
      }
      \draw [color=white,line width=2pt] (u-1) circle (4pt);
      \draw [color=white,line width=2pt] (u5) circle (4pt);
    \end{tikzpicture}
  }
  \subfigure[The wedge over the complement of the facet defined by $C$]{
    \begin{tikzpicture}
      \foreach \x in {0,1,2,3,5} {
        \coordinate (u\x) at (\x,1);
        \coordinate (v\x) at (\x,0);
      }
      
      \coordinate (u-1) at (-1.3,1);
      \coordinate (u6) at (6.3,1);
      
      \draw [line width=1pt] (u0) -- (v0);
      \draw [line width=1pt] (u0) -- (v1);
      \draw [line width=1pt] (u1) -- (v0);
      \draw [line width=1pt] (u1) -- (v1);
      \draw [line width=1pt] (u1) -- (v2);
      \draw [line width=1pt] (u2) -- (v2);
      \draw [line width=1pt] (u2) -- (v3);
      \draw [line width=1pt] (u3) -- (v2);
      \draw [line width=1pt] (u3) -- (v3);

      \draw [line width=1pt] (u2) -- (v4);
      \draw [line width=1pt] (u4) -- (v4);
      \draw [line width=1pt] (u4) -- (v1);

      \draw [line width=2pt,dashed] (u4) -- (v5);
      \draw [line width=2pt,dashed] (u5) -- (v5);
      \draw [line width=2pt,dashed] (u5) -- (v4);
      
      \foreach \x in {0,1,2,3,4,5} {
        \draw [color=white,fill=red,line width=2pt] (u\x) circle (4pt);
        \draw [color=white,fill=red,line width=2pt] (v\x) circle (4pt);
      }
    \end{tikzpicture}
  }
  \caption{Wedge over a facet and its complement}
  \label{fig:corwegde}
\end{figure}
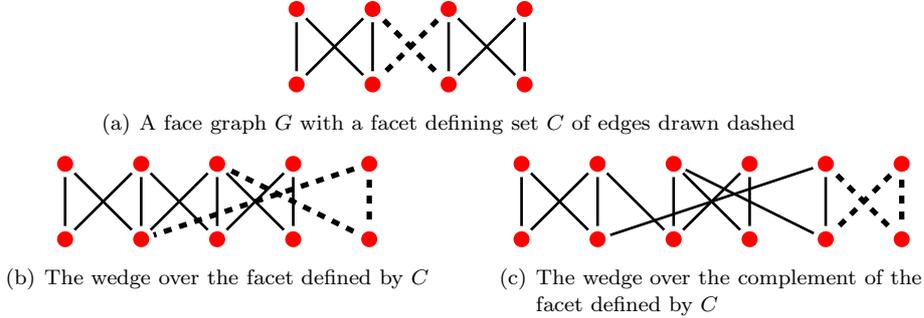

\subsection{Pyramids}

We want to discuss the structure of face graphs that correspond to pyramids. This will be important
for the classification of faces with large Birkhoff dimension. In particular, we will see that for
many faces $F$ of $\birkhoff$ the pyramid over $F$ is again face of a Birkhoff polytope.
\begin{lemma}
  Let $G$ be a connected face graph containing an edge $e$ that appears only in a single perfect
  matching $M$. Then $M$ defines an edge of $\face(G)$ with any other perfect matching in $G$.
\end{lemma}
\begin{proof}
  Suppose not.  Then there is a perfect matching $M'$ such that $M\cup M'$ contains more than one
  cycle $C_1, \ldots, C_k$. The edge $e$ is contained in such a cycle, as otherwise $M'$ would use
  $e$. Assume this is $C_1$.  However, using the edges of $M'$ in $C_2, \ldots, C_k$ and the edges
  of $M$ in $C_1$ defines another perfect matching $M''$ using $e$ and different from both $M$ and
  $M'$.  This is a contradiction to the uniqueness of $e$.
\end{proof}
\begin{theorem}
  Let $G$ be a connected irreducible face graph. Then $\face(G)$ is a pyramid if and only if $G$ has
  an edge $e$ that is contained in a unique perfect matching $M$.
\end{theorem}
See \autoref{fig:edgeunique} for an example.
\begin{figure}[t]
  \centering
  \begin{tikzpicture}[scale=1.1]
    \foreach \x in {-1,0,1,2,3,4,5,6} {
      \coordinate (u\x) at (\x,1);
      \coordinate (v\x) at (\x,0);
    }

    \draw [line width=1pt] (u0) -- (v0);
    \draw [line width=1pt] (u0) -- (v1);
    \draw [line width=1pt] (u1) -- (v0);
    \draw [line width=1pt] (u1) -- (v1);
    \draw [line width=1pt] (u2) -- (v2);
    \draw [line width=1pt] (u2) -- (v3);
    \draw [line width=1pt] (u3) -- (v2);
    \draw [line width=1pt] (u3) -- (v3);
    \draw [line width=1pt] (u4) -- (v4);
    \draw [line width=1pt] (u4) -- (v5);
    \draw [line width=1pt] (u5) -- (v4);
    \draw [line width=1pt] (u5) -- (v5);

    \draw [line width=1pt] (u1) -- (v2);
    \draw [line width=1pt] (u3) -- (v4);
    \draw [line width=2pt] (u5) -- (v0);

    \foreach \x in {0,1,2,3,4,5} {
      \draw [color=white,fill=red,line width=2pt] (u\x) circle (4pt);
      \draw [color=white,fill=red,line width=2pt] (v\x) circle (4pt);
    }
    \draw [color=white,line width=2pt] (u-1) circle (4pt);
    \draw [color=white,line width=2pt] (u6) circle (4pt);
  \end{tikzpicture}
  \caption{The thick long edge is contained in only one perfect
  matching in the graph. The graph defines a pyramid over a $3$-cube.}
  \label{fig:edgeunique}
\end{figure}
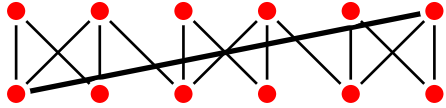
\begin{proof}
  If $G$ has such an edge, then the union of all perfect matchings in $G$ except $M$ defines a
  proper face $R$ of $\face(G)$ containing all but one vertex. Hence, $\face(G)$ must be a pyramid
  over $S$ with apex $M$.

  If $\face(G)$ is a pyramid, then let $M$ be the perfect matching corresponding to the apex. Assume
  by contradiction that any edge $e$ of $M$ is contained in some other perfect matching $M_e$
  different from $M$. Let $H$ be the subgraph defined by the union of these perfect matchings. Then
  $\face(H)$ is the smallest face $S$ of $\face(G)$ containing those vertices. But $H$ contains $M$,
  so $S$ contains the apex. This is a contradiction, as then $S$ is a pyramid with apex $M$ whose
  base already contains all vertices corresponding to the $M_e$.
\end{proof}
Let $G$ be a face graph with connected components $G_0, \ldots, G_{k-1}$. We define the
\emph{circular connection} $\circular(G)$ of $G$ in the following way.  For each $1\le i\le k$, let
$u^i$ be a node in the upper and $v^i$ a node in the lower layer of $G_i$. Then the nodes of
$\circular(G)$ are the disjoint union of the nodes of $G_i$, and the edges of $\circular(G)$ are
those of $G_i$ together with edges from $u^{i+1}$ to $v^i$ for $0\le i\le k-1$ (with indices taken
modulo $k$). See \autoref{fig:circularconnection} for an illustration. Note that the circular
connection is in general not a face graph. It is, if $G_i-\{u^i,v^i\}$ has a perfect matching for
all $i$. If the perfect matchings in $G_i-\{u^i,v^i\}$ are also unique, then the circular connection
is a face graph whose associated face is the pyramid over the face of $G$.  This motivates the
following definition.
\begin{definition}\label{def:pyramidal}
  Let $G$ be a face graph with connected components $G_0, \ldots, G_{k-1}$.  A choice $S(G):=\{u^0,
  v^0, \ldots, u^{k-1}, v^{k-1}\}$ of nodes $u^i, v^i\in G_i$ with $u^i$ in the upper and $v^i$ in
  the lower layer for $0\le i\le {k-1}$ is \emph{pyramidal} if the graph $G-S(G)$ has a unique
  perfect matching.
\end{definition}
\begin{corollary}
  Let $G$ be a connected irreducible face graph. If $G$ has a node $u\in U$ and $v\in V$ such that
  $(u,v)\not\in E$ and $G-\{u,v\}$ has a unique perfect matching, then $H:=G+\{(u,v)\}$ defines a
  face graph that corresponds to the pyramid over $\face(G)$.
\end{corollary}
\begin{proof}
  $\{(u,v)\}$ is a facet defining set and the facet contains all but one vertex of $H$.
\end{proof}

\begin{corollary}\label{cor:faces:pyramid}
  Let $G$ be a face graph with an edge $e$ contained in a unique perfect matching. Then the pyramid
  over the face of $G$ is again a face of $G$.
\end{corollary}
\begin{proof}
  The apex of a pyramid is the complement of a facet, and by \autoref{thm:StandardWedge} the wedge
  over any complement of a facet exists.
\end{proof}

\subsection{$d$-dimensional combinatorial types with Birkhoff dimension $\bdim(\lattice)\ge 2d-2$.}
\label{subsec:2d-2}

\begin{figure}[tb]
  \centering
  \begin{tikzpicture}[scale=1.4]
    \coordinate (g111) at (0,0);
    \coordinate (g112) at (2,0);
    \coordinate (g121) at (0,1);
    \coordinate (g122) at (2,1);

    \coordinate (g111b) at (0,-.3);
    \coordinate (g112b) at (2,-.3);

    \coordinate (g211) at (3,0);
    \coordinate (g212) at (6,0);
    \coordinate (g221) at (3,1);
    \coordinate (g222) at (6,1);

    \coordinate (g211b) at (3,-.3);
    \coordinate (g212b) at (6,-.3);

    \coordinate (g311) at (7,0);
    \coordinate (g312) at (9,0);
    \coordinate (g321) at (7,1);
    \coordinate (g322) at (9,1);

    \coordinate (g311b) at (7,-.3);
    \coordinate (g312b) at (9,-.3);

    \coordinate (g1) at (1,-1);
    \coordinate (g2) at (4.5,-1);
    \coordinate (g3) at (8,-1);

    \coordinate (u1) at (barycentric cs:g112=1,g121=8,g111=1);%
    \coordinate (v1) at (barycentric cs:g121=1,g112=8,g122=1);%

    \coordinate (u2) at (barycentric cs:g212=1,g211=2,g221=12);%
    \coordinate (v2) at (barycentric cs:g221=1,g212=8,g222=1);%

    \coordinate (u3) at (barycentric cs:g311=1,g321=8,g312=1);%
    \coordinate (v3) at (barycentric cs:g321=1,g312=8,g322=1);%

    \fill[color=black!20] (g111) -- (g112) -- (g122) -- (g121) -- cycle;
    \fill[color=black!20] (g211) -- (g212) -- (g222) -- (g221) -- cycle;
    \fill[color=black!20] (g311) -- (g312) -- (g322) -- (g321) -- cycle;

    \draw [decorate,decoration={brace,amplitude=4pt}] (g112b) -- (g111b) node [black,midway,below=4pt] {$G_1$};
    \draw [decorate,decoration={brace,amplitude=4pt}] (g212b) -- (g211b) node [black,midway,below=4pt] {$G_2$};
    \draw [decorate,decoration={brace,amplitude=4pt}] (g312b) -- (g311b) node [black,midway,below=4pt] {$G_3$};

    \node at (u3) {} edge [line width=1pt,auto] (v2);
    \node at (u2) {} edge [line width=1pt,auto] (v1);
    \node at (u1) {} edge [line width=1pt,auto] (v3);

    \draw[color=black!20,fill=red,line width=1pt] (u1) circle (3pt) node[black,left,xshift=-5pt] {$u^1$};
    \draw[color=black!20,fill=red,line width=1pt] (v1) circle (3pt) node[black,left,xshift=-3pt] {$v^1$};
    \draw[color=black!20,fill=red,line width=1pt] (u2) circle (3pt) node[black,right,xshift=3pt] {$u^2$};
    \draw[color=black!20,fill=red,line width=1pt] (v2) circle (3pt) node[black,left,xshift=-3pt] {$v^2$};
    \draw[color=black!20,fill=red,line width=1pt] (u3) circle (3pt) node[black,right,xshift=3pt] {$u^3$};
    \draw[color=black!20,fill=red,line width=1pt] (v3) circle (3pt) node[black,right,xshift=5pt] {$v^3$};
  \end{tikzpicture}
  \caption{The circular connection of a graph.}
  \label{fig:circularconnection}
\end{figure}
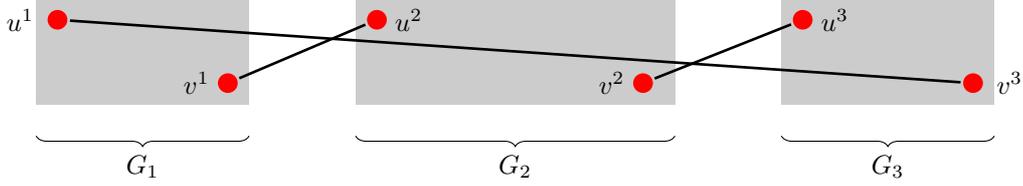
For the remainder of this section we will study combinatorial types of faces with large Birkhoff
dimension. We have seen above that for a given combinatorial type $\lattice$ this is bounded by
$\bdim(\lattice)\le 2d$. The next proposition characterizes the case of equality.
\begin{proposition}\label{prop:cube}
  Let $\lattice$ be a $d$-dimensional combinatorial type of a face of $\birkhoff$ with
  $\bdim(\lattice)=2d$. Then $\lattice$ is a cube.
\end{proposition}
\begin{proof}
  We prove this by induction. If $d=1$, then this follows from \autoref{prop:lowNodeNumbers}.

  By \autoref{cor:irredNodeBound} a connected irreducible graph of a $d$-dimensional face has at
  most $2d-2$ nodes in each layer. Hence, the graph $G$ of $F$ must be disconnected. Let $m$, $n$,
  and $k$ be its number of edges, nodes, and connected components, resp.  Let $G_1$ be a connected
  component of $G$, and $G_2$ the remaining graph. Both are irreducible graphs. Let $G_i$ have $n_i$
  nodes, $m_i$ edges, $k_i$ components and define a face of dimension $d_i$.

  The dimension formula gives $d=m-2n+k=m-4d+k$, so $m=5d-k$.  We argue that $n_1=2$. Suppose
  otherwise. If $n_1\ge 4$, then by \autoref{cor:MinNumberOfEdges} we can estimate
  \begin{align*}
    5d-k\ =\ m\ &\ge \ 2n_1+\left\lceil\frac{n_1}2\right\rceil +2n_2+\left\lceil\frac{n_2}2\right\rceil \ge\
    2n+\frac n2 = 5d\,.
  \end{align*}
  As $k\ge 1$ this is not possible. Now if $n_1=3$, then $n_2=2d-3$, and by
  \autoref{prop:lowNodeNumbers} we know $m_1\ge 7$. So we can compute
  \begin{align*}
    5d-k\ =\ m\ &\ge \ 7 +2n_2+\left\lceil\frac{n_2}2\right\rceil \ge\ 7+4d-6+\frac {2d-2}2 = 5d\,.
  \end{align*}
  which again contradicts $k\ge 1$. So $n_1=2$, and $G_1$ defines a segment. $G_2$ now is an
  irreducible graph of dimension $d-1$ on $2(d-1)$ nodes. By induction, this must be a cube.
\end{proof}

\begin{proposition}\label{prop:2dm1}
  A $d$-dimensional combinatorial type $\lattice$ of a face of a Birkhoff polytope with
  $\bdim(\lattice)\ge 2d-1$ is a product of a cube and a triangle.
\end{proposition}
\begin{proof}
  As in the previous proof, our graph has $d+2n-k=d+4d-2-k=5d-2-k$ edges. Let $n_1, \ldots, n_k$ be
  the number of nodes of the upper layer of each component of the graph. Let $k_{2/3}$ be the number
  of components with two or three nodes and $k_o$ the number of components with an odd number of
  nodes. Using \autoref{cor:MinNumberOfEdges} we can estimate the number of edges in the graph by
  \begin{align*}
    e\ge \sum_{i=1}^k \left(2n_i+\left\lceil\frac{n_i}2\right\rceil\right) -k_{2/3}\ &=\
    4d-2+d-\frac12+\frac{k_o}{2}-k_{2/3}\ =\ 5d-\frac 52+\frac{k_o}{2}-k_{2/3}\,.
  \end{align*}
  Hence,
  \begin{align*}
    e\ &=\ 5d-k-2\ \ge\ 5d-\frac 52+\frac{k_o}2-k_{2/3}\ \quad
    \Longleftrightarrow\ \quad  k\le k_{2/3}-\frac{k_o-1}{2}\,,
  \end{align*}
  and we conclude $k_{2/3}=k$ and $k_o\le 1$. So at most one component has more then two nodes on
  each layer. However, $2d-1$ is odd, hence $k_o=1$. This implies the proposition.
\end{proof}

\begin{proposition}\label{prod:2dm2}
  A combinatorial type $\lattice$ of a Birkhoff face $F$ of dimension $d\ge 3$ with
  $\bdim(\lattice)\ge 2d-2$ is either a product of two lower dimensional faces or a pyramid over a
  cube of dimension $d-1$.  For $1$-dimensional and $2$-dimensional types $\lattice$ we have
  $\bdim(\lattice)=3$.
\end{proposition}
\begin{figure}[t]
  \centering
    \begin{tikzpicture}[scale=1.2]
      \foreach \x in {0,...,11} {
        \coordinate (u\x) at (\x,1);
        \coordinate (v\x) at (\x,0);
      }
      \coordinate (x) at (5.5,.5);
      
      \draw [line width=1pt] (u0) -- (v0);
      \draw [line width=1pt] (u0) -- (v1);
      \draw [line width=1pt] (u1) -- (v0);
      \draw [line width=1pt] (u1) -- (v1);
      \draw [line width=1.5pt,dashed] (u1) -- (v2);
      \draw [line width=1pt] (u2) -- (v2);
      \draw [line width=1pt] (u2) -- (v3);
      \draw [line width=1pt] (u3) -- (v3);
      \draw [line width=1pt] (u3) -- (v2);
      \draw [line width=1.5pt,dashed] (u3) -- (v4);
      \draw [line width=1pt] (u4) -- (v4);
      \draw [line width=1pt] (u4) -- (v5);
      \draw [line width=1pt] (u5) -- (v4);
      \draw [line width=1pt] (u5) -- (v5);
      \draw [line width=1.5pt,dashed] (u5) -- (v6);
      \draw [line width=1pt] (u6) -- (v7);
      \draw [line width=1pt] (u6) -- (v6);
      \draw [line width=1pt] (u7) -- (v6);
      \draw [line width=1pt] (u7) -- (v7);
      \draw [line width=1.5pt,dashed] (u7) -- (v8);
      \draw [line width=1pt] (u8) -- (v9);
      \draw [line width=1pt] (u8) -- (v8);
      \draw [line width=1pt] (u9) -- (v8);
      \draw [line width=1pt] (u9) -- (v9);
      \draw [line width=1.5pt,dashed] (u9) -- (v10);
      \draw [line width=1pt] (u10) -- (v10);
      \draw [line width=1pt] (u10) -- (v11);
      \node at (u11) {} edge [line width=1.5pt,dashed,bend left=7] (x);
      \node at (x) {} edge [line width=1.5pt,dashed,bend right=7] (v0);
      \draw [line width=1pt] (u11) -- (v10);
      \draw [line width=1pt] (u11) -- (v11);
      
      \foreach \x in {0,...,11} {
        \draw [color=white,fill=red,line width=2pt] (u\x) circle (4pt);
        \draw [color=white,fill=red,line width=2pt] (v\x) circle (4pt);
      }
    \end{tikzpicture}
    \caption{The only connected irreducible face graph on $2d-2$ nodes (here $d=7$). The face set
      defining the base is drawn with dashed lines. Those edges appear in the unique matching
      corresponding to the apex.}
  \label{fig:ElementaryPyramid}
  \label{fig:PyramidOverCube}
  \end{figure}
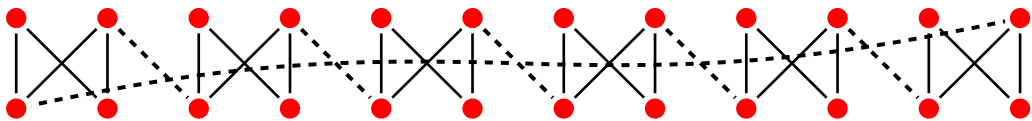
\begin{proof}
  Assume that $F$ is not a product. We first consider the case that $d\ge 4$.  In this case $F$
  corresponds to an irreducible face graph $G$ on $2d-2$ nodes with $e=d+2(2d-2)-1=5d-5$ edges. By
  \autoref{prop:degree} the maximum degree of a node in the graph is $3$. Hence, in each layer we
  have $d-1$ nodes of degree $2$ (\emph{i.e.}, minimal nodes, as $G$ is irreducible) and $d-1$ nodes
  of degree $3$. 

  Consider the nodes in the upper layer. By \autoref{cor:partnerbound} a node of degree $3$ is
  partner for precisely one minimal node, and any minimal node has a unique partner. As also in the
  lower layer the degree of a node is at most $3$ we can deduce that a node in the lower layer is on
  at most one path connecting a minimal node with its partner. By counting nodes we see that each
  node in the lower layer is on precisely one such path.

  We split the graph according to these paths. Let $u_1, \ldots, u_{d-1}$ be the minimal nodes in
  the upper layer, and for each $1\le i\le d-1$ define $N_i$ to be the graph induced by $u_i$, its
  unique partner, and the two paths of length $2$ between them. By the above argument, the graphs
  $N_i$ are pairwise disjoint. Hence, their union $N:=\bigcup N_i$ define a face subgraph of $G$
  whose corresponding face is the product of $d-1$ segments, \emph{i.e.}, it is isomorphic to a
  $(d-1)$-dimensional cube. 

  The graphs $N$ and $G$ have the same number of nodes, and $G$ has $d-1$ additional edges. $N$ is
  disconnected, so those $d-1$ edges must connect the $d-1$ components of $N$. As $G$ is a face
  graph, \emph{i.e.}, each edge must be contained in some perfect matching of $G$, the graphs $N_i$
  can only be connected circularly. Hence, up to relabeling and flipping upper and lower layer in
  the graphs $N_i$, the graph $G$ must look like \autoref{fig:PyramidOverCube}. This is the circular
  connection of the $N_i$.

  For $d=3$ and $n=4$ there is a second irreducible graph on four nodes, see
  \autoref{fig:tworep}\subref{fig:additionalSimplex}. This graph defines a tetrahedron. However, the
  graph in \autoref{fig:tworep}\subref{fig:additionalSimplexsmall} also defines a tetrahedron, so
  this face already appears in $\birkhoff[3]$.
\end{proof}
We can combine the three previous \propositionsref{prop:cube}, \ref{prop:2dm1}, and \ref{prod:2dm2}
into the following slightly extended theorem.
\begin{theorem}\label{thm:CompleteClassification}
  Let $\lattice$ be a combinatorial type of a $d$-dimensional face of a Birkhoff polytope with
  $\bdim(\lattice)\ge 2d-2$. Then $\lattice$ is a
  \begin{enumerate}
  \item a cube, if $\bdim(\lattice)=2d$,
  \item a product of a cube and a triangle, if $\bdim(\lattice)=2d-1$,
  \item a polytope of one of the following types, if $\bdim(\lattice)=2d-2$:
    \begin{enumerate}
    \item a pyramid over a cube,
    \item a product of a cube and a pyramid over a cube,
    \item a product of two triangles and a cube.
    \end{enumerate}
  \end{enumerate}
\end{theorem}
\begin{proof}
  The only claim not contained in the previous propositions is the classification of products
  leading to a $d$-dimensional face $F$ on $2d-2$ nodes.  Assume that $F$ is a product $F_1\times
  F_2$ of polytopes $F_1$, $F_2$ (which may itself be products) of dimensions $d_1$ and $d_2$. Let
  $\gr(F_i)$ have $n_i$ nodes in each layer.  Then $d_1+d_2=d$. Define non-negative numbers
  $r_i:=2d_i-n_i$, $i=1,2$.

  Assume that $\gr(F)$ has $k$ and $\gr(F_i)$ has $k_i$ components, $i=1,2$.  Then $k_1+k_2=k$,
  $\gr(F)$ has $2n+d-k=2(2d-2)+d-k=5d-4-k$ edges, and $\gr(F_i)$ has
  \begin{align*}
    2n_i+d_i-k_i\ &=\ 2(2d_i-r_i)+d_i-k_i\ =\ 5d_i-2r_i-k_i
  \end{align*}
  edges, for $i=1,2$.  This implies $r_1+r_2=2$.  Hence, $r_1=2, r_2=0$ or $r_1=r_2=1$ or
  $r_1=0,r_2=2$ and the claim follows by induction.
\end{proof}

\subsection{$d$-dimensional combinatorial types with Birkhoff dimension at least $2d-3$.}
\label{subsec:2d-3}
\begin{figure}[bt]
\centering
\label{fig:tworeps}
  \subfigure[A second $3$-dimensional face with an irreducible graph on four nodes.\label{fig:additionalSimplex}]{
    \begin{tikzpicture}[scale=1.1]
      \foreach \x in {-1,0,1,2,3,5} {
        \coordinate (u\x) at (\x,1);
        \coordinate (v\x) at (\x,0);
      }
      
      \draw [line width=1pt] (u0) -- (v0);
      \draw [line width=1pt] (u0) -- (v1);
      \draw [line width=1pt] (u1) -- (v0);
      \draw [line width=1pt] (u1) -- (v1);
      \draw [line width=1pt] (u1) -- (v2);
      \draw [line width=1pt] (u1) -- (v3);
      \draw [line width=1pt] (u2) -- (v1);
      \draw [line width=1pt] (u2) -- (v2);
      \draw [line width=1pt] (u3) -- (v1);
      \draw [line width=1pt] (u3) -- (v3);
      
      \foreach \x in {0,1,2,3} {
        \draw [color=white,fill=red,line width=2pt] (u\x) circle (4pt);
        \draw [color=white,fill=red,line width=2pt] (v\x) circle (4pt);
      }
      \draw [color=white] (u-1) circle (4pt);
      \draw [color=white] (u5) circle (4pt);
    \end{tikzpicture}
  }
  \subfigure[A smaller representation of the same face.\label{fig:additionalSimplexsmall}]{
    \begin{tikzpicture}[scale=1.1]
      \foreach \x in {-1,0,1,2,4} {
        \coordinate (u\x) at (\x,1);
        \coordinate (v\x) at (\x,0);
      }
      
      \draw [line width=1pt] (u0) -- (v0);
      \draw [line width=1pt] (u0) -- (v1);
      \draw [line width=1pt] (u1) -- (v0);
      \draw [line width=1pt] (u1) -- (v1);
      \draw [line width=1pt] (u1) -- (v2);
      \draw [line width=1pt] (u2) -- (v1);
      \draw [line width=1pt] (u2) -- (v2);
      \draw [line width=1pt] (u2) -- (v0);
      
      \foreach \x in {0,1,2} {
        \draw [color=white,fill=red,line width=2pt] (u\x) circle (4pt);
        \draw [color=white,fill=red,line width=2pt] (v\x) circle (4pt);
      }
      \draw [color=white] (u-1) circle (4pt);
      \draw [color=white] (u4) circle (4pt);
    \end{tikzpicture}
  }
  \caption{Two representations of the same combinatorial type of
  face.\label{fig:tworep}}
\end{figure}
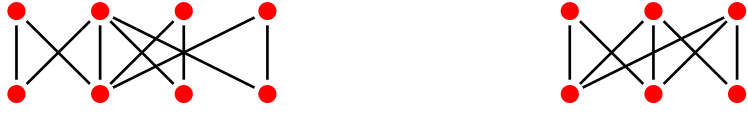

In this section we introduce a new construction for polytopes, the \emph{joined products} and
\emph{reduced joined products} and use them to classify faces of $\birkhoff$ for $n=2d-3$, but also
many other faces of $\birkhoff$ are of this type.  We give a combinatorial description and deduce
their corresponding face graph. We use these to classify combinatorial types of faces
in~\autoref{thm:ClassWedges}.

Let $Q_1, \ldots, Q_k$ be polytopes in $\R^m$ (not necessarily all $m$-dimensional). The
\emph{Cayley sum} of $Q_1, \ldots, Q_k$ is the polytope
\begin{align*}
  \cayley(Q_1, \ldots, Q_k)\ :=\ \conv(Q_1\times e_1,\, Q_2\times e_2,\, \ldots, Q_k\times e_k)\,,
\end{align*}
where $e_1, e_2, \ldots, e_k$ are the $k$-dimensional unit vectors. We use this construction for a
special set of polytopes $Q_1, \ldots, Q_k$. Let $\zv{d}$ be the $d$-dimensional zero vector, and
$P_i$ $d_i$-dimensional polytopes for $1\le i\le k$. We define
\begin{align*}
  Q_i\ &:=\ P_1\times \cdots \times P_{i-1}\times \zv{d_i}\times P_{i+1}\times \cdots\times P_k\,,
  \intertext{and}
  Q_0\ &:=\ P_1\times \cdots\times P_k\,.
\end{align*}
\begin{definition}
  The \emph{joined product} of the polytopes $P_1, \ldots, P_k$ is
  \begin{align}
    \joinedprod(P_1, \ldots, P_k)\ &:=\  \cayley(Q_1, \ldots, Q_k)\,,\label{eq:joinedprod}\\
    \intertext{and the \emph{reduced joined product} is} \joinedprod^{\rm red}(P_1, \ldots, P_k)\
    &:=\ \cayley(Q_0, \ldots, Q_k)\,.\label{eq:redjoinedprod}
  \end{align}
\end{definition}
The reduced joined product is the special case of the joined product where one of the factors is
just a point. Hence, in the following considerations on combinatorial properties we restrict to
joined products. 

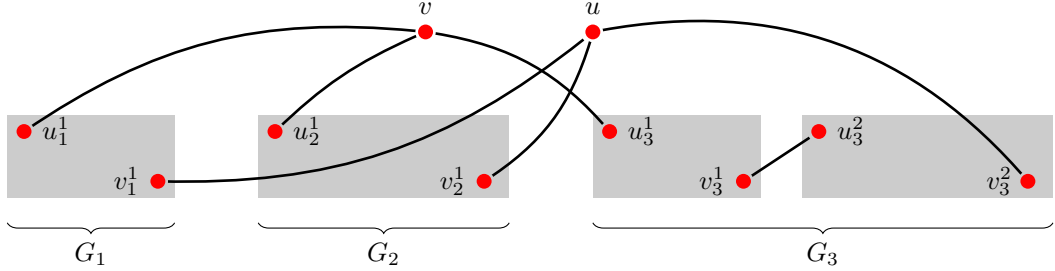
\begin{figure}[t]
  \centering
  \begin{tikzpicture}[scale=1.1]
    \coordinate (g111) at (0,0);
    \coordinate (g112) at (2,0);
    \coordinate (g121) at (0,1);
    \coordinate (g122) at (2,1);

    \coordinate (g111b) at (0,-.3);
    \coordinate (g112b) at (2,-.3);

    \coordinate (g211) at (3,0);
    \coordinate (g212) at (6,0);
    \coordinate (g221) at (3,1);
    \coordinate (g222) at (6,1);

    \coordinate (g211b) at (3,-.3);
    \coordinate (g212b) at (6,-.3);

    \coordinate (g3a11) at (7,0);
    \coordinate (g3a12) at (9,0);
    \coordinate (g3a21) at (7,1);
    \coordinate (g3a22) at (9,1);
    \coordinate (g3b11) at (9.5,0);
    \coordinate (g3b12) at (12.5,0);
    \coordinate (g3b21) at (9.5,1);
    \coordinate (g3b22) at (12.5,1);

    \coordinate (g311b) at (7,-.3);
    \coordinate (g312b) at (12.5,-.3);

    \coordinate (g1) at (1,-1);
    \coordinate (g2) at (4.5,-1);
    \coordinate (g3) at (10,-1);

    \coordinate (u1) at (barycentric cs:g112=1,g121=8,g111=1);%
    \coordinate (v1) at (barycentric cs:g121=1,g112=8,g122=1);%

    \coordinate (u2) at (barycentric cs:g212=1,g211=2,g221=12);%
    \coordinate (v2) at (barycentric cs:g221=1,g212=8,g222=1);%

    \coordinate (u3a) at (barycentric cs:g3a11=1,g3a21=8,g3a12=1);%
    \coordinate (v3a) at (barycentric cs:g3a21=1,g3a12=8,g3a22=1);%
    \coordinate (u3b) at (barycentric cs:g3b12=1,g3b21=12,g3b11=2);%
    \coordinate (v3b) at (barycentric cs:g3b21=1,g3b12=8,g3b22=1);%

    \coordinate (v) at (5,2);
    \coordinate (u) at (7,2);
   
    \fill[color=black!20] (g111) -- (g112) -- (g122) -- (g121) -- cycle;
    \fill[color=black!20] (g211) -- (g212) -- (g222) -- (g221) -- cycle;
    \fill[color=black!20] (g3a11) -- (g3a12) -- (g3a22) -- (g3a21) -- cycle;
    \fill[color=black!20] (g3b11) -- (g3b12) -- (g3b22) -- (g3b21) -- cycle;

    \draw [decorate,decoration={brace,amplitude=4pt}] (g112b) -- (g111b) node [black,midway,below=4pt] {$G_1$};
    \draw [decorate,decoration={brace,amplitude=4pt}] (g212b) -- (g211b) node [black,midway,below=4pt] {$G_2$};
    \draw [decorate,decoration={brace,amplitude=4pt}] (g312b) -- (g311b) node [black,midway,below=4pt] {$G_3$};

    \node at (u) {} edge [line width=1pt,bend left=20,auto] (v1);
    \node at (v) {} edge [line width=1pt,bend right=20,auto] (u1);
    \node at (u) {} edge [line width=1pt,bend left=20,auto] (v2);
    \node at (v) {} edge [line width=1pt,bend right=10,auto] (u2);
    \node at (v) {} edge [line width=1pt,bend left=20,auto] (u3a);
    \node at (u) {} edge [line width=1pt,bend left=30,auto] (v3b);
    \node at (u3b) {} edge [line width=1pt,auto] (v3a);

    \draw[color=black!20,fill=red,line width=1pt] (u1) circle (3pt) node[black,right,xshift=3pt] {$u_1^1$};
    \draw[color=black!20,fill=red,line width=1pt] (v1) circle (3pt) node[black,left,xshift=-3pt] {$v_1^1$};
    \draw[color=black!20,fill=red,line width=1pt] (u2) circle (3pt) node[black,right,xshift=3pt] {$u_2^1$};
    \draw[color=black!20,fill=red,line width=1pt] (v2) circle (3pt) node[black,left,xshift=-3pt] {$v_2^1$};
    \draw[color=black!20,fill=red,line width=1pt] (u3a) circle (3pt) node[black,right,xshift=3pt] {$u_3^1$};
    \draw[color=black!20,fill=red,line width=1pt] (v3a) circle (3pt) node[black,left,xshift=-3pt] {$v_3^1$};
    \draw[color=black!20,fill=red,line width=1pt] (u3b) circle (3pt) node[black,right,xshift=3pt] {$u_3^2$};
    \draw[color=black!20,fill=red,line width=1pt] (v3b) circle (3pt) node[black,left,xshift=-2pt] {$v_3^2$};

    \draw[color=white,fill=red,line width=1pt] (u) circle (3pt) node[black,above,yshift=3pt] {$u$};
    \draw[color=white,fill=red,line width=1pt] (v) circle (3pt) node[black,above,yshift=3pt] {$v$};
  \end{tikzpicture}
  \caption{The joined product of three graphs. The reduced joined product additionally has an edge
    be\-tween $u$ and $v$.}
  \label{fig:joinedprodgraph}
\end{figure}
Let $m_i$ be the number of vertices of $P_i$, and $M:=\prod_{i=1}^k m_i$. We will show that the
joined product $P$ of $P_1, \ldots, P_k$ has $\sum_{i=1}^k \frac{M}{m_i}$ vertices.  Assume by
contradiction that there is $v\in \bigcup_{i=1}^k \vertices(Q_i)\times e_i$ that is not a vertex of
$P$.  Then $v$ is a convex combination of some of the other vertices, say $v_1, \ldots, v_r$. The
point $v$, and any $v_j$, $1\le j\le r$ has exactly one entry different from $0$ among the last $k$
coordinates.  Hence, any points in the convex combination of $v$ with a positive coefficient
coincide with $v$ in that entry. By construction, this implies that $v$ and all points in its convex
combination are completely contained in one the factors $Q_i\times e_i$, for some $i$. But $v$ and
$v_j$, $1\le j\le r$ are vertices of $Q_i\times e_i$, a contradiction. Hence, any vertex of some
$Q_i$ corresponds to a vertex of the joined product.

The joined products have two obvious types of facets.  For any facet $F_1$ of $P_1$ the polytope
$\joinedprod(F_1, P_2, \ldots, P_k)$ is a facet of $P$. We can do this for any factor of the product
and any two such facets are distinct.  For $k\ge 3$ we also have the facets
\begin{align*}
  \conv(Q_1\times e_1, \ldots, Q_{i-1}\times e_{i-1}, Q_{i+1}\times e_{i+1}, \ldots, Q_k\times
  e_k)\,.
\end{align*}
We construct the corresponding graphs. Let $G_1, \ldots, G_k$, $k\ge 2$ be (not necessarily
connected) face graphs. Let $H$ be a graph with two isolated nodes, and $\ov H$ a graph with two
nodes and the edge between them.  We define the \emph{joined product} $\JPG(G_1, \ldots, G_k)$ of
$G_1, \ldots, G_k$ as the union of the circular connections of $H$ with each $G_i$, and the
\emph{reduced joined product} $\JPG^{\rm red}(G_1, \ldots, G_k)$ as the union of the circular
connections of $\ov H$ with each $G_i$.

We make this construction more precise. For each graph $G_i$ with connected components $G_i^1, G_i^2,
\ldots, G_i^{c_i}$ we select a set $S(G_i)$ of a pair of nodes $u_i^j,v_i^j$ in each $G_i^j$, $1\le
j\le c_i$, with $u_i^j$ in the upper and $v_i^j$ in the lower layer. Let $u,v$ be the nodes of
$H$. Then the \emph{joined product} $\JPG(G_1, \ldots, G_k)$ is the disjoint union of $H$ and $G_1,
\ldots, G_k$ together with the edges $(u,v_i^{c_i})$ and $(v,u_i^1)$ for $1\le i\le k$ and $(v_i^j,
u_i^{j+1})$ for $1\le i\le k$ and $1\le j\le c_i-1$. The \emph{reduced joined product} is obtained
in the same way with $\ov H$ instead of $H$ (with labels $u$ and $v$ for the nodes of $\ov H$).  See
\autoref{fig:joinedprodgraph} for an illustration.

Clearly, the isomorphism types of the resulting graphs depend on the choice of the two nodes in each
component of the $G_i$. In general, they will not be face graphs. More precisely, the joined product
graph $\JPG(G_1,\ldots, G_k)$ is a face graph if and only if for each $1\le i\le k$ the circular
connection of $H$ and $G_i$ using the nodes in $S(G_i)$ is a face graph, and similarly for
$\JPG^{\rm red}(G_1, \ldots, G_k)$. Note that we have called a choice $S(G_i)$ of nodes
$u_i^j,v_i^j$, $1\le j\le c_i$ \emph{pyramidal}, if $G_i-S(G_i)$ has a unique perfect matching, for
$1\le i\le k$, see~\autoref{def:pyramidal}.
\begin{theorem}
  Let $G_1, \ldots, G_k$ be face graphs with pyramidal sets $S_i(G_i)$ of nodes and
  $F_i:=\face(G_i)$, $1\le i\le k$.
  \begin{enumerate}
  \item $G:=\JPG(G_1, \ldots, G_k)$ is a face graph with face given by $\joinedprod(F_1, \ldots,
    F_k)$.
  \item $G:=\JPG^{\rm red}(G_1, \ldots, G_k)$ is a face graph with face given by $\joinedprod^{\rm
      red}(F_1, \ldots, F_k)$.
  \end{enumerate}
  \end{theorem}
  \begin{proof}
  We prove only the first statement. The proof of the second is analogous. 

  Let $H$ be the graph with two isolated nodes $u$ and $v$ as above. By construction, for each $i$
  the circular connection of the disjoint union of $H$ and $G_i$ has a unique perfect matching
  $M_i$. This matching is given by the edges $(u, v_i^{c_i})$, $(v,u_i^1)$, the edges
  $(v_i^j,u_i^{j+1})$ for $1\le j\le c_i-1$ and the unique perfect matching in $G_i-S(G_i)$.  Hence,
  $G$ is a face graph. Its perfect matchings are precisely products of some $M_i$ with a choice of a
  perfect matching in all $G_j$ for $j\ne i$.
  
  It remains to show that convex hull of the vertices defined by the perfect matchings in the joined
  product of the graphs is affinely isomorphic to the joined product of the $F_i$. For this, it
  suffices to note that a perfect matching that contains, for some $1\le i\le k$, one of the edges
  $(u,v_i^{c_i})$, $(v,u_i^1$ or $(v_i^j,u_i^{j+1})$ for $1\le j\le c_i-1$ necessarily also contains
  the others. Hence, up to affine isomorphism, we can forget all but one of the corresponding
  coordinates. This gives the Cayley structure~\eqref{eq:joinedprod} with the products of the
  remaining $F_j$, $j\ne i$.
\end{proof}
Note that, as a face can have more than one representation as an irreducible face graph, it does not
follow from this theorem that all graphs of faces of some $\birkhoff$ that are joined products of
some other faces are of the form given in the theorem. However, if a face is a joined product of
some polytopes, then those are again faces of some Birkhoff polytope.  We need one more lemma before
we can continue our classification.
\begin{lemma}\label{lemma:reductionlemma}
  Let $G$ be an irreducible face graph with a minimal node $v$ in the upper layer. Let $w_1, w_2$ be
  the neighbors of $v$. If $x\ne v$ is a node adjacent to $w_1$ but not to $w_2$, then the graph
  $G'$ obtained by replacing $(x,w_1)$ with $(x,w_2)$ is a face graph with the same combinatorial
  type as $G$.
\end{lemma}
\begin{proof}
  The node $v$ has degree $2$ in both graphs, and the reduction at $v$ gives the same graph for $G$
  and $G'$.
\end{proof}
We are ready to classify all $d$-dimensional combinatorial types $\lattice$ with $\bdim(\lattice)\ge
2d-3$.
\begin{figure}[bt]
  \centering
  \subfigure[$d$-faces on $2d-3$ nodes: type (a),    \label{fig:sub:HighestWedgesTypeA}]{

    \begin{tikzpicture}[scale=.75]
      \foreach \x in {0,...,12} {%
        \coordinate (u\x) at (\x,1.8);%
        \coordinate (v\x) at (\x,0);%
      }%
      \foreach \x in {0,...,12} {%
        \draw[line width=1pt] (u\x) -- (v\x);
      }%

      \coordinate (uh) at (6,.9);

      \draw[line width=1pt] (u0) -- (v1);
      \draw[line width=1pt] (u1) -- (v0);
      \draw[line width=1pt] (u1) -- (v2);
      \draw[line width=1pt] (u2) -- (v1);

      \draw[line width=1pt] (u3) -- (v4);
      \draw[line width=1pt] (u4) -- (v3);

      \draw[line width=1pt] (u5) -- (v6);
      \draw[line width=1pt] (u6) -- (v5);

      \draw[line width=1pt] (u7) -- (v8);
      \draw[line width=1pt] (u8) -- (v7);

      \draw[line width=1pt] (u9) -- (v10);
      \draw[line width=1pt] (u10) -- (v9);

      \draw[line width=1pt] (u11) -- (v12);
      \draw[line width=1pt] (u12) -- (v11);

      \draw[line width=1pt] (u2) -- (v3);
      \draw[line width=1pt] (u4) -- (v5);
      \draw[line width=1pt] (u6) -- (v7);
      \draw[line width=1pt] (u8) -- (v9);
      \draw[line width=1pt] (u10) -- (v11);
      \node at (u12) {} edge [line width=1pt,bend left=10] (uh);
      \node at (v0) {} edge [line width=1pt,bend left=10] (uh);

      \foreach \x in {0,...,12} {%
        \draw[color=white,fill=red,line width=2pt] (u\x) circle (5pt);
        \draw[color=white,fill=red,line width=2pt] (v\x) circle (5pt);
      }%

      \draw[color=white,fill=red,line width=2pt] (u2) circle (5pt) node[black,above] {$u^0$};
      \draw[color=white,fill=red,line width=2pt] (u4) circle (5pt) node[black,above] {$u^1$};
      \draw[color=white,fill=red,line width=2pt] (u6) circle (5pt) node[black,above] {$u^2$};
      \draw[color=white,fill=red,line width=2pt] (u8) circle (5pt) node[black,above] {$u^3$};
      \draw[color=white,fill=red,line width=2pt] (u10) circle (5pt) node[black,above] {$u^4$};
      \draw[color=white,fill=red,line width=2pt] (u12) circle (5pt) node[black,above] {$u^5$};
      \draw[color=white,fill=red,line width=2pt] (v0) circle (5pt) node[black,below] {$v^0$};
      \draw[color=white,fill=red,line width=2pt] (v3) circle (5pt) node[black,below] {$v^1$};
      \draw[color=white,fill=red,line width=2pt] (v5) circle (5pt) node[black,below] {$v^2$};
      \draw[color=white,fill=red,line width=2pt] (v7) circle (5pt) node[black,below] {$v^3$};
      \draw[color=white,fill=red,line width=2pt] (v9) circle (5pt) node[black,below] {$v^4$};
      \draw[color=white,fill=red,line width=2pt] (v11) circle (5pt) node[black,below] {$v^5$};

    \end{tikzpicture}
  }

  \bigskip

  \subfigure[$d$-faces on $2d-3$ nodes: type (b)    \label{fig:sub:HighestWedgesTypeB}]{
    \begin{tikzpicture}[scale=.75]
      \foreach \x in {0,...,12} {%
        \coordinate (u\x) at (\x,1.8);%
        \coordinate (v\x) at (\x,0);%
      }%
      \foreach \x in {0,...,12} {%
        \draw[line width=1pt] (u\x) -- (v\x);
      }%

      \draw[line width=1pt] (u0) -- (v1);
      \draw[line width=1pt] (u1) -- (v0);

      \draw[line width=1pt] (u2) -- (v3);
      \draw[line width=1pt] (u3) -- (v2);

      \draw[line width=1pt] (u4) -- (v5);
      \draw[line width=1pt] (u5) -- (v4);

      \draw[line width=1pt] (u7) -- (v8);
      \draw[line width=1pt] (u8) -- (v7);

      \draw[line width=1pt] (u9) -- (v10);
      \draw[line width=1pt] (u10) -- (v9);

      \draw[line width=1pt] (u11) -- (v12);
      \draw[line width=1pt] (u12) -- (v11);

      \draw[line width=1pt] (u1) -- (v2);
      \draw[line width=1pt] (u3) -- (v4);
      \draw[line width=1pt] (u5) -- (v6);
      \draw[line width=1pt] (u6) -- (v6);
      \draw[line width=1pt] (u6) -- (v7);
      \draw[line width=1pt] (u8) -- (v9);
      \draw[line width=1pt] (u10) -- (v11);
      \draw[line width=1pt] (u6) -- (v0);
      \draw[line width=1pt] (u12) -- (v6);

      \foreach \x in {0,...,12} {%
        \draw[color=white,fill=red,line width=2pt] (u\x) circle (5pt);
        \draw[color=white,fill=red,line width=2pt] (v\x) circle (5pt);
      }%
      \draw[color=white,fill=red,line width=2pt] (v0) circle (5pt) node[black,below] {$v_{0}^0$};
      \draw[color=white,fill=red,line width=2pt] (v2) circle (5pt) node[black,below] {$v_{0}^1$};
      \draw[color=white,fill=red,line width=2pt] (v4) circle (5pt) node[black,below] {$v_{0}^2$};
      \draw[color=white,fill=red,line width=2pt] (v7) circle (5pt) node[black,below] {$v_{1}^0$};
      \draw[color=white,fill=red,line width=2pt] (v9) circle (5pt) node[black,below] {$v_{1}^1$};
      \draw[color=white,fill=red,line width=2pt] (v11) circle (5pt) node[black,below] {$v_{1}^2$};
      \draw[color=white,fill=red,line width=2pt] (v6) circle (5pt) node[black,below,yshift=-4pt] {$v$};

      \draw[color=white,fill=red,line width=2pt] (u1) circle (5pt) node[black,above] {$u_{0}^0$};
      \draw[color=white,fill=red,line width=2pt] (u3) circle (5pt) node[black,above] {$u_{0}^1$};
      \draw[color=white,fill=red,line width=2pt] (u5) circle (5pt) node[black,above] {$u_{0}^2$};
      \draw[color=white,fill=red,line width=2pt] (u8) circle (5pt) node[black,above] {$u_{1}^0$};
      \draw[color=white,fill=red,line width=2pt] (u10) circle (5pt) node[black,above] {$u_{1}^1$};
      \draw[color=white,fill=red,line width=2pt] (u12) circle (5pt) node[black,above] {$u_{1}^2$};
      \draw[color=white,fill=red,line width=2pt] (u6) circle (5pt) node[black,above,yshift=4pt] {$u$};
    \end{tikzpicture}
  } 

  \bigskip

  \subfigure[$d$-faces on $2d-3$ nodes: type (c)    \label{fig:sub:HighestWedgesTypeC}]{
    \begin{tikzpicture}[scale=.74]
      \foreach \x in {0,...,18} {%
        \coordinate (u\x) at (\x,1.8);%
        \coordinate (v\x) at (\x,0);%
      }%

      \coordinate (uh) at (12,.9);

      \foreach \x in {0,...,5} {%
        \draw[line width=1pt] (u\x) -- (v\x);
      }%
      \foreach \x in {7,...,18} {%
        \draw[line width=1pt] (u\x) -- (v\x);
      }%

      \draw[line width=1pt] (u0) -- (v1);
      \draw[line width=1pt] (u1) -- (v0);

      \draw[line width=1pt] (u2) -- (v3);
      \draw[line width=1pt] (u3) -- (v2);

      \draw[line width=1pt] (u4) -- (v5);
      \draw[line width=1pt] (u5) -- (v4);

      \draw[line width=1pt] (u7) -- (v8);
      \draw[line width=1pt] (u8) -- (v7);

      \draw[line width=1pt] (u9) -- (v10);
      \draw[line width=1pt] (u10) -- (v9);

      \draw[line width=1pt] (u11) -- (v12);
      \draw[line width=1pt] (u12) -- (v11);

      \draw[line width=1pt] (u13) -- (v14);
      \draw[line width=1pt] (u14) -- (v13);

      \draw[line width=1pt] (u15) -- (v16);
      \draw[line width=1pt] (u16) -- (v15);

      \draw[line width=1pt] (u17) -- (v18);
      \draw[line width=1pt] (u18) -- (v17);

      \draw[line width=1pt] (u1) -- (v2);
      \draw[line width=1pt] (u3) -- (v4);
      \draw[line width=1pt] (u5) -- (v6);
      \draw[line width=1pt] (u6) -- (v7);
      \draw[line width=1pt] (u8) -- (v9);
      \draw[line width=1pt] (u10) -- (v11);
      \draw[line width=1pt] (u6) -- (v0);
      \draw[line width=1pt] (u12) -- (v6);

      \draw[line width=1pt] (u6) -- (v13);
      \draw[line width=1pt] (u14) -- (v15);
      \draw[line width=1pt] (u16) -- (v17);
      \node at (u18) {} edge [line width=1pt,bend left=10] (uh);
      \node at (v6) {} edge [line width=1pt,bend left=10] (uh);

      \foreach \x in {0,...,18} {%
        \draw[color=white,fill=red,line width=2pt] (u\x) circle (5pt);
        \draw[color=white,fill=red,line width=2pt] (v\x) circle (5pt);
      }%
      \draw[color=white,fill=red,line width=2pt] (v0) circle (5pt) node[black,below] {$v_{0}^0$};
      \draw[color=white,fill=red,line width=2pt] (v2) circle (5pt) node[black,below] {$v_{0}^1$};
      \draw[color=white,fill=red,line width=2pt] (v4) circle (5pt) node[black,below] {$v_{0}^2$};
      \draw[color=white,fill=red,line width=2pt] (v7) circle (5pt) node[black,below] {$v_{1}^0$};
      \draw[color=white,fill=red,line width=2pt] (v9) circle (5pt) node[black,below] {$v_{1}^1$};
      \draw[color=white,fill=red,line width=2pt] (v11) circle (5pt) node[black,below] {$v_{1}^2$};
      \draw[color=white,fill=red,line width=2pt] (v13) circle (5pt) node[black,below] {$v_{2}^2$};
      \draw[color=white,fill=red,line width=2pt] (v15) circle (5pt) node[black,below] {$v_{2}^2$};
      \draw[color=white,fill=red,line width=2pt] (v17) circle (5pt) node[black,below] {$v_{2}^2$};
      \draw[color=white,fill=red,line width=2pt] (v6) circle (5pt) node[black,below,yshift=-4pt] {$v$};

      \draw[color=white,fill=red,line width=2pt] (u1) circle (5pt) node[black,above] {$u_{0}^0$};
      \draw[color=white,fill=red,line width=2pt] (u3) circle (5pt) node[black,above] {$u_{0}^1$};
      \draw[color=white,fill=red,line width=2pt] (u5) circle (5pt) node[black,above] {$u_{0}^2$};
      \draw[color=white,fill=red,line width=2pt] (u8) circle (5pt) node[black,above] {$u_{1}^0$};
      \draw[color=white,fill=red,line width=2pt] (u10) circle (5pt) node[black,above] {$u_{1}^1$};
      \draw[color=white,fill=red,line width=2pt] (u12) circle (5pt) node[black,above] {$u_{1}^2$};
      \draw[color=white,fill=red,line width=2pt] (u14) circle (5pt) node[black,above] {$u_{2}^2$};
      \draw[color=white,fill=red,line width=2pt] (u16) circle (5pt) node[black,above] {$u_{2}^2$};
      \draw[color=white,fill=red,line width=2pt] (u18) circle (5pt) node[black,above] {$u_{2}^2$};
      \draw[color=white,fill=red,line width=2pt] (u6) circle (5pt) node[black,above,yshift=4pt] {$u$};
    \end{tikzpicture}
  } 
  \caption{$d$-dimensional face graphs on $2d-3$ nodes.}
  \label{fig:elHighestWedges}
\end{figure}
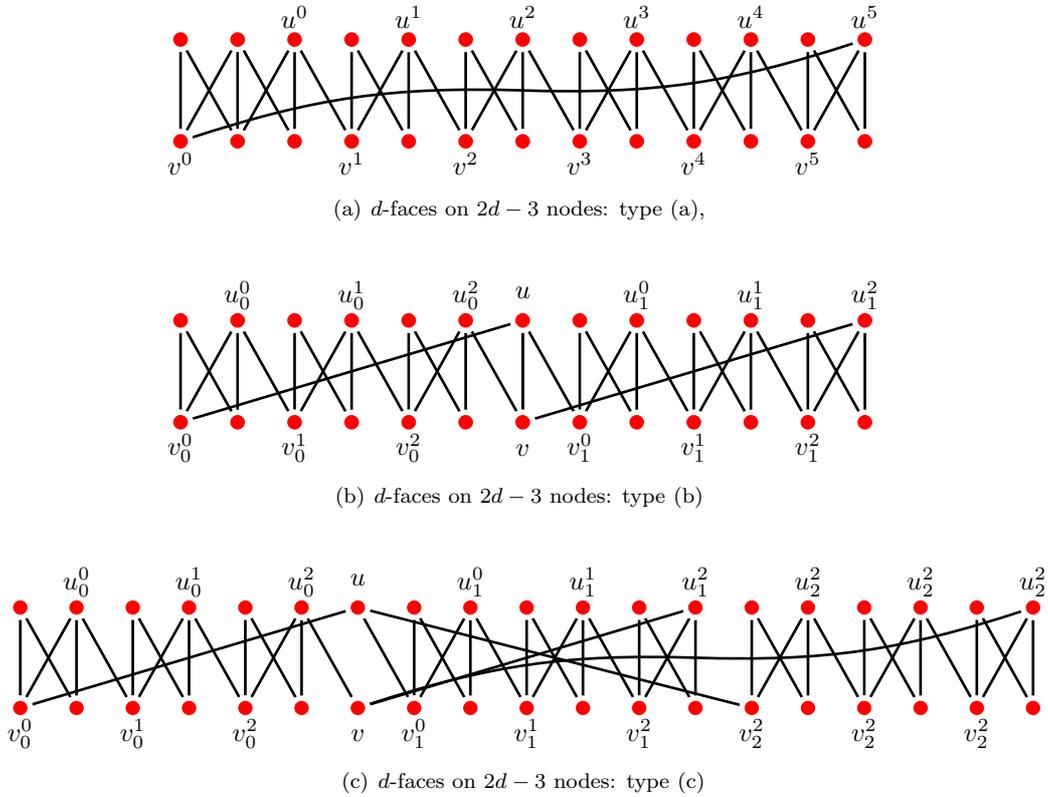
\begin{theorem}\label{thm:ClassWedges}
  Let $\lattice$ be a combinatorial type of a $d$-dimensional face with $\bdim(\lattice)\ge
  2d-3$. Then $\lattice$ is of one of the following four types.
  \begin{enumerate}
  \item\label{thm:ClassWedges:TypeA} Pyramid over a product of a cube and a triangle. See
    \autoref{fig:elHighestWedges}\subref{fig:sub:HighestWedgesTypeA}.
  \item\label{thm:ClassWedges:TypeC} A reduced joined product of two cubes (of possibly different
    dimensions).  See \autoref{fig:elHighestWedges}\subref{fig:sub:HighestWedgesTypeB}.
  \item\label{thm:ClassWedges:TypeD} A joined product of three cubes (of possibly different
    dimensions).  See \autoref{fig:elHighestWedges}\subref{fig:sub:HighestWedgesTypeC}.
  \end{enumerate}
\end{theorem}
Note that the theorem does not claim that these faces do really only appear in $\birkhoff[2d-3]$,
but only that, if a face appears in $\birkhoff[2d-3]$ for the first time, then it must be of one of
these types.  The stronger statement is certainly true for the first cases, as it contains the
product of a $(d-3)$-cube with a triangle as a proper face, and this cannot be represented with less
nodes.
\begin{proof}
  We classify the possible face graphs. The previous theorem translates this into combinatorial
  types of faces. An irreducible face graph on $2d-3$ nodes has $5d-7$ edges. We have either $d-1$
  or $d-2$ minimal nodes, and the maximal degree of a node is $4$. By counting we conclude that
  there is at most one node of degree $4$ in each layer. Further, if there is a node of degree $4$
  in one layer, then each node of degree at least $3$ in this layer is partner for some minimal
  node. In the other case at most one of the nodes of degree $3$ is not a partner.  We show first
  that we can reduce to the case that the maximal degree in $G$ is $3$.

  Let $v$ be a node of degree $4$ in the upper layer with neighbors $w_1, \ldots, w_4$. Then $v$ is
  partner for two nodes $u_1, u_2$. For both there are two disjoint paths of length $2$ connecting
  them to $v$. Let $w_1, w_2$ be the intermediate nodes of the paths to $u_1$. The intermediate
  nodes of the paths to $u_2$ cannot coincide with those two, as otherwise $(v,w_1)$ and $(v,w_2)$
  are not part of a perfect matching. So we are left with the cases sketched in
  \figuresref{fig:deg4nodes}\subref{fig:sub:deg4node1}
  and~\ref{fig:deg4nodes}\subref{fig:sub:deg4node2}, up to additional edges incident to some of the
  $w_i$.

Consider first the case given in \autoref{fig:deg4nodes}\subref{fig:sub:deg4node1}. Assume that the
degree of $w_1, \ldots, w_4$ is at least $3$. As all but at most one node are partner for some
minimal node, and at most one node has degree $4$, we can pick one of $w_1, \ldots, w_4$ that has
degree $3$ and is a partner for some minimal node $u$. Assume this is $w_1$. We need two different
paths of length $2$ connecting $w_1$ to its partner. Hence, one of the paths must use one of the
edges $(w_1,u_1)$ or $(w_1,v)$. Thus, the partner $u$ must be one of the nodes $w_2, w_3, w_4$. This
contradicts the assumption that all four nodes have degree at least $3$. So at least one of $w_1,
w_2, w_3, w_4$ has degree $2$. Assume that this is $w_2$.  Hence, we can use
\autoref{lemma:reductionlemma} for $w_2$ and $u_1, v$ to move one of the edges incident to $v$ to
$u_1$ to obtain a graph with maximal degree $3$ in the upper layer. See
\autoref{fig:deg4nodesproof}. By our assumption that the corresponding combinatorial type of the
face has Birkhoff dimension $\bdim(\lattice)\ge 2d-3$ we know that the graph remains irreducible.

Now consider the graph given in \autoref{fig:deg4nodes}\subref{fig:sub:deg4node2}. By the same
argument as above at least one of $w_1, \ldots, w_4$ has degree $2$. If this is $w_1$ or $w_3$ we
can proceed as above and move an edge incident to $v$ to either $u_1$ or $u_2$. So assume that only
$w_4$ has degree $2$. So $w_1, w_2, w_3$ have degree at least $3$. If $w_2$ has degree $4$, then
it is partner for two minimal nodes. So at least one of $w_1, w_3$ would have degree $2$, contrary
to our assumption. So $w_2$ has degree $3$. If it were a partner for some minimal node, then this
would have to be $w_1$ or $w_3$, again contradicting our assumption. So $w_2$ is not a partner of
some minimal node. By assumption this implies that the maximal degree in the lower layer is $3$,
and both $w_1$ and $w_3$ are partner of some minimal node. By construction, this can only be
$w_4$. But, again by assumption, a minimal node has a unique partner, so this case does not occur.

We can repeat the same argument for the lower layer. This transforms $G$ into a face graph whose
combinatorial type is combinatorially isomorphic to the original one, but the graph has maximal
degree $3$ in both layers.  

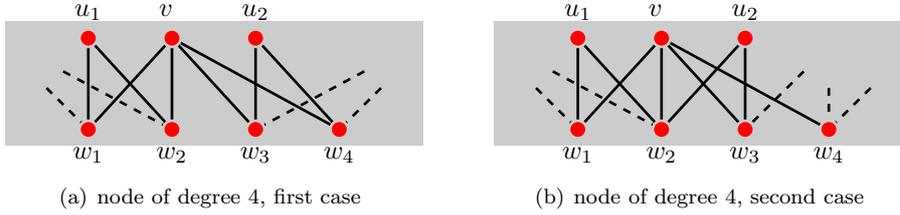
\begin{figure}[t]
  \centering
  \subfigure[node of degree $4$, first case \label{fig:sub:deg4node1}]{
    \begin{tikzpicture}[scale=1.1]
      \coordinate (g111) at (0,0);
      \coordinate (g112) at (5,0);
      \coordinate (g121) at (0,1.5);
      \coordinate (g122) at (5,1.5);
      
      \coordinate (g111b) at (0,-.3);
      \coordinate (g112b) at (5,-.3);

      \coordinate (u2) at (1,1.3);%
      \coordinate (u1) at (2,1.3);%
      \coordinate (u3) at (3,1.3);%
      
      \coordinate (v1) at (1,.2);%
      \coordinate (v2) at (2,.2);%
      \coordinate (v3) at (3,.2);%
      \coordinate (v4) at (4,.2);%
      
      \coordinate (w1) at (.5,.7);%
      \coordinate (w2) at (4.5,.7);%
      \coordinate (w3) at (4.3,.9);%
      \coordinate (w4) at (.7,.9);%

      \fill[color=black!20] (g111) -- (g112) -- (g122) -- (g121) -- cycle;
      
      \draw[line width=1pt] (u1) -- (v1);
      \draw[line width=1pt] (u1) -- (v2);
      \draw[line width=1pt] (u1) -- (v3);
      \draw[line width=1pt] (u1) -- (v4);
      
      \draw[line width=1pt] (u2) -- (v1);
      \draw[line width=1pt] (u2) -- (v2);
      
      \draw[line width=1pt] (u3) -- (v3);
      \draw[line width=1pt] (u3) -- (v4);
      
      \draw[line width=1pt,dashed] (w1) -- (v1);
      \draw[line width=1pt,dashed] (w2) -- (v4);
      \draw[line width=1pt,dashed] (w4) -- (v2);
      \draw[line width=1pt,dashed] (w3) -- (v3);

      \draw[color=black!20,fill=red,line width=1pt] (u1) circle (3pt) node[black,above,yshift=3pt] {$v\phantom{_1}$};
      \draw[color=black!20,fill=red,line width=1pt] (u2) circle (3pt) node[black,above,yshift=3pt] {$u_1$};
      \draw[color=black!20,fill=red,line width=1pt] (u3) circle (3pt) node[black,above,yshift=3pt] {$u_2$};
      \draw[color=black!20,fill=red,line width=1pt] (v1) circle (3pt) node[black,below,yshift=-3pt] {$w_1$};
      \draw[color=black!20,fill=red,line width=1pt] (v2) circle (3pt) node[black,below,yshift=-3pt] {$w_2$};
      \draw[color=black!20,fill=red,line width=1pt] (v3) circle (3pt) node[black,below,yshift=-3pt] {$w_3$};
      \draw[color=black!20,fill=red,line width=1pt] (v4) circle (3pt) node[black,below,yshift=-3pt] {$w_4$};
    \end{tikzpicture}
  }\qquad
  \subfigure[node of degree $4$, second case \label{fig:sub:deg4node2}]{
    \begin{tikzpicture}[scale=1.1]
      \coordinate (g111) at (0,0);
      \coordinate (g112) at (5,0);
      \coordinate (g121) at (0,1.5);
      \coordinate (g122) at (5,1.5);
      
      \coordinate (g111b) at (0,-.3);
      \coordinate (g112b) at (5,-.3);

      \coordinate (u2) at (1,1.3);%
      \coordinate (u1) at (2,1.3);%
      \coordinate (u3) at (3,1.3);%
      
      \coordinate (v1) at (1,.2);%
      \coordinate (v2) at (2,.2);%
      \coordinate (v3) at (3,.2);%
      \coordinate (v4) at (4,.2);%
      
      \coordinate (w1) at (.5,.7);%
      \coordinate (w2) at (4.5,.7);%
      \coordinate (w3) at (4,.7);%
      \coordinate (w4) at (.7,.9);%
      \coordinate (w5) at (3.7,.9);%

      \fill[color=black!20] (g111) -- (g112) -- (g122) -- (g121) -- cycle;
      
      \draw[line width=1pt] (u1) -- (v1);
      \draw[line width=1pt] (u1) -- (v2);
      \draw[line width=1pt] (u1) -- (v3);
      \draw[line width=1pt] (u3) -- (v2);
      
      \draw[line width=1pt] (u2) -- (v1);
      \draw[line width=1pt] (u2) -- (v2);
      
      \draw[line width=1pt] (u3) -- (v3);
      
      \draw[line width=1pt,dashed] (w1) -- (v1);
      \draw[line width=1pt] (u1) -- (v4);

      \draw[line width=1pt,dashed] (w2) -- (v4);
      \draw[line width=1pt,dashed] (w3) -- (v4);
      \draw[line width=1pt,dashed] (w4) -- (v2);
      \draw[line width=1pt,dashed] (w5) -- (v3);
      
      \draw[color=black!20,fill=red,line width=1pt] (u1) circle (3pt) node[black,above,yshift=3pt] {$v\phantom{_1}$};
      \draw[color=black!20,fill=red,line width=1pt] (u2) circle (3pt) node[black,above,yshift=3pt] {$u_1$};
      \draw[color=black!20,fill=red,line width=1pt] (u3) circle (3pt)
      node[black,above,yshift=3pt] {$u_2$};
      \draw[color=black!20,fill=red,line width=1pt] (v1) circle (3pt) node[black,below,yshift=-3pt] {$w_1$};
      \draw[color=black!20,fill=red,line width=1pt] (v2) circle (3pt) node[black,below,yshift=-3pt] {$w_2$};
      \draw[color=black!20,fill=red,line width=1pt] (v3) circle (3pt) node[black,below,yshift=-3pt] {$w_3$};
      \draw[color=black!20,fill=red,line width=1pt] (v4) circle (3pt) node[black,below,yshift=-3pt] {$w_4$};
    \end{tikzpicture}    
  }
  \caption{The two cases of a node of degree $4$ in
    \autoref{thm:ClassWedges}. The dashed partial edges indicate
    that there may or may not be additional edges incident to some of
    the $w_i$.}
  \label{fig:deg4nodes}
\end{figure}
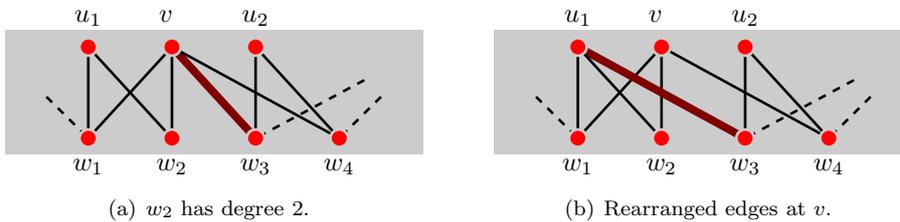
\begin{figure}[b]
  \centering \subfigure[$w_2$ has degree
  $2$. \label{fig:sub:deg4node3}]{
    \begin{tikzpicture}[scale=1.1]
      \coordinate (g111) at (0,0);
      \coordinate (g112) at (5,0);
      \coordinate (g121) at (0,1.5);
      \coordinate (g122) at (5,1.5);
      
      \coordinate (g111b) at (0,-.3);
      \coordinate (g112b) at (5,-.3);

      \coordinate (u2) at (1,1.3);%
      \coordinate (u1) at (2,1.3);%
      \coordinate (u3) at (3,1.3);%
      
      \coordinate (v1) at (1,.2);%
      \coordinate (v2) at (2,.2);%
      \coordinate (v3) at (3,.2);%
      \coordinate (v4) at (4,.2);%
      
      \coordinate (w1) at (.5,.7);%
      \coordinate (w2) at (4.5,.7);%
      \coordinate (w3) at (4.3,.9);%
      \coordinate (w4) at (.7,.9);%

      \fill[color=black!20] (g111) -- (g112) -- (g122) -- (g121) -- cycle;
      
      \draw[line width=1pt] (u1) -- (v1);
      \draw[line width=3pt,color=red!50!black] (u1) -- (v3);
      \draw[line width=1pt] (u1) -- (v2);
      \draw[line width=1pt] (u1) -- (v4);
      
      \draw[line width=1pt] (u2) -- (v1);
      \draw[line width=1pt] (u2) -- (v2);
      
      \draw[line width=1pt] (u3) -- (v3);
      \draw[line width=1pt] (u3) -- (v4);
      
      \draw[line width=1pt,dashed] (w1) -- (v1);
      \draw[line width=1pt,dashed] (w2) -- (v4);
      \draw[line width=1pt,dashed] (w3) -- (v3);

      \draw[color=black!20,fill=red,line width=1pt] (u1) circle (3pt) node[black,above,yshift=4pt] {$v\phantom{_1}$};
      \draw[color=black!20,fill=red,line width=1pt] (u2) circle (3pt) node[black,above,yshift=4pt] {$u_1$};
      \draw[color=black!20,fill=red,line width=1pt] (u3) circle (3pt) node[black,above,yshift=4pt] {$u_2$};
      \draw[color=black!20,fill=red,line width=1pt] (v1) circle (3pt) node[black,below,yshift=-4pt] {$w_1$};
      \draw[color=black!20,fill=red,line width=1pt] (v2) circle (3pt) node[black,below,yshift=-4pt] {$w_2$};
      \draw[color=black!20,fill=red,line width=1pt] (v3) circle (3pt) node[black,below,yshift=-4pt] {$w_3$};
      \draw[color=black!20,fill=red,line width=1pt] (v4) circle (3pt) node[black,below,yshift=-4pt] {$w_4$};
    \end{tikzpicture}
  }\qquad
  \subfigure[Rearranged edges at $v$. \label{fig:sub:deg4node4}]{
    \begin{tikzpicture}[scale=1.1]
      \coordinate (g111) at (0,0);
      \coordinate (g112) at (5,0);
      \coordinate (g121) at (0,1.5);
      \coordinate (g122) at (5,1.5);
      
      \coordinate (g111b) at (0,-.3);
      \coordinate (g112b) at (5,-.3);

      \coordinate (u2) at (1,1.3);%
      \coordinate (u1) at (2,1.3);%
      \coordinate (u3) at (3,1.3);%
      
      \coordinate (v1) at (1,.2);%
      \coordinate (v2) at (2,.2);%
      \coordinate (v3) at (3,.2);%
      \coordinate (v4) at (4,.2);%
      
      \coordinate (w1) at (.5,.7);%
      \coordinate (w2) at (4.5,.7);%
      \coordinate (w3) at (4.3,.9);%
      \coordinate (w4) at (.7,.9);%

      \fill[color=black!20] (g111) -- (g112) -- (g122) -- (g121) -- cycle;
      
      \draw[line width=1pt] (u1) -- (v1);
      \draw[line width=1pt] (u1) -- (v2);
      \draw[line width=3pt,color=red!50!black] (u2) -- (v3);
      \draw[line width=1pt] (u1) -- (v4);
      
      \draw[line width=1pt] (u2) -- (v1);
      \draw[line width=1pt] (u2) -- (v2);
      
      \draw[line width=1pt] (u3) -- (v3);
      \draw[line width=1pt] (u3) -- (v4);
      
      \draw[line width=1pt,dashed] (w1) -- (v1);
      \draw[line width=1pt,dashed] (w2) -- (v4);
      \draw[line width=1pt,dashed] (w3) -- (v3);

      \draw[color=black!20,fill=red,line width=1pt] (u1) circle (3pt) node[black,above,yshift=4pt] {$v\phantom{_1}$};
      \draw[color=black!20,fill=red,line width=1pt] (u2) circle (3pt) node[black,above,yshift=4pt] {$u_1$};
      \draw[color=black!20,fill=red,line width=1pt] (u3) circle (3pt) node[black,above,yshift=4pt] {$u_2$};
      \draw[color=black!20,fill=red,line width=1pt] (v1) circle (3pt) node[black,below,yshift=-4pt] {$w_1$};
      \draw[color=black!20,fill=red,line width=1pt] (v2) circle (3pt) node[black,below,yshift=-4pt] {$w_2$};
      \draw[color=black!20,fill=red,line width=1pt] (v3) circle (3pt) node[black,below,yshift=-4pt] {$w_3$};
      \draw[color=black!20,fill=red,line width=1pt] (v4) circle (3pt) node[black,below,yshift=-4pt] {$w_4$};
    \end{tikzpicture}
  }
  \caption{Moving one edge incident to $v$ in the proof of
    \autoref{thm:ClassWedges}. The dashed partial edges indicate
    that there may or may not be additional edges incident to some of
    the $w_i$.}
  \label{fig:deg4nodesproof}
\end{figure}
Hence, in the following we can assume that the maximal degree of a node in $G$ is $3$. In that case
the graph necessarily has $d-2$ minimal nodes in each layer, and $d-1$ nodes of degree $3$. Pick a
partner $p^u_i$ for each minimal node $x_i^u$ in the upper, and $p^l_i$ for each minimal node
$x^l_i$ in the lower layer, $1\le i\le d-2$. The $p^u_i$ are pairwise distinct as a node of degree
$k\ge 3$ is partner for at most $k-2$ nodes (unless it is the only partner in the graph, see
\autoref{cor:partnerbound}). See also \autoref{fig:decomp} for two examples. Let $y_u$ and $y_l$ be
the remaining node in each layer. Let $N_i$ be the induced graph on $p^u_i, x^u_i$ and the two paths
of length $2$ between them. The $N_i$, $1\le i\le d-2$, are pairwise disjoint, as otherwise there
would be a node of degree $4$ in the lower layer. Let $z_l$ be the node in the lower layer not
contained in any $N_i$. We distinguish various cases.
  \begin{enumerate}
  \item $y_u$ and $z_l$ are connected, and $z_l$ is minimal. See~\autoref{fig:sub:case1}. In this
    case $z_l$ has a partner $p_l$ contained in some $N_i$. We may assume that this is $N_1$. We
    replace $N_1$ by the graph induced by $N_1$, $y_u$ and $z_l$. Then $N_1$ has $6$ nodes and at
    least $7$ edges. Hence, $N:=\bigcup N_i$ defines a face subgraph of $G$ with the same number of
    nodes, with $d-2$ components, and at most $d-2$ edges less than $G$. Thus, $N$ and $G$ differ by
    precisely $d-2$ edges. As any edge must be contained in a perfect matching, those edges must
    connect the components of $H$ in a circular way. Further, $y_u$ and $y_l$ have degree $2$ in
    $N_1$, but degree $3$ in $G$, so $N_1-\{y_u, y_l\}$ has a unique perfect matching. Hence, $G$ is
    a pyramid over $N$, and $N$ is a product of segments and a triangle.
  \item $y_u$ and $z_l$ are connected, and both have degree $3$.  See~\autoref{fig:sub:case2}. In
    that case, $z_l=y_l$. Let $N_0$ be the subgraph induced by $y_u$ and $y_l$ (with one edge). Then
    $N:=\bigcup N_i$ is a face subgraph of $G$ with $4d-7$ edges. The only way to obtain a connected
    irreducible face graph from $N$ by adding $d$ edges is to split the set of $N_i$, $i\ge 1$ into
    two nontrivial sets and connect both with $N_0$ circularly. This gives a reduced joined circular
    product of two cubes (not necessarily of equal dimension)
\begin{figure}[t]
\centering
  \subfigure[The graph of case (1) in the proof of \autoref{thm:ClassWedges}\label{fig:sub:case1}]{
    \begin{tikzpicture}[scale=.75]
      \foreach \x in {-2,...,8} {%
        \coordinate (u\x) at (1.5*\x,1.8);%
        \coordinate (v\x) at (1.5*\x,0);%
      }%

      \fill[color=black!20] (-.3,2.1) -- (-.3,-.3) -- (1.8,-.3) -- (1.8,2.1) -- cycle;
      \fill[color=black!20] (4.2,2.1) -- (4.2,-.3) -- (6.3,-.3) -- (6.3,2.1) -- cycle;
      \fill[color=black!20] (7.2,2.1) -- (7.2,-.3) -- (9.3,-.3) -- (9.3,2.1) -- cycle;

      \foreach \x in {0,...,6} {%
        \draw[line width=1pt] (u\x) -- (v\x);
      }

      \draw[line width=1pt] (u0) -- (v1);
      \draw[line width=1pt] (u1) -- (v0);
      \draw[line width=1pt] (u1) -- (v2);
      \draw[line width=1pt] (u2) -- (v1);

      \draw[line width=1pt] (u3) -- (v4);
      \draw[line width=1pt] (u4) -- (v3);

      \draw[line width=1pt] (u5) -- (v6);
      \draw[line width=1pt] (u6) -- (v5);

      \draw[line width=1pt] (u2) -- (v3);
      \draw[line width=1pt] (u4) -- (v5);
      \draw[line width=1pt] (u6) -- (v0);

      \foreach \x in {0,...,6} {%
        \draw[color=black!20,fill=red,line width=2pt] (u\x) circle (5pt);
        \draw[color=black!20,fill=red,line width=2pt] (v\x) circle (5pt);
      }%
      \draw[color=black!20,fill=red,line width=2pt] (u0) circle (5pt) node[black,above,yshift=4pt] {$x_1^u$};
      \draw[color=black!20,fill=red,line width=2pt] (u1) circle (5pt) node[black,above,yshift=4pt] {$p_1^u$};
      \draw[color=white,fill=red,line width=2pt] (u2) circle (5pt) node[black,above,yshift=4pt] {$y_u$};
      \draw[color=black!20,fill=red,line width=2pt] (u3) circle (5pt) node[black,above,yshift=4pt] {$x_2^u$};
      \draw[color=black!20,fill=red,line width=2pt] (u4) circle (5pt) node[black,above,yshift=4pt] {$p_2^u$};
      \draw[color=black!20,fill=red,line width=2pt] (u5) circle (5pt) node[black,above,yshift=4pt] {$x_3^u$};
      \draw[color=black!20,fill=red,line width=2pt] (u6) circle (5pt) node[black,above,yshift=4pt] {$p_3^u$};

      \draw[color=black!20,fill=red,line width=2pt] (v0) circle (5pt) node[black,below,yshift=-4pt] {$y_l\phantom{^l}$};
      \draw[color=black!20,fill=red,line width=2pt] (v1) circle (5pt) node[black,below,yshift=-4pt] {$p_1^l$};
      \draw[color=white,fill=red,line width=2pt] (v2) circle (5pt) node[black,below,yshift=-4pt] {$x_1^l=z_l$};
      \draw[color=black!20,fill=red,line width=2pt] (v3) circle (5pt) node[black,below,yshift=-4pt] {$p_2^l$};
      \draw[color=black!20,fill=red,line width=2pt] (v4) circle (5pt) node[black,below,yshift=-4pt] {$x_2^l$};
      \draw[color=black!20,fill=red,line width=2pt] (v5) circle (5pt) node[black,below,yshift=-4pt] {$p_3^l$};
      \draw[color=black!20,fill=red,line width=2pt] (v6) circle (5pt) node[black,below,yshift=-4pt] {$x_3^l$};

      \node at (.75,3.4) {$N_1$};
      \node at (5.25,3.4) {$N_2$};
      \node at (8.25,3.4) {$N_3$};

      \node at (u-2) {};
      \node at (u8) {};
    \end{tikzpicture}
  }

  \subfigure[The graph of case (2) in the proof of \autoref{thm:ClassWedges}\label{fig:sub:case2}]{
    \begin{tikzpicture}[scale=.75]
      \foreach \x in {0,...,8} {%
        \coordinate (u\x) at (1.5*\x,1.8);%
        \coordinate (v\x) at (1.5*\x,0);%
      }%

      \fill[color=black!20] (-.3,2.1) -- (-.3,-.3) -- (1.8,-.3) -- (1.8,2.1) -- cycle;

      \fill[color=black!20] (2.7,2.1) -- (2.7,-.3) -- (4.8,-.3) -- (4.8,2.1) -- cycle;

      \fill[color=black!20] (7.2,2.1) -- (7.2,-.3) -- (9.3,-.3) -- (9.3,2.1) -- cycle;

      \fill[color=black!20] (10.2,2.1) -- (10.2,-.3) -- (12.3,-.3) -- (12.3,2.1) -- cycle;

      \foreach \x in {0,...,8} {%
        \draw[line width=1pt] (u\x) -- (v\x);
      }

      \draw[line width=1pt] (u0) -- (v1);
      \draw[line width=1pt] (u1) -- (v0);

      \draw[line width=1pt] (u2) -- (v3);
      \draw[line width=1pt] (u3) -- (v2);

      \draw[line width=1pt] (u5) -- (v6);
      \draw[line width=1pt] (u6) -- (v5);

      \draw[line width=1pt] (u7) -- (v8);
      \draw[line width=1pt] (u8) -- (v7);

      \draw[line width=1pt] (u1) -- (v2);
      \draw[line width=1pt] (u3) -- (v4);
      \draw[line width=1pt] (u4) -- (v5);
      \draw[line width=1pt] (u6) -- (v7);
      \draw[line width=1pt] (u8) -- (v4);
      \draw[line width=1pt] (u4) -- (v0);

      \foreach \x in {0,1,2,3,5,6,7,8} {%
        \draw[color=black!20,fill=red,line width=2pt] (u\x) circle (5pt);
        \draw[color=black!20,fill=red,line width=2pt] (v\x) circle (5pt);
      }%
      \draw[color=black!20,fill=red,line width=2pt] (u0) circle (5pt) node[black,above,yshift=4pt] {$x_1^u$};
      \draw[color=black!20,fill=red,line width=2pt] (u1) circle (5pt) node[black,above,yshift=4pt] {$p_1^u$};
      \draw[color=black!20,fill=red,line width=2pt] (u2) circle (5pt) node[black,above,yshift=4pt] {$x_2^u$};
      \draw[color=black!20,fill=red,line width=2pt] (u3) circle (5pt) node[black,above,yshift=4pt] {$p_2^u$};
      \draw[color=white,fill=red,line width=2pt] (u4) circle (5pt) node[black,above,yshift=4pt] {$y_u$};
      \draw[color=black!20,fill=red,line width=2pt] (u5) circle (5pt) node[black,above,yshift=4pt] {$x_3^u$};
      \draw[color=black!20,fill=red,line width=2pt] (u6) circle (5pt) node[black,above,yshift=4pt] {$p_3^u$};
      \draw[color=black!20,fill=red,line width=2pt] (u7) circle (5pt) node[black,above,yshift=4pt] {$x_4^u$};
      \draw[color=black!20,fill=red,line width=2pt] (u8) circle (5pt) node[black,above,yshift=4pt] {$p_4^u$};

      \draw[color=black!20,fill=red,line width=2pt] (v0) circle (5pt) node[black,below,yshift=-4pt] {$p_1^l$};
      \draw[color=black!20,fill=red,line width=2pt] (v1) circle (5pt) node[black,below,yshift=-4pt] {$x_1^l$};
      \draw[color=black!20,fill=red,line width=2pt] (v2) circle (5pt) node[black,below,yshift=-4pt] {$p_2^l$};
      \draw[color=black!20,fill=red,line width=2pt] (v3) circle (5pt) node[black,below,yshift=-4pt] {$x_2^l$};
      \draw[color=white,fill=red,line width=2pt] (v4) circle (5pt) node[black,below,yshift=-4pt] {$z_l=y_l\phantom{^l}$};
      \draw[color=black!20,fill=red,line width=2pt] (v5) circle (5pt) node[black,below,yshift=-4pt] {$p_3^l$};
      \draw[color=black!20,fill=red,line width=2pt] (v6) circle (5pt) node[black,below,yshift=-4pt] {$x_3^l$};
      \draw[color=black!20,fill=red,line width=2pt] (v7) circle (5pt) node[black,below,yshift=-4pt] {$p_4^l$};
      \draw[color=black!20,fill=red,line width=2pt] (v8) circle (5pt) node[black,below,yshift=-4pt] {$x_4^l$};

      \node at (.75,3.4) {$N_1$};
      \node at (3.75,3.4) {$N_2$};
      \node at (8.25,3.4) {$N_3$};
      \node at (11.25,3.4) {$N_4$};
    \end{tikzpicture}
  } 
  \caption{Decomposing graphs in the proof of \autoref{thm:ClassWedges}}
  \label{fig:decomp}
\end{figure}
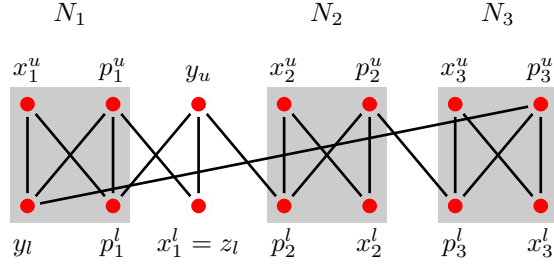
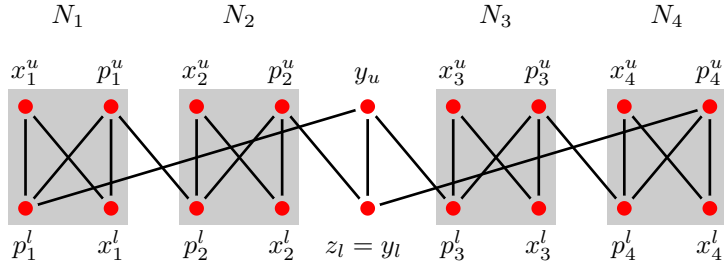
  \item $y_u$ and $z_l$ are not connected.  Assume that the degree of $z_l$ is $2$. Then $z_l\ne
    y_l$, and it needs a partner in the lower layer. The two edges incident to $z_l$ cannot both end
    in the same $N_i$, as one node in the upper layer of each $N_i$ has degree $2$.  Hence, the two
    incident edges end in different $N_1$, say at nodes $s_1\in N_1$ and $s_2\in N_2$. Yet, $z_l$
    needs a partner, so there is either an edge from $s_1$ to a node of $N_2$ or from $s_2$ or a
    node of $N_1$. Hence, either $s_1$ or $s_2$ have degree $4$. By construction, such a node does
    not exist, so we can assume that $z_l$ has degree $3$.  Again, the edges incident to $z_l$
    necessarily end in different $N_i$, as one node in the upper layer of each $N_i$ has degree
    $2$. Hence, $z_l$ cannot be a partner, so $z_l=y_l$. 

    The graph $N:=\bigcup N_i\cup\{y_u, y_l\}$ has $4d-8$ edges. $G$ has $d+1$ additional edges, and
    as $G$ is a connected face graph each $N_1$ is incident to at least two of them (as, in
    particular, each edge must be contained in a perfect matching). But as $y_u, y_l$ have degree
    $3$ we conclude that each $N_i$ is incident to exactly two of the $d+1$ additional edges. Hence,
    as before, the only way to create a connected face graph by adding only $d+1$ edges is to split
    the set of the $N_i$ into three nontrivial parts and connect them to $y_u, y_l$
    circularly. \qedhere
  \end{enumerate}
\end{proof}

\section{Low-dimensional Classification}
\label{sec:append-low-dimens}
\stepcounter{section}
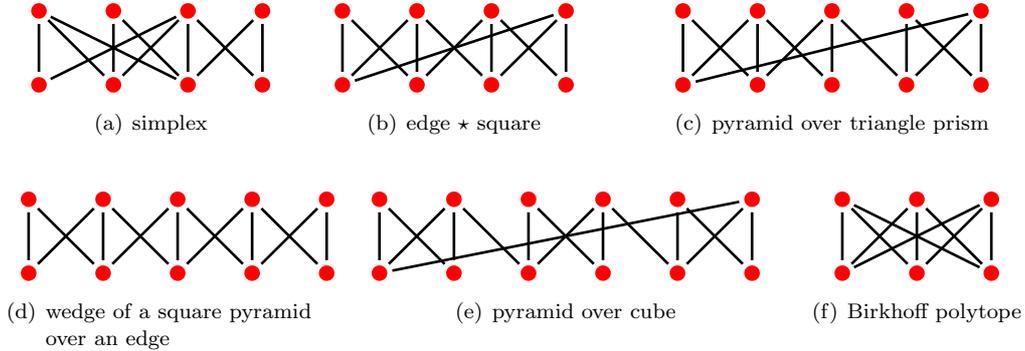
\begin{figure}[tb]
\centering
  \subfigure[simplex]{
    \begin{tikzpicture}[scale=.98]
      \foreach \x in {0,1,2,3} {
        \coordinate (u\x) at (\x,1);
        \coordinate (v\x) at (\x,0);
      }
           
      \draw [line width=1pt] (u0) -- (v0);
      \draw [line width=1pt] (u0) -- (v1);
      \draw [line width=1pt] (u0) -- (v2);

      \draw [line width=1pt] (u1) -- (v1);
      \draw [line width=1pt] (u1) -- (v2);

      \draw [line width=1pt] (u2) -- (v0);
      \draw [line width=1pt] (u2) -- (v1);
      \draw [line width=1pt] (u2) -- (v2);
      \draw [line width=1pt] (u2) -- (v3);

      \draw [line width=1pt] (u3) -- (v2);
      \draw [line width=1pt] (u3) -- (v3);

      \foreach \x in {0,1,2,3} {
        \draw [color=white,fill=red,line width=2pt] (u\x) circle (4pt);
        \draw [color=white,fill=red,line width=2pt] (v\x) circle (4pt);
      }
    \end{tikzpicture}
  }
  \quad
  \subfigure[edge $\star$ square]{
    \begin{tikzpicture}[scale=.98]
      \foreach \x in {0,1,2,3} {
        \coordinate (u\x) at (\x,1);
        \coordinate (v\x) at (\x,0);
      }
           
      \draw [line width=1pt] (u0) -- (v0);
      \draw [line width=1pt] (u0) -- (v1);

      \draw [line width=1pt] (u1) -- (v0);
      \draw [line width=1pt] (u1) -- (v1);
      \draw [line width=1pt] (u1) -- (v2);

      \draw [line width=1pt] (u2) -- (v1);
      \draw [line width=1pt] (u2) -- (v2);
      \draw [line width=1pt] (u2) -- (v3);

      \draw [line width=1pt] (u3) -- (v0);
      \draw [line width=1pt] (u3) -- (v2);
      \draw [line width=1pt] (u3) -- (v3);

      \foreach \x in {0,1,2,3} {
        \draw [color=white,fill=red,line width=2pt] (u\x) circle (4pt);
        \draw [color=white,fill=red,line width=2pt] (v\x) circle (4pt);
      }
    \end{tikzpicture}
  }
\quad
  \subfigure[pyramid over triangle prism]{
    \begin{tikzpicture}[scale=.98]
      \foreach \x in {0,1,2,3,4} {
        \coordinate (u\x) at (\x,1);
        \coordinate (v\x) at (\x,0);
      }
           
      \draw [line width=1pt] (u0) -- (v0);
      \draw [line width=1pt] (u0) -- (v1);

      \draw [line width=1pt] (u1) -- (v0);
      \draw [line width=1pt] (u1) -- (v1);
      \draw [line width=1pt] (u1) -- (v2);

      \draw [line width=1pt] (u2) -- (v1);
      \draw [line width=1pt] (u2) -- (v2);
      \draw [line width=1pt] (u2) -- (v3);

      \draw [line width=1pt] (u3) -- (v3);
      \draw [line width=1pt] (u3) -- (v4);

      \draw [line width=1pt] (u4) -- (v0);
      \draw [line width=1pt] (u4) -- (v3);
      \draw [line width=1pt] (u4) -- (v4);

      \foreach \x in {0,1,2,3,4} {
        \draw [color=white,fill=red,line width=2pt] (u\x) circle (4pt);
        \draw [color=white,fill=red,line width=2pt] (v\x) circle (4pt);
      }
      \draw [color=white,fill=white,line width=2pt] (-.5,0) circle
      (4pt);
      \draw [color=white,fill=white,line width=2pt] (4.5,0) circle (4pt);
    \end{tikzpicture}
  }

  \bigskip

  \subfigure[wedge of a square pyramid over an edge]{
    \begin{tikzpicture}[scale=.98]
      \foreach \x in {0,1,2,3,4} {
        \coordinate (u\x) at (\x,1);
        \coordinate (v\x) at (\x,0);
      }
           
      \draw [line width=1pt] (u0) -- (v0);
      \draw [line width=1pt] (u0) -- (v1);

      \draw [line width=1pt] (u1) -- (v0);
      \draw [line width=1pt] (u1) -- (v1);
      \draw [line width=1pt] (u1) -- (v2);

      \draw [line width=1pt] (u2) -- (v1);
      \draw [line width=1pt] (u2) -- (v2);
      \draw [line width=1pt] (u2) -- (v3);

      \draw [line width=1pt] (u3) -- (v2);
      \draw [line width=1pt] (u3) -- (v3);
      \draw [line width=1pt] (u3) -- (v4);

      \draw [line width=1pt] (u4) -- (v3);
      \draw [line width=1pt] (u4) -- (v4);

      \foreach \x in {0,1,2,3,4} {
        \draw [color=white,fill=red,line width=2pt] (u\x) circle (4pt);
        \draw [color=white,fill=red,line width=2pt] (v\x) circle (4pt);
      }
    \end{tikzpicture}
  }
  \subfigure[pyramid over cube]{
    \begin{tikzpicture}[scale=.98]
      \foreach \x in {0,1,2,3,4,5} {
        \coordinate (u\x) at (\x,1);
        \coordinate (v\x) at (\x,0);
      }
           
      \draw [line width=1pt] (u0) -- (v0);
      \draw [line width=1pt] (u0) -- (v1);

      \draw [line width=1pt] (u1) -- (v0);
      \draw [line width=1pt] (u1) -- (v1);
      \draw [line width=1pt] (u1) -- (v2);

      \draw [line width=1pt] (u2) -- (v2);
      \draw [line width=1pt] (u2) -- (v3);

      \draw [line width=1pt] (u3) -- (v2);
      \draw [line width=1pt] (u3) -- (v3);
      \draw [line width=1pt] (u3) -- (v4);

      \draw [line width=1pt] (u4) -- (v4);
      \draw [line width=1pt] (u4) -- (v5);

      \draw [line width=1pt] (u5) -- (v0);
      \draw [line width=1pt] (u5) -- (v4);
      \draw [line width=1pt] (u5) -- (v5);

      \foreach \x in {0,1,2,3,4,5} {
        \draw [color=white,fill=red,line width=2pt] (u\x) circle (4pt);
        \draw [color=white,fill=red,line width=2pt] (v\x) circle (4pt);
      }
    \end{tikzpicture}
  }
  \subfigure[Birkhoff polytope]{
    \begin{tikzpicture}[scale=.98]
      \foreach \x in {0,1,2} {
        \coordinate (u\x) at (\x,1);
        \coordinate (v\x) at (\x,0);
      }
           
      \draw [line width=1pt] (u0) -- (v0);
      \draw [line width=1pt] (u0) -- (v1);
      \draw [line width=1pt] (u0) -- (v2);

      \draw [line width=1pt] (u1) -- (v0);
      \draw [line width=1pt] (u1) -- (v1);
      \draw [line width=1pt] (u1) -- (v2);

      \draw [line width=1pt] (u2) -- (v0);
      \draw [line width=1pt] (u2) -- (v1);
      \draw [line width=1pt] (u2) -- (v2);

      \foreach \x in {0,1,2} {
        \draw [color=white,fill=red,line width=2pt] (u\x) circle (4pt);
        \draw [color=white,fill=red,line width=2pt] (v\x) circle (4pt);
      }

      \draw [color=white,fill=white,line width=2pt] (-.5,0) circle
      (4pt);
      \draw [color=white,fill=white,line width=2pt] (2.5,0) circle (4pt);
    \end{tikzpicture}
  }
\caption{The $4$-dimensional Birkhoff faces which are not products}
\label{fig:Graphs4DNonProd}
\end{figure}
\addtocounter{section}{-1}

For a classification of faces of a given dimension $d$ it is essentially sufficient to classify
those faces that have a connected face graph. The others are products of lower dimensional faces of
the Birkhoff polytope, and they can thus be obtained as pairs of face graphs of a lower dimension.

The three dimensional faces have been classified before by several others, see, \emph{e.g.},
\cite{BG77-1} or \cite{BS96}.  By \autoref{thm:billera} we know that any $d$-dimensional face
appearing in some Birkhoff polytope does so in a Birkhoff polytope $\birkhoff$ for $n\le 2d$.

We have implemented an algorithm that generates all \emph{irreducible connected} face graphs of a
given dimension and with a given number of nodes. The implementation is done as an
extension~\cite{polymake_birkhoff} to \texttt{polymake}~\cite{GJ00}. The algorithm provides a method
\texttt{generate\_face\_graphs} that takes two arguments, the number of nodes of the graph in one of
the layers, and the dimension of the face. It constructs all irreducible face graphs with this
number of nodes and the given dimension, up to combinatorial isomorphism of the corresponding face
(as some combinatorial types have irreducible face graphs with different number of nodes they can
appear in several times different runs of the method). Dimension and number of nodes fixes the
number of edges, and, roughly, the method recursively adds edges to an empty graph until it reaches
the appropriate number.

It distinguishes between graphs of minimal degree $3$ and those with at least one node of degree
$2$. Constructing those with minimal degree $3$ is simple, as filling each node with at least
three edges does not leave much choice for a face graph. This can be done by a simple recursion
using some of the results in~\autoref{sec:irred} and~\autoref{sec:combtypes}.

For the other graphs we iterate over the number of nodes of degree $2$, and first equip each such
node with a partner and the necessary edges, and add further edges until all remaining nodes have
degree $3$. Here we use the results of~\autoref{sec:irred} and~\autoref{sec:combtypes} to prune
the search tree at an early stage if graphs in this branch either will not be irreducible or not a
face graph. The few remaining edges are then again filled in recursively. See also the comments in
the code.

\begin{table}[t]
  \centering
  \begin{tabular}{r|rrrrrrrr}
    dim & 1& 2& 3& 4& 5& 6& 7& 8\\
    \midrule
    \# non-product types &1 & 1& 2& 6& 20& 86& 498& 3712\\
    \# product types &0 & 1& 3& 5& 13& 43& 163& 818\\
    \midrule
    \# pyramids& 1 & 2 & 2& 4 & 10 & 28 & 98 & 416
  \end{tabular}
  
  \bigskip

  \caption{Low dimensional faces of Birkhoff polytopes. The last row of the table collects the number of pyramids among the non-product types.}
  \label{tab:lowdim}
  \vspace{-.4cm}
\end{table}
The data in \texttt{polymake} format can be obtained from~\cite{BirkhoffFaces}.  The following
theorem summarizes the results. For the product types we have just counted the non-isomorphic
products of connected irreducible graphs.
\begin{theorem}
  \begin{enumerate}
  \item In dimension $2$ there are two combinatorial types of faces, a square and a triangle.
  \item In dimension $3$ there are two combinatorial types that are products of lower dimensional
    faces, and two other types, the $3$-simplex and the pyramid over a square.
  \item In dimension $4$ there are five combinatorial types that are products, and six other types:
    \begin{enumerate*}
    \item a simplex,
    \item the join of a segment and a square,
    \item a wedge $W_1$ over an edge of the base of a square pyramid,
    \item the Birkhoff polytope $\birkhoff[3]$, and
    \item the pyramids over a cube and a triangle prism.
    \end{enumerate*}
 See \autoref{fig:Graphs4DNonProd} for examples of face graphs for the
    non-product types.
  \item In dimension $5$ there are 13 combinatorial types that are product, and 20 other types:
    \begin{enumerate*}
    \item the pyramids over all $4$-dimensional types except $\birkhoff[3]$,
    \item the join of two squares, the wedges over a facet and an edge of $\birkhoff[3]$,
    \item the wedge over the complement of the square pyramid in $W_1$,
    \item the wedge over a $3$-simplex in $W_1$,
    \item the wedge over the complement of a $3$-simplex in $W_1$ and its dual,
    \item the wedge over the complement of a triangle prism in $W_1$,
    \item the wedge over a triangle of the prism over a triangle in the pyramid over a triangle,
    \item the wedge over the edge of a square in the double pyramid over the square, and
    \item the join of two squares.
    \end{enumerate*}
  \item In dimension $6$ there are 43 combinatorial types that are product, and 86 other types.
  \item In dimension $7$ there are 163 combinatorial types that are product, and 498 other types.
  \item In dimension $8$ there are 818 combinatorial types that are
    product, and 3712 other types.
  \end{enumerate}
\end{theorem}
The descriptions given in the theorem are not unique. \autoref{tab:lowdim} summarizes this theorem.

\small

\providecommand{\bysame}{\leavevmode\hbox to3em{\hrulefill}\thinspace}
\providecommand{\MR}{\relax\ifhmode\unskip\space\fi MR }
\newcommand{\doiref}[1]{\textsc{doi:\,}\href{http://dx.doi.org/#1}{\nolinkurl{#1}}}
\newcommand{\urlref}[1]{\textsc{url:\,}\href{http://#1}{\nolinkurl{#1}}}
\newcommand{\arxivref}[2]{\textsc{arxiv:\,}\href{http://arxiv.org/#1}{\nolinkurl{#1 [#2]}}}
\renewcommand{\MR}[1]{}

\end{document}